\pdfoutput=1
\documentclass[english]{shinybook}
\newcommand{\half}{\frac{1}{2}}
\newcommand{\N}{\mathbb{N}}

\newcommand{\R}{\mathbb{R}}
\newcommand{\1}{\mathbb{1}}
\newcommand{\Rnn}{\mathbb{R}^{n\times n}}
\newcommand{\calI}{\mathcal{I}}

\newcommand{\calN}{\mathcal{N}}
\newcommand{\calP}{\mathcal{P}}

\newcommand{\calT}{\mathcal{T}}

\newcommand{\eps}{\varepsilon}
\renewcommand{\phi}{\varphi}
\renewcommand{\bar}{\overline}
\newcommand{\ran}{\operatorname{\mathrm{ran}}}
\renewcommand{\div}{\operatorname{\mathrm{div}}}
\newtheorem{theorem}{Theorem}[chapter]
\newtheorem{lemma}[theorem]{Lemma}
\newtheorem{cor}[theorem]{Corollary}

\theoremstyle{definition}
\newtheorem{defn}[theorem]{Definition}

\usepackage{functan}
\Macro{Lp}{{L}^{p}(\Omega)}
\Macro{Lq}{{L}^{q}(\Omega)}
\Macro{Linf}{{L}^{\infty}(\Omega)}
\Macro{L2}{{L}^{2}(\Omega)}
\Macro{L1}{{L}^{1}(\Omega)}
\Macro{L1loc}{{L}^{1}_{\mathrm{loc}}(\Omega)}
\Macro{l2}{{L}^{2}}
\Macro{linf}{{L}^{\infty}}
\Macro{H1}{{H}^{1}(\Omega)}
\Macro{H2}{{H}^{2}(\Omega)}
\Macro{Hm1}{{{H}^{-1}(\Omega)}}
\Macro{H10}{{{H}^{1}_0}(\Omega)}
\Macro{h1}{{H}^{1}}
\Macro{Hdiv}{{H}^{\mathrm{div}}(\Omega)}
\Macro{hdiv}{{H}^{\mathrm{div}}}
\Macro{c}{{{C}^0}}
\Macro{Ck}{{{C}^k(\Omega)}}
\Macro{Ckb}{{{C}^k(\bar\Omega)}}
\Macro{Ck0}{{{C}_0^k(\bar\Omega)}}
\Macro{Cinf}{{{C}^\infty(\bar\Omega)}}
\Macro{Cinf0}{{{C}_0^\infty(\bar\Omega)}}
\Macro{Wkp}{{{W}^{k,p}(\Omega)}}
\Macro{Wkp0}{{{W}^{k,p}_0(\Omega)}}
\Macro{W1p}{{{W}^{1,p}(\Omega)}}
\Macro{W1p0}{{{W}^{1,p}_0(\Omega)}}

\newcommand{\evolution}[2]{{#1}(0,T;{#2})}
\Macro{CkT}{\evolution{C^k}}
\Macro{CT}{\evolution{C}}
\Macro{LpT}{\evolution{L^p}}
\Macro{LqT}{\evolution{L^q}}
\Macro{L2T}{\evolution{L^2}}

\newcommand{\inner}[1]{\left(#1\right)}
\newcommand{\dual}[4][auto]{\deltwoarg[#1]{#3}{#4}{\langle}{\rangle}{,}_{{#2}^*,{#2}}}
\usepackage{stmaryrd}
\newcommand{\jump}[1]{\left\llbracket#1\right\rrbracket}
\newcommand{\avg}[1]{\{\!\{#1\}\!\}}
\newcommand{\setof}[2]{\left\{#1:#2\right\}}

\DeclareFontFamily{U}{matha}{\hyphenchar\font45}
\DeclareFontShape{U}{matha}{m}{n}{
    <5> <6> <7> <8> <9> <10> gen * matha
    <10.95> matha10 <12> <14.4> <17.28> <20.74> <24.88> matha12
}{}
\DeclareSymbolFont{matha}{U}{matha}{m}{n}
\DeclareFontSubstitution{U}{matha}{m}{n}
\DeclareFontFamily{U}{mathx}{\hyphenchar\font45}
\DeclareFontShape{U}{mathx}{m}{n}{
    <5> <6> <7> <8> <9> <10>
    <10.95> <12> <14.4> <17.28> <20.74> <24.88>
    mathx10
}{}
\DeclareSymbolFont{mathx}{U}{mathx}{m}{n}
\DeclareFontSubstitution{U}{mathx}{m}{n}
\DeclareMathDelimiter{\vvvert}     {0}{matha}{"7E}{mathx}{"17}
\newnorm{dg}{\delonearg[#1]{#2}{\vvvert}{\vvvert}}

\usepackage{booktabs}
\usepackage{tikz}
\usetikzlibrary{calc}
\usetikzlibrary{shapes.arrows}
\usepackage{enumerate}
\usepackage{sidecap}

\usepackage{algpseudocode}

\title{Introduction to\\ Finite Element Methods}
\subtitle{Lecture Notes}
\author{Christian Clason}
\publishers{\small\email{c.clason@uni-graz.at}\\\url{https://homepage.uni-graz.at/c.clason}}
\date{%
    \today
    \\
    {\small\sffamily\textsc{arxiv}:\,\href{https://arxiv.org/abs/1709.08618}{\nolinkurl{1709.08618v2}}}
}

\begin{document}
\maketitle

\frontmatter

\tableofcontents

\mainmatter

\chapter*{Preface}

Partial differential equations appear in many mathematical models of physical, biological and economic phenomena, such as elasticity, electromagnetics, fluid dynamics, quantum mechanics, pattern formation or derivative valuation. However, closed-form or analytic solutions of these equations are only available in very specific cases (e.g., for simple geometries or constant coefficients), and so one has to resort to numerical approximations of these solutions.

In these notes, we will consider \emph{finite element methods}, which have developed into one of the most flexible and powerful frameworks for the numerical (approximate) solution of partial differential equations. They were first proposed by Richard Courant in \cite{Courant:1943}; but the method did not catch on until engineers started applying similar ideas in the early 1950s. Their mathematical analysis began later, with the works of Milo{\v{s}} Zl{\'a}mal, starting with \cite{Zlamal:1968}.

Knowledge of real analysis (in particular, Lebesgue integration theory) and functional analysis (especially Hilbert space theory) as well as some familiarity of the weak theory of partial differential equations is assumed, although the fundamental results of the latter (Sobolev spaces and the variational formulation of elliptic equations) are recalled in \cref{chap:weak_elliptic}.

\enlargethispage{1cm}
These notes are mostly based on the following works:
{\small
\begin{enumerate}[{[1]}]
    \item \fullcite{Braess:2007}
    \item \fullcite{Brezzi:2013}
    \item \fullcite{Brenner:2008}
    \item \fullcite{Ern:2004}
    \item \fullcite{Rannacher}
    \item \fullcite{Thomee:2006}
\end{enumerate}
}

\part{Background}

\chapter{Overview of the finite element method}\label{chap:overview}

We begin with a \enquote{bird's-eye view} of the finite element method by considering a simple one-dimensional example. Since the goal here is to give the flavor of the results and techniques used in the construction and analysis of finite element methods, not all arguments will be completely rigorous (especially those involving derivatives and function spaces). These gaps will be filled by the more general theory in the following chapters.

\section{Variational form of elliptic PDEs}

Consider for a given function $f:(0,1)\to\R$ the solution $u:(0,1)\to \R$ of the two-point boundary value problem
\begin{equation}\label{eq:bvp1d:strong}
    \left\{ \begin{aligned}
            -u''(x) &= f(x) \quad \text{for }x\in(0,1),\\
            u(0) &= 0, \qquad u'(1) = 0.
    \end{aligned} \right.\tag{BVP}
\end{equation}
The idea is to pass from \eqref{eq:bvp1d:strong} to a system of linear equations -- which can be solved on a computer -- by projection onto a finite-dimensional subspace. Any projection requires some kind of inner product, which we introduce now. We begin by multiplying the differential equation with any sufficiently regular \emph{test function} $v$ with $v(0)=0$, integrating over $x\in(0,1)$, and integrating by parts. Then any solution $u$ of \eqref{eq:bvp1d:strong} satisfies
\begin{equation*}
    \begin{aligned}[t]
        \inner{f,v} &:= \int_0^1 f(x)v(x)\,dx = -\int_0^1 u''(x)v(x)\,dx\\
        &= \int_0^1 u'(x)v'(x)\,dx\\
        &=: a(u,v),
    \end{aligned}
\end{equation*}
where we have used that $u'(1) = 0$ and $v(0)=0$. Let us (formally for now) define the space
\begin{equation*}
    V := \setof{v:(0,1)\to\R \text{ integrable}}{a(v,v) < \infty,\, v(0)=0}.
\end{equation*}
Then we can pose the following problem: Find $u\in V$ such that
\begin{equation}\label{eq:bvp1d:weak}
    a(u,v) = (f,v) \qquad \text{ for all } v\in V\tag{W}
\end{equation}
holds. This is called the \emph{weak} or \emph{variational} form of \eqref{eq:bvp1d:strong} (since $v$ varies over all $V$). If the solution $u$ of \eqref{eq:bvp1d:weak} is twice continuously differentiable and $f$ is continuous, one can prove (by taking suitable test functions $v$) that $u$ satisfies \eqref{eq:bvp1d:strong}. On the other hand, there are solutions of \eqref{eq:bvp1d:weak} even for a discontinuous right-hand side $f$. Since then the second derivative of $u$ is discontinuous, $u$ is not necessarily a solution of \eqref{eq:bvp1d:strong}. For this reason, $u\in V$ satisfying \eqref{eq:bvp1d:weak} is called a \emph{weak solution} of \eqref{eq:bvp1d:strong}.

Note that the \emph{Dirichlet boundary condition} $u(0)=0$ appears explicitly in the definition of $V$, while the \emph{Neumann condition} $u'(1) = 0$ is implicitly incorporated in the variational formulation. In the context of finite element methods, Dirichlet conditions are therefore frequently called \emph{essential conditions}, while Neumann conditions are referred to as \emph{natural conditions}.

\section{Ritz--Galerkin approximation}\label{sec:overview:galerkin}

The fundamental idea is now to approximate $u$ by considering \eqref{eq:bvp1d:weak} on a \emph{finite-dimensional} subspace $S\subset V$. We are thus looking for $u_S\in S$ satisfying
\begin{equation}\label{eq:bvp1d:galerkin}
    a(u_S,v_S) = (f,v_S) \qquad \text{ for all } v_S\in S.\tag{W$_S$}
\end{equation}
Note that this is still the same equation; only the function spaces have changed. This is a crucial point in (conforming) finite element methods. (Nonconforming methods, for which $S\nsubseteq V$ or $v\notin V$, will be treated in \cref{part:nonconforming}.)

We first have to ask whether \eqref{eq:bvp1d:galerkin} has a unique solution. Since $S$ is finite-dimensional, there exists a basis $\phi_1,\dots,\phi_n$ of $S$. Due to the bilinearity of $a(\cdot,\cdot)$, it suffices to require that $u_S = \sum_{i=1}^n U_i \phi_i \in S$, $U_i\in\R$ for $i=1,\dots,n$, satisfies
\begin{equation*}
    a(u_S,\phi_j) = (f,\phi_j) \qquad \text{ for all } 1\leq j \leq n.
\end{equation*}
This is now a system of linear equations for the unknown coefficients $U_i$. If we define
\begin{align*}
    \mathbf{U} &= (U_1,\dots,U_n)^T\in\R^n,\\
    \mathbf{F} &= (F_1,\dots,F_n)^T\in\R^n,\quad F_i = \inner{f,\phi_i},\\
    \mathbf{K} &= (K_{ij}) \in \R^{n\times n}, \quad\qquad K_{ij} = a(\phi_i,\phi_j),
\end{align*}
we have that $u_S$ satisfies \eqref{eq:bvp1d:galerkin} if and only if (\enquote{iff}) $\mathbf{KU} = \mathbf{F}$. This linear system has a unique solution iff $\mathbf{KV} = {0}$ implies $\mathbf{V}={0}$. To show this, we set $v_S := \sum_{i=1}^n V_i \phi_i \in S$. Then,
\begin{equation*}
    {0} = \mathbf{KV} = (a(v_S,\phi_1),\dots,a(v_S,\phi_n))^T
\end{equation*}
implies that
\begin{equation*}
    0 = \sum_{i=1}^n V_i a(v_S,\phi_i) = a(v_S,v_S) = \int_0^1 v_S'(x)^2\, dx.
\end{equation*}
This means that $v_S'$ must vanish almost everywhere and thus that $v_S$ is constant. (This argument will be made rigorous in the next chapter.) Since $v_S(0)=0$, we deduce that $v_S\equiv 0$, and hence it follows for the linear independence of the $\phi_i$ that $V_i=0$ for all $1\leq i \leq n$.

There are two remarks to made here. First, we have argued unique solvability of the finite-dimensional system by appealing to the properties of the variational problem to be approximated. This is a standard argument in finite element methods, and the fact that the approximation \enquote{inherits} the well-posedness of the variational problem is one of the strengths of the Galerkin approach. Second, this argument shows that the \emph{stiffness matrix} $\mathbf{K}$ is (symmetric and) positive definite, since $\mathbf{V}^T\mathbf{KV} = a(v_S,v_S) > 0$ for all $\mathbf{V}\neq 0$.

\bigskip

Now that we have an approximate solution $u_S\in S$, we are interested in estimating the \emph{discretization error} $\norm*{}{u_S-u}$, which of course depends on the choice of $S$. The fundamental observation is that by subtracting \eqref{eq:bvp1d:weak} and \eqref{eq:bvp1d:galerkin} for the same test function $v_S\in S$, we obtain
\begin{equation}\label{eq:bvp1d:galerkin_orth}
    a(u-u_S,v_S) = 0 \quad\text{ for all } v_S\in S.
\end{equation}
This key property is called \emph{Galerkin orthogonality}, and expresses that the discretization error is (in some sense) orthogonal to $S$. This can be exploited to derive error estimates in the \emph{energy norm}
\begin{equation*}
    \norm*{E}{v}^2 = a(v,v) \quad \text{ for } v\in V.
\end{equation*}
It is straightforward to verify that this indeed defines a norm, which satisfies the \emph{Cauchy--Schwarz} inequality
\begin{equation*}
    a(v,w) \leq \norm*{E}{v}\norm*{E}{w} \quad \text{ for all } v,w\in V.
\end{equation*}
We can thus show that for any $v_S\in S$,
\begin{equation*}
    \begin{aligned}
        \norm*{E}{u-u_S}^2 &= a(u-u_S,u-v_S) + a(u-u_S,v_S-u_S) \\
        &= a(u-u_S,u-v_S) \\
        &\leq \norm*{E}{u-u_S}\norm*{E}{u-v_S}
    \end{aligned}
\end{equation*}
due to the Galerkin orthogonality for $v_S-u_S\in S$. Taking the infimum over all $v_S$, we obtain
\begin{equation*}
    \norm*{E}{u-u_S} \leq \inf_{v_S\in S}\norm*{E}{u-v_S},
\end{equation*}
and equality holds -- and hence this infimum is attained -- for $u_S\in S$ solving \eqref{eq:bvp1d:galerkin}. The discretization error is thus completely determined by the approximation error of the solution $u$ of \eqref{eq:bvp1d:weak} by functions in $S$:
\begin{equation}\label{eq:bvp1d:cea}
    \norm*{E}{u-u_S} = \min_{v_S\in S}\norm*{E}{u-v_S}.
\end{equation}

To derive error estimates in the $\m{l2}(0,1)$ norm
\begin{equation*}
    \norm{l2}{v}^2 = \inner{v,v} = \int_0^1v(x)^2\,dx ,
\end{equation*}
we apply a \emph{duality argument} (also called \emph{Aubin--Nitsche trick}). Let $w$ be the solution of the \emph{dual} (or \emph{adjoint}) \emph{problem}
\begin{equation}\label{eq:bvp1d:dual}
    \left\{ \begin{aligned}
            -w''(x) &= u(x)-u_S(x) \quad \text{for }x\in(0,1),\\
            w(0) &= 0, \qquad w'(1) = 0.
    \end{aligned} \right.
\end{equation}
Inserting this into the error and integrating by parts (using $(u-u_S)(0)=0=w'(1)$ and adding the productive zero), we obtain for all $v_S\in S$ the estimate
\begin{equation*}
    \begin{aligned}
        \norm{l2}{u-u_S}^2 &= \inner{u-u_S,u-u_S} = \inner{u-u_S,-w''}\\
        &= \inner{(u-u_S)',w'}\\
        &= a(u-u_S,w) - a(u-u_S,v_S)\\
        &= a(u-u_S,w-v_S)\\
        &\leq \norm*{E}{u-u_S}\norm*{E}{w-v_S}.
    \end{aligned}
\end{equation*}
Dividing by $\norm{l2}{u-u_S}=\norm{l2}{w''}$ from \eqref{eq:bvp1d:dual} and taking the infimum over all $v_S\in S$ yields
\begin{equation*}
    \norm{l2}{u-u_S} \leq \inf_{v_S\in S} \norm*{E}{w-v_S}\norm*{E}{u-u_S}\norm{l2}{w''}^{-1}.
\end{equation*}
To continue, we require an \emph{approximation property} for $S$: There exists a constant $c_S >0$ such that
\begin{equation}\label{eq:bvp1d:approx}
    \inf_{v_S\in S} \norm*{E}{g-v_S} \leq c_S \norm{l2}{g''}
\end{equation}
holds for sufficiently smooth $g\in V$. If we can apply this estimate to $w$ and $u$, we obtain
\begin{equation*}
    \begin{aligned}
        \norm{l2}{u-u_S} &\leq c_S \norm*{E}{u-u_S} = c_S \min_{v_S\in S}\norm*{E}{u-v_S}\\
        &\leq c_S^2 \norm{l2}{u''} = c_S^2 \norm{l2}{f}.
    \end{aligned}
\end{equation*}
This is another key observation: The error estimate depends on the regularity of the weak solution $u$, and hence on the data $f$. The smoother $u$, the better the approximation. Of course, we wish that $c_S$ can be made arbitrarily small by choosing $S$ sufficiently large.
The finite element method is characterized by a special class of subspaces -- of piecewise polynomials -- which have these approximation properties.

\section{Approximation by piecewise polynomials}\label{sec:overview:ppoly}

Given a set of \emph{nodes}
\begin{equation*}
    0 = x_0 < x_1 < \dots < x_n = 1,
\end{equation*}
set
\begin{equation*}
    S := \setof{v\in\m{c}(0,1)}{v|_{[x_{i-1},x_i]}\in P_1 \text{ and } v(0) = 0},
\end{equation*}
where $P_1$ is the space of all linear polynomials. (The fact that $S\subset V$ is not obvious, and will be proved later.) This is a subspace of the space of linear splines. A basis of $S$, which is especially convenient for the implementation, is formed by the linear B-splines  (\emph{hat functions})
\begin{equation*}
    \phi_i(x) = \begin{cases} \frac{x-x_{i-1}}{x_i-x_{i-1}} & \text{ if } x\in[x_{i-1},x_i],\\
        \frac{x_{i+1}-x}{x_{i+1}-x_i} & \text{ if } x\in[x_{i},x_{i+1}] \text{ and }i<n,\\
        0 & \text{ else,}
    \end{cases}
\end{equation*}
for $1\leq i \leq n$, which satisfy $\phi_i(0)=0$ and hence $\phi_i\in S$. Furthermore,
\begin{equation*}
    \phi_i(x_j) = \delta_{ij}:= \begin{cases} 1 & \text{ if } i=j,\\
        0& \text{ if } i\neq j.
    \end{cases}
\end{equation*}
This \emph{nodal basis property} immediately yields linear independence of the $\phi_i$. To show that the $\phi_i$ span $S$, we consider the \emph{interpolant} $v_I\in S$ of a given $v\in V$, defined via
\begin{equation*}
    v_I := \sum_{i=1}^n v(x_i)\phi_i(x).
\end{equation*}
For $v_S\in S$, the interpolation error $v_S-(v_S)_I$ is piecewise linear as well, and since $(v_S)_I(x_i)=v_S(x_i)$ for all $1\leq i \leq n$, this implies that $v_S-(v_S)_I\equiv 0$. Any $v_S\in S$ can thus be written as a unique linear combination of $\phi_i$ (given by its interpolant), and hence the $\phi_i$ form a basis of $S$.  We also note that this implies that the \emph{interpolation operator} $\calI:V\to S$, $v\mapsto v_I$ is a projection (i.e., $\calI\circ\calI = \calI$).

We are now in a position to prove the approximation property \eqref{eq:bvp1d:approx} of $S$. Let
\begin{equation*}
    h:= \max_{1\leq i \leq n} h_i ,\qquad h_i:=x_i-x_{i-1},
\end{equation*}
denote the \emph{mesh size}. Since the best approximation error  is certainly not bigger than the interpolation error, it suffices to show that there exists a constant $C>0$ such that for all sufficiently smooth $u\in V$,
\begin{equation*}
    \norm*{E}{u-u_I}\leq C h \norm{l2}{u''}.
\end{equation*}
We now consider this error separately on each \emph{element} $[x_{i-1},x_i]$, i.e., we show that
\begin{equation*}
    \int_{x_{i-1}}^{x_i} (u-u_I)'(x)^2\,dx \leq C^2 h_i^2  \int_{x_{i-1}}^{x_i} u''(x)^2\,dx.
\end{equation*}
First, since $u_I$ is piecewise linear, the error $e:=u-u_I$ satisfies $(e|_{[x_{i-1},x_i]})'' = (u|_{[x_{i-1},x_i]})''$. Using the affine transformation $\tilde e(t):=e(x(t))$ with $x(t) = x_{i-1}+t(x_i-x_{i-1})$ (a \emph{scaling argument}), the previous estimate is equivalent to
\begin{equation}\label{eq:bvp1d:poincare}
    \int_0^1 \tilde e'(t)^2\,dt \leq C^2 \int_0^1 \tilde e''(t)^2\,dt.
\end{equation}
(This is an elementary version of \emph{Poincar\'e's inequality}). Since $u_I$ is the nodal interpolant of $u$, the error satisfies $e(x_{i-1}) = e(x_{i}) = 0$. In addition, $u_I$ is linear and $u$ continuously differentiable on $[x_{i-1},x_i]$. Hence, $\tilde e$ is continuously differentiable on $[0,1]$ with $\tilde e(0) = \tilde e(1) = 0$, and Rolle's theorem yields a $\xi\in(0,1)$ with $\tilde e'(\xi) = 0$. Thus, for all $y\in [0,1]$ we have (with $\int_a^b f(t)\,dt = -\int_b^a f(t)\,dt$ for $a>b$)
\begin{equation*}
    \tilde e'(y) = \tilde e'(y) - \tilde e'(\xi) =  \int_\xi^y \tilde e''(t)\,dt.
\end{equation*}
We can now use the Cauchy--Schwarz inequality to estimate
\begin{equation*}
    \begin{aligned}
        |\tilde e'(y)|^2 &= \left|\int_\xi^y \tilde e''(t)\,dt\right|^2
        \leq \left|\int_\xi^y 1^2\,dt\right| \cdot \left| \int_\xi^y \tilde e''(t)^2\,dt\right|\\
        &\leq |y-\xi| \int_0^1 \tilde e''(t)^2\,dt.
    \end{aligned}
\end{equation*}
Integrating both sides with respect to $y$ and taking the supremum over all $\xi\in(0,1)$ yields \eqref{eq:bvp1d:poincare} with
\begin{equation*}
    C^2 := \sup_{\xi\in(0,1)} \int_0^1 |y-\xi|\,dy = \frac12.
\end{equation*}

Summing over all elements and estimating $h_i$ by $h$ shows the approximation property \eqref{eq:bvp1d:approx} for $S$ with $c_S:=Ch$.
For this choice of $S$, the solution $u_S$ of \eqref{eq:bvp1d:galerkin} satisfies
\begin{equation*}
    \norm*{E}{u-u_S} \leq \min_{v_S\in S}\norm*{E}{u-v_S} \leq \norm*{E}{u-u_I}\leq Ch \norm{l2}{u''}
\end{equation*}
as well as
\begin{equation}\label{eq:bvp1d:apriori}
    \norm{l2}{u-u_S} \leq C^2 h^2 \norm{l2}{u''}.
\end{equation}
These are called \emph{a priori estimates}, since they only require knowledge of the given data $f=u''$ but not of the solution $u_S$. They tell us that if we can make the mesh size $h$ arbitrarily small, we can approximate the solution $u$ of \eqref{eq:bvp1d:weak} arbitrarily well. Note that the power of $h$ is one order higher for the $\m{l2}(0,1)$ norm compared to the energy norm, which represents the fact that it is more difficult to control errors in the derivative than errors in the function value.

\section{Implementation}

As seen in \cref{sec:overview:galerkin}, the numerical computation of $u_S\in S$ boils down to solving the linear system $\mathbf{KU}=\mathbf{F}$ for the vector of coefficients $\mathbf{U}$. The missing step is the computation of the elements $K_{ij} = a(\phi_i,\phi_j)$ of $\mathbf{K}$ and the entries $F_j = \inner{f,\phi_j}$ of $\mathbf{F}$. (This procedure is called \emph{assembly}.) In principle, this can be performed by computing the integrals for each pair $(i,j)$ in a nested loop (\emph{node-based assembly}). A more efficient approach (especially in higher dimensions) is \emph{element-based assembly}: The integrals are split into sums of contributions from each element, e.g.,
\begin{equation*}
    a(\phi_i,\phi_j) = \int_0^1  \phi_i'(x)\phi_j'(x)\,dx = \sum_{k=1}^n \int_{x_{k-1}}^{x_k} \phi_i'(x)\phi_j'(x)\,dx=:\sum_{k=1}^n a_k(\phi_i,\phi_j),
\end{equation*}
and the contributions from a single element for all $(i,j)$ are computed simultaneously. Here we can exploit that by its definition, $\phi_i$ is non-zero only on the two elements $[x_{i-1},x_i]$ and $[x_i,x_{i+1}]$. Hence, for each element $[x_{k-1},x_k]$, the integrals are non-zero only for pairs $(i,j)$ with $k-1\leq i,j\leq k$. Note that this implies that $\mathbf{K}$ is tridiagonal and therefore \emph{sparse} (meaning that the number of non-zero elements grows as $n$, not $n^2$), which allows efficient solution of the linear system even for large $n$, e.g., by the method of conjugate gradients (since $\mathbf{K}$ is also symmetric and positive definite).

Another useful observation is that except for an affine transformation, the basis functions are the same on each element. We can thus use the substitution rule to transform the integrals over $[x_{k-1},x_k]$ to the \emph{reference element} $[0,1]$.
Setting $\xi(x) = \frac{x-x_{k-1}}{x_k-x_{k-1}}$ and
\begin{equation*}
    \hat \phi_1(\xi) = 1-\xi,\qquad \hat\phi_2(\xi) = \xi,
\end{equation*}
we have that $\phi_{k-1}(x) = \hat\phi_1(\xi(x))$ and $\phi_k(x) = \hat\phi_2(\xi(x))$. Using $\xi'(x) =(x_k-x_{k-1})^{-1}=h_k^{-1}$, the integrals for $i,j\in\{k-1,k\}$ can therefore be computed via
\begin{equation*}
    \int_{x_{k-1}}^{x_k}\phi_i'(x)\phi_j'(x)\,dx = h_k^{-1}\int_0^1\hat\phi_{\tau(i)}'(\xi)\hat\phi_{\tau(j)}'(\xi)\,d\xi,
\end{equation*}
where
\begin{equation*}
    \tau(i) =
    \begin{cases}
        1 &\text{if } i=k-1,\\
        2 &\text{if } i=k,
    \end{cases}
\end{equation*}
is the so-called \emph{global-to-local index}. (Correspondingly, the inverse mapping $\tau^{-1}$ is called the \emph{local-to-global index}.) Since the derivatives of $\hat \phi_1,\hat\phi_2$ are constant, the contribution from the element $[x_{k-1},x_k]$ to $K_{ij}=a(\phi_i,\phi_j)$ for $i,j\in\{k-1,k\}$ (the contribution for all other pairs $(i,j)$ being zero) is thus
\begin{equation*}
    a_k(\phi_i,\phi_j) = \begin{cases} \phantom{-}h_k^{-1} & \text{if } i=j,\\
        -h_k^{-1}&  \text{if } i\neq j.
    \end{cases}
\end{equation*}

The right-hand side $\inner{f,\phi_j}$ can be computed in a similar way, using numerical quadrature if necessary. Alternatively, one can replace $f$ by its nodal interpolant $f_I = \sum_{i=0}^n f(x_i)\phi_i$ and use
\begin{equation*}
    \inner{f,\phi_j} \approx \inner{f_I,\phi_j} = \sum_{i=0}^n f(x_i)\inner{\phi_i,\phi_j}.
\end{equation*}
The elements $M_{ij}:=\inner{\phi_i,\phi_j}$ of the \emph{mass matrix} $\mathbf{M}$ are again computed elementwise using transformation to the reference element:
\begin{equation*}
    \int_{x_{k-1}}^{x_k}\phi_i(x)\phi_j(x)\,dx = h_k \int_0^1\hat\phi_{\tau(i)}(\xi)\hat\phi_{\tau(j)}(\xi)\,d\xi =\begin{cases}
        \frac{h_k}3  & \text{if } i=j,\\
        \frac{h_k}6 &  \text{if } i\neq j.
    \end{cases}
\end{equation*}
This can be done at the same time as assembling $\mathbf{K}$. Setting $\mathbf{f}:=(f(x_1),\dots,f(x_n))^T$, the right-hand side of the linear system is then given by $\mathbf{F} = \mathbf{M}\mathbf{f}$.

Finally, the Dirichlet condition $u(0)=0$ can be enforced by replacing the first equation in the linear system by $U_0 = 0$, i.e., replacing the first row of $\mathbf{K}$ by $(1,0,\dots)$ and the first element of $\mathbf{F}$ by $0$. The main advantage of this approach is that it can easily be extended to non-homogeneous Dirichlet conditions $u(0)=g$ (by replacing the first element with $g$). The full algorithm (in \textsc{matlab}-like notation) for our boundary value problem is given in \cref{alg:bvp1d}.

\begin{algorithm}
    \caption{Finite element method in 1\textsc{d}}\label{alg:bvp1d}
    \begin{algorithmic}[1]
        \Require{$0=x_0<\dots<x_n=1$, $F:=(f(x_0),\dots,f(x_n))^T$}
        \State Set $K_{ij} = M_{ij} = 0$
        \For{$k=1,\dots,n$}
        \State Set $h_k = x_{k}-x_{k-1}$
        \State Set $K_{k-1:k,k-1:k} \gets  K_{k-1:k,k-1:k} + \frac{1}{h_k} \begin{pmatrix}1&-1\\-1&1\end{pmatrix}$
        \State Set $M_{k-1:k,k-1:k} \gets  M_{k-1:k,k-1:k} + \frac{h_k}{6} \begin{pmatrix}2&1\\1&2\end{pmatrix}$
        \EndFor
        \State $K_{0,1:n} = 0$, $K_{0,0}=1$, $M_{0,0:n} = 0$
        \State Solve $KU = MF$
        \Ensure{$U$}
    \end{algorithmic}
\end{algorithm}

\section{A posteriori error estimates and adaptivity}

The a priori estimate \eqref{eq:bvp1d:apriori} is important for proving convergence as the mesh size $h\to0$, but often pessimistic in practice since it depends on the global regularity of $u''$. If $u''(x)$ is large only in some parts of the domain, it would be preferable to reduce the mesh size locally. For this, \emph{a posteriori estimates} are useful, which are localized error estimates for each element but involve the computed solution $u_S$. This gives information on which elements should be refined (i.e., replaced by a larger number of smaller elements).

We consider again the space $S$ of piecewise linear finite elements on the nodes $x_0,\dots,x_n$ with mesh size $h$, as defined in \cref{sec:overview:ppoly}. We once more apply a duality trick: Let $w$ be the solution of
\begin{equation}\label{eq:bvp1d:dualz}
    \left\{ \begin{aligned}
            -w''(x) &= u(x)-u_S(x) \quad \text{for }x\in(0,1),\\
            w(0) &= 0, \qquad w'(1) = 0,
    \end{aligned} \right.
\end{equation}
and proceed as before, yielding
\begin{equation*}
    \norm{l2}{u-u_S}^2 = a(u-u_S,w-v_S)
\end{equation*}
for all $v_S\in S$. We now choose $v_S=w_I\in S$, the interpolant of $w$. Then we have
\begin{equation*}
    \begin{aligned}
        \norm{l2}{u-u_S}^2 &= a(u-u_S,w-w_I) = a(u,w-w_I) - a(u_S,w-w_I) \\
        &= \inner{f,w-w_I} - a(u_S,w-w_I).
    \end{aligned}
\end{equation*}
Note that the unknown solution $u$ of \eqref{eq:bvp1d:weak} no longer appears on the right-hand side. We now use the specific choice of $v_S$ to localize the error inside each element $[x_{i-1},x_{i}]$: Writing the integrals over $[0,1]$ as sums of integrals over the elements, we can integrate by parts on each element and use the fact that $(w-w_I)(x_i)=0$ to obtain
\begin{equation*}
    \begin{aligned}
        \norm{l2}{u-u_S}^2&= \sum_{i=1}^n \int_{x_{i-1}}^{x_i} f(x)(w-w_I)(x)\,dx -  \sum_{i=1}^n \int_{x_{i-1}}^{x_i} u_S'(x)(w-w_I)'(x)\,dx \\
        &=  \sum_{i=1}^n \int_{x_{i-1}}^{x_i} (f+u_S'')(x)(w-w_I)(x)\,dx\\
        &\leq  \sum_{i=1}^n \left(\int_{x_{i-1}}^{x_i} (f+u_S'')(x)^2\,dx\right)^{\half}\left(\int_{x_{i-1}}^{x_i} (w-w_I)(x)^2\,dx\right)^{\half}
    \end{aligned}
\end{equation*}
by the Cauchy--Schwarz inequality. The first term contains the \emph{finite element residual}
\begin{equation*}
    R_h:= f + u_S'',
\end{equation*}
which we can evaluate after computing $u_S$. For the second term, one can show (similarly as in the proof of the a priori error estimate \eqref{eq:bvp1d:apriori}) that
\begin{equation*}
    \left(\int_{x_{i-1}}^{x_i} (w-w_I)(x)^2\,dx\right)^\half  \leq \frac{h_i^2}{2}\norm{l2}{w''}
\end{equation*}
holds, from which we obtain
\begin{equation*}
    \begin{aligned}
        \norm{l2}{u-u_S}^2 &\leq \frac1{2}\norm{l2}{w''}\sum_{i=1}^n h_i^2  \norm*{\m{l2}(x_{i-1},x_i)}{R_h}\\
        &=  \frac1{2}\norm{l2}{u-u_S}\sum_{i=1}^n h_i^2  \norm*{\m{l2}(x_{i-1},x_i)}{R_h}
    \end{aligned}
\end{equation*}
by the definition of $w$. This yields the \emph{a posteriori estimate}
\begin{equation*}
    \norm{l2}{u-u_S} \leq \frac1{2}\sum_{i=1}^n h_i^2  \norm*{\m{l2}(x_{i-1},x_i)}{R_h}.
\end{equation*}

\newpage
This estimate can be used for an adaptive procedure: Given a tolerance $\tau>0$,
\begin{algorithmic}[1]
    \State choose initial mesh $0=x_0^{(0)}< \dots x_{n^{(0)}}^{(0)} = 1$, compute corresponding solution $u_{S^{(0)}}$, evaluate $R_{h^{(0)}}$, set $m=0$
    \While{$\sum_{i=1}^{n^{m+1}} (h^{(m)}_i)^2  \norm*{\m{l2}(x^{(m)}_{i-1},x^{(m)}_i)}{R_{h^{(m)}}}\geq \tau$}
    \State choose new mesh $0=x_0^{(m+1)}< \dots x_{n^{(m+1)}}^{(m+1)} = 1$
    \State compute corresponding solution $u_{S^{(m+1)}}$
    \State evaluate $R_{h^{(m+1)}}$
    \State set $m\gets m+1$
    \EndWhile
\end{algorithmic}

There are different strategies to choose the new mesh. A common requirement is that the strategy should be \emph{reliable}, meaning that the error on the new mesh in a certain norm can be guaranteed to be less than a given tolerance, as well as \emph{efficient}, meaning that the number of new nodes should not be larger than necessary. One (simple) possibility is to refine those elements where $\norm*{}{R_h}$ is largest (or larger than a given threshold) by replacing them with two elements of half size.

 \chapter{Variational theory of elliptic PDEs}\label{chap:weak_elliptic}

 In this chapter, we collect -- for the most part without proof -- some necessary results from functional analysis and the weak theory of (elliptic) partial differential equations.
 Details and proofs can be found in, e.g., \cite{Adams:2003a},  \cite{Evans:2010} and \cite{Zeidler:1995a}.

 \section{Function spaces}\label{sec:sobolev}

 As we have seen, the regularity of the solution of partial differential equations plays a crucial role in how well it can be approximated numerically. This regularity can be described by the two properties of (Lebesgue-)\emph{integrability} and \emph{differentiability}.

 \paragraph{Lebesgue spaces}
 Let $\Omega$ be an open subset of $\R^n$, $n\in\N$. We recall that for $1\leq p \leq \infty$,
 \begin{equation*}
     \m{Lp} := \setof{f \text{ measurable}}{\norm{Lp}{f}<\infty}
 \end{equation*}
 with
 \begin{align*}
     \norm{Lp}{f} &:= \left(\int_\Omega |f(x)|^p\,dx\right)^\frac1p \quad\text{ for }1\leq p<\infty,\\
     \norm{Linf}{f} &:= \mathrm{ess}\sup_{x\in\Omega} |f(x)|,
 \end{align*}
 are Banach spaces of (equivalence classes up to equality apart from a set of zero measure of) Lebesgue-integrable functions. The corresponding norms satisfy \emph{Hölder's inequality}
 \begin{equation*}
     \norm{L1}{fg} \leq \norm{Lp}{f}\norm{Lq}{g}
 \end{equation*}
 if $p^{-1} + q^{-1} = 1$ (with $\infty^{-1}:=0$). For bounded $\Omega$, this implies  that $\m{Lp}\hookrightarrow\m{Lq}$ for $p\geq q$. We will also use the space
 \begin{equation*}
     \m{L1loc} := \setof{f}{f|_K \in{L}^1(K) \text{ for all compact } K\subset \Omega}.
 \end{equation*}

 For $p=2$, $\m{Lp}$ is a Hilbert space with inner product
 \begin{equation*}
     \inner{f,g}:=\scalprod{L2}{f}{g} = \int_{\Omega} f(x)g(x)\,dx,
 \end{equation*}
 and Hölder's inequality for $p=q=2$ reduces to the \emph{Cauchy--Schwarz inequality}.

 \paragraph{Hölder spaces}
 We now consider functions which are continuously differentiable. It will be convenient to use a \emph{multi-index}
 \begin{equation*}
     \alpha := (\alpha_1,\dots,\alpha_n)\in\N^n,
 \end{equation*}
 for which we define its \emph{length} $|\alpha| := \sum_{i=1}^n \alpha_i$, to describe the (partial) \emph{derivative of order~$|\alpha|$},
 \begin{equation*}
     D^\alpha f(x_1,\dots,x_n):= \frac{\partial^{|\alpha|}f(x_1,\dots,x_n)}{\partial x_1^{\alpha_1}\cdots\partial x_n^{\alpha_n}}.
 \end{equation*}
 For brevity, we will often write $\partial_i:=\frac{\partial}{\partial x_i}$. We denote by $\m{Ck}$ the set of all continuous functions $f$ for which $D^\alpha f$ is continuous for all $|\alpha|\leq k$. If $\Omega$ is bounded, $\m{Ckb}$ is the set of all functions in $\m{Ck}$ for which all $D^\alpha f$ can be extended to a continous function on $\bar\Omega$, the closure of $\Omega$. These spaces are Banach spaces if equipped with the norm
 \begin{equation*}
     \norm{Ckb}{f} = \sum_{|\alpha|\leq k} \sup_{x\in\bar\Omega} |D^{\alpha} f(x)|.
 \end{equation*}
 Finally, we define $\m{Ck0}$ as the space of all $f\in\m{Ckb}$ whose support (the closure of $\setof{x\in\Omega}{f(x) \neq 0}$) is a compact subset of $\Omega$, as well as
 \begin{equation*}
     \m{Cinf0} = \bigcap_{k\geq 0} \m{Ck0}
 \end{equation*}
 (and similarly $\m{Cinf}$).

 \paragraph{Sobolev spaces}

 If we are interested in weak solutions, it is clear that the Hölder spaces entail a too strong notion of (pointwise) differentiability. All we required is that the derivative is integrable, and that an integration by parts is meaningful. This motivates the following definition: A function $f\in\m{L1loc}$ has a \emph{weak derivative} if there exists $g\in\m{L1loc}$ such that
 \begin{equation}\label{eq:weak:derivative}
     \int_\Omega g(x)\phi(x) \,dx = (-1)^{|\alpha|}\int_\Omega f(x) D^{\alpha}\phi(x)\,dx
 \end{equation}
 for all $\phi\in\m{Cinf0}$. In this case, the weak derivative is (uniquely) defined as $D^\alpha f := g$. For $f\in \m{Ck}$, the weak derivative coincides with the usual (pointwise) derivative (justifying the abuse of notation), but the weak derivative exists for a larger class of functions such as continuous and piecewise smooth functions. For example, $f(x) = |x|$, $x\in\Omega = (-1, 1)$, has the weak derivative $Df(x) = \mathop{\mathrm{sign}}(x)$, while $Df(x)$ itself does not have any weak derivative.

 We can now define the \emph{Sobolev spaces} $\m{Wkp}$ for $k\in\N_0$ and $1\leq p \leq \infty$:
 \begin{equation*}
     \m{Wkp} := \setof{f\in\m{Lp}}{D^\alpha f\in\m{Lp} \text{ for all } |\alpha|\leq k},
 \end{equation*}
 which are Banach spaces when endowed with the norm
 \begin{align*}
     \norm{Wkp}{f} &:= \left(\sum_{|\alpha|\leq k}\norm{Lp}{D^\alpha f}^p\right)^\frac1p \quad\text{ for } 1\leq p <\infty,\\
     \norm*{{W}^{k,\infty}(\Omega)}{f} &:= \sum_{|\alpha|\leq k}\norm{Linf}{D^\alpha f}.
 \end{align*}
 We shall also use the corresponding semi-norms
 \begin{align*}
     |f|_{\m{Wkp}} &:= \left(\sum_{|\alpha| = k}\norm{Lp}{D^\alpha f}^p\right)^\frac1p \quad\text{ for } 1\leq p <\infty,\\
     |f|_{{W}^{k,\infty}(\Omega)} &:= \sum_{|\alpha| = k}\norm{Linf}{D^\alpha f}.
 \end{align*}

 We are now concerned with the relation between the different norms introduced so far. For many of these results to hold, we require that the boundary $\partial\Omega$ of $\Omega$ is sufficiently smooth. We shall henceforth assume -- if not otherwise stated -- that $\Omega\subset\R^n$ has a \emph{Lipschitz boundary}, meaning that $\partial\Omega$ can be parametrized by a finite set of functions which are uniformly Lipschitz continuous. (This condition is satisfied, for example, by polygons for $n=2$ and polyhedra for $n=3$.) Similarly, a \emph{$C^m$ boundary} can be parametrized by a finite set of $m$ times continuously differentiable functions.
 A fundamental result is then the following approximation property (which does not hold for arbitrary domains).
 \begin{theorem}[density\footnote{The key result was shown by Meyers and Serrin in a paper rightfully celebrated both for its content and the brevity of its title, \enquote{$H=W$}. For the proof, see, e.g., \cite[\S\,5.3.3, Theorem~3]{Evans:2010}, \cite[Theorem~3.17]{Adams:2003a}}]\label{thm:weak:dense}
     For $1\leq p <\infty$ and any $k\in\N_0$, $\m{Cinf}$ is dense in $\m{Wkp}$.
 \end{theorem}
 This theorem allows us to prove results for Sobolev spaces -- such as chain rules -- by showing them for smooth functions (in effect, transferring results for usual derivatives to their weak counterparts). This is called a \emph{density argument}.

 Using a density argument, one can show that Sobolev spaces behave well under sufficiently smooth coordinate transformations.
 \begin{theorem}[coordinate transformation\footnote{e.g.,\cite[Theorem 3.41]{Adams:2003a}}]\label{thm:weak:transform}
     Let $\Omega,\Omega'\subset\R^n$ be two domains, and $T:\Omega\to\Omega'$ be a $k$-diffeomorphism (i.e., $T$ is a bijection, $T$ and its inverse $T^{-1}$ are continuous with $k$ bounded and continuous derivatives on $\overline \Omega$ and $\overline \Omega'$, and the determinant of the Jacobian of $T$ is uniformly bounded from above and below). Then the mapping $v\mapsto v\circ T$ is bounded from $\m{Wkp}$ to $\mathrm{W}^{k,p}(\Omega')$ and has a bounded inverse.
 \end{theorem}
 Corresponding chain rules for weak derivatives can be obtained from the classical ones using a density argument as well. \Cref{thm:weak:transform} can also be used to define Sobolev spaces on (sufficiently smooth) manifolds via a local coordinate charts. In particular, if $\Omega$ has a $\mathrm{C}^k$ boundary, $k\geq 1$, we can define $\mathrm{W}^{k,p}(\partial\Omega)$ by (local) transformation to $\mathrm{W}^{k,p}(D)$, where $D\subset \R^{n-1}$.

 The next theorem states that, within limits determined by the spatial dimension, we can trade differentiability for integrability for Sobolev space functions.
 \begin{theorem}[Sobolev\footnote{e.g., \cite[\S\,5.6]{Evans:2010}, \cite[Theorem~4.12]{Adams:2003a}}, Rellich--Kondrachov\footnote{e.g., \cite[\S\,5.7]{Evans:2010}, \cite[Theorem~6.3]{Adams:2003a}} embedding]\label{thm:weak:sobolev}
     Let $1\leq p,q<\infty$ and $\Omega\subset\R^n$ be a bounded open set with Lipschitz boundary. Then the following embeddings are continuous:
     \begin{equation*}
         \m{Wkp} \hookrightarrow \begin{cases}
             \m{Lq} & \text{ if } p < \frac{n}{k}\text{ and } p\leq q \leq \frac{np}{n-p},\\
             \m{Lq} & \text{ if } p = \frac{n}{k}\text{ and } p\leq q < \infty,\\
             {C}^0(\bar\Omega) &\text{ if } p> \frac{n}{k}.
         \end{cases}
     \end{equation*}
     Moreover, the following embeddings are compact:
     \begin{equation*}
         \m{Wkp}\hookrightarrow  \begin{cases} \m{Lq} & \text{ if } p \leq \frac{n}{k}\text{ and } 1\leq q < \frac{n-pk}{np},\\
             {C}^0(\bar\Omega) &\text{ if } p> \frac{n}{k}.
         \end{cases}
     \end{equation*}
     In particular, the embedding $\m{Wkp}\hookrightarrow{W}^{k-1,p}(\Omega)$ is compact for all $k$ and $1\leq p\leq \infty$.
 \end{theorem}

 We can also ask if conversely, continuous functions are weakly differentiable. Intuitively, this is the case if the points of (classical) non-differentiability form a set of Lebesgue measure zero. Indeed, continuous and piecewise differentiable functions are weakly differentiable.
 \begin{theorem}\label{thm:weak:piecewise}
     Let $\Omega\subset\R^n$ be a bounded Lipschitz domain which can be partitioned into $N\in \N$  Lipschitz subdomains $\Omega_j$ (i.e., $\bar\Omega = \bigcup_{j=1}^N \bar\Omega_j$ and $\Omega_i\cap\Omega_j = \emptyset$ for all $i\neq j$). Then for every $k\geq 1$ and $1\leq p \leq \infty$,
     \begin{equation*}
         \setof{v\in{C}^{k-1}(\bar\Omega)}{v|_{\Omega_j} \in {C}^{k}(\bar\Omega_j), 1\leq j\leq N}
         \hookrightarrow \m{Wkp}.
     \end{equation*}
 \end{theorem}
 \begin{proof}
     It suffices to show the inclusion for $k=1$. Let $v\in{C}^0(\bar\Omega)$ such that $v|_{\Omega_j} \in {C}^{1}(\bar\Omega_j)$ for all $1\leq j\leq N$. We need to show that $\partial_i v$ exists as a weak derivative for all $1\leq i\leq n$ and that $\partial_i v\in\m{Lp}$. An obvious candidate is
     \begin{equation*}
         w_i := \begin{cases} \partial_i v|_{\Omega_j}(x) &\text{if } x\in\Omega_j \text{ for some } j\in\{1,\dots,N\},\\
             c &\text{else}
         \end{cases}
     \end{equation*}
     for arbitrary $c\in\R$.
     By the embedding ${C}^0(\bar\Omega_j)\hookrightarrow\m{linf}(\Omega_j)$ and the boundedness of $\Omega$, we have that $w_i\in\m{Lp}$ for any $1\leq p\leq \infty$. It remains to verify \eqref{eq:weak:derivative}. By splitting the integration into a sum over the $\Omega_j$ and integrating by parts on each subdomain (where $v$ is continuously differentiable), we obtain for any $\phi\in\m{Cinf0}$
     \begin{equation*}
         \begin{aligned}
             \int_\Omega w_i(x)\phi(x)\,dx &= \sum_{j=1}^N\int_{\Omega_j}\partial_i (v|_{\Omega_j})(x)\phi(x)\,dx\\
             &=  \sum_{j=1}^N\int_{\partial\Omega_j} v|_{\Omega_j}(x) \phi(x)\, [\nu_j(x)]_i\,dx -  \sum_{j=1}^N\int_{\Omega_j} v|_{\Omega_j}(x) \partial_i\phi(x) \,dx\\
             &=  \sum_{j=1}^N\int_{\partial\Omega_j} v|_{\Omega_j}(x) \phi(x) \,[\nu_j(x)]_i\,dx -\int_{\Omega} v(x)\partial_i\phi(x) \,dx,
         \end{aligned}
     \end{equation*}
     where $\nu_j = ((\nu_j)_1,\dots,(\nu_j)_n)$ is the outer normal vector to $\Omega_j$, which exists almost everywhere since $\Omega_j$ is a Lipschitz domain. Now the sum over the boundary integrals vanishes since either $\phi(x) = 0$ if $x\in\partial\Omega_j\subset\partial\Omega$ or  $ v|_{\Omega_j}(x) \phi(x) (\nu_j)_i(x) = -v|_{\Omega_k}(x) \phi(x) (\nu_k)_i(x)$ if $x\in\partial\Omega_j \cap \partial\Omega_k$ due to the continuity of $v$. This implies $\partial_i v = w_i$ by definition.
 \end{proof}

 Next, we would like to see how Dirichlet boundary conditions make sense for weak solutions. For this, we define a \emph{trace operator} $T$ (via limits of approximating continous functions) which maps a function $f$ on a bounded domain $\Omega\subset\R^n$ to a function $Tf$ on $\partial\Omega$.
 \begin{theorem}[trace theorem\footnote{e.g., \cite[\S\,5.5]{Evans:2010}, \cite[Theorem~5.36]{Adams:2003a}, \cite[Theorem~1.5.2.8]{Grisvard:1985}}]
     \label{thm:weak:trace}
     Let $kp<n$ and $p\leq q \leq (n-1)p/(n-kp)$, and $\Omega\subset\R^n$ be a bounded open set with $C^m$ boundary or a polygon in $\R^2$. Then $T:\m{Wkp}\to{L}^q(\partial\Omega)$ is a bounded linear operator, i.e., there exists a constant $C>0$ depending only on $p$ and $\Omega$ such that for all $f\in\m{Wkp}$,
     \begin{equation*}
         \norm*{{L}^q(\partial\Omega)}{Tf} \leq  C \norm{Wkp}{f}.
     \end{equation*}
     If $kp=n$, this holds for any $p\leq q<\infty$.
 \end{theorem}

 This implies (although it is not obvious)\footnote{e.g., \cite[\S\,5.5, Theorem~2]{Evans:2010}, \cite[Theorem~5.37]{Adams:2003a}} that
 \begin{equation*}
     \m{Wkp0} := \setof{f\in\m{Wkp}}{T(D^\alpha f) = 0\in{L}^p(\partial\Omega) \text{ for all } |\alpha| < k}
 \end{equation*}
 is well-defined, and that $\m{Wkp}\cap\m{Cinf0}$ is dense in $\m{Wkp0}$.

 For functions in $\m{W1p0}$, the semi-norm $|\cdot|_{\m{W1p}}$ is equivalent to the full norm $\norm{W1p}{\cdot}$.
 \begin{theorem}[Poincaré's inequality\footnote{e.g, \cite[Corollary~6.31]{Adams:2003a}}]\label{thm:weak:poincare}
     Let $1\leq p<\infty $ and let $\Omega$ be a bounded open set. Then there exists a constant $c_\Omega>0$ depending only on $\Omega$ and $p$ such that for all $f\in\m{W1p0}$,
     \begin{equation*}
         \norm{W1p}{f}\leq c_\Omega |f|_{\m{W1p}}.
     \end{equation*}
 \end{theorem}
 The proof is very similar to the argumentation in \cref{chap:overview}, using the density of $\m{Cinf0}$ in $\m{W1p0}$; in particular, it is sufficient that $Tf$ is zero on a part of the boundary $\partial\Omega$ of non-zero measure.
 In general, we have that any $f\in\m{W1p}$, $1\leq p \leq \infty$, for which $D^\alpha f = 0 $ almost everywhere in $\Omega$ for all $|\alpha|= 1$ must be constant (cf.~\cref{lem:bramble:kernel}).

 Again, $\m{Wkp}$ is a Hilbert space for $p=2$, with inner product
 \begin{equation*}
     \scalprod*{{W}^{k,2}(\Omega)}{f}{g} = \sum_{|\alpha| \leq k} \inner{D^\alpha f,D^\alpha g}.
 \end{equation*}
 For this reason, one usually writes ${H}^k(\Omega) := {W}^{k,2}(\Omega)$. In particular, we will often consider $\m{H1} := {W}^{1,2}(\Omega)$ and $\m{H10} := {W}^{1,2}_0(\Omega)$. With the usual notation $\nabla f := (\partial_1 f,\dots,\partial_n f)$ for the gradient of $f$, we can write
 \begin{equation*}
     |f|_{\m{H1}} = \norm*{\m{L2}^n}{\nabla f}
 \end{equation*}
 for the semi-norm on $\m{H1}$ (which, by the Poincar\'e inequality (\cref{thm:weak:poincare}), is equivalent to the full norm on $\m{H10}$) and
 \begin{equation*}
     \scalprod{H1}{f}{g} = \inner{f,g}+\inner{\nabla f, \nabla g}
 \end{equation*}
 for the inner product on $\m{H1}$. Finally, we denote the topological dual of $\m{H10}$ (i.e., the space of all continuous linear functionals on $\m{H10}$) by $\m{Hm1}:=(\m{H10})^*$, which is endowed with the operator norm
 \begin{equation*}
     \norm{Hm1}{f}  = \sup_{\phi\in\m{H10},\phi\neq 0} \frac{\langle f,\phi\rangle_{\m{Hm1},\m{H10}}}{\norm{H1}{\phi}},
 \end{equation*}
 where $\langle f,\phi\rangle_{V^*,V}:=f(\phi)$ denotes the \emph{duality pairing} between a Banach space $V$ and its dual $V^*$.

 We can now tie together some loose ends from \cref{chap:overview}. The space $V$ can be rigorously defined as
 \begin{equation*}
     V:= \setof{v\in{H}^1(0,1)}{ v(0) = 0},
 \end{equation*}
 which makes sense due to the embedding (for $n=1$) of ${H}^1(0,1)$ in ${C}([0,1])$.
 Due to Poincar\'e's inequality, $|v|_{\m{H1}}^2 = a(v,v) = 0$ implies $\norm{H1}{v} = 0$ and hence $v=0$. Similarly, the existence of a unique weak solution $u\in V$ follows from the Riesz representation theorem.
 Finally, \cref{thm:weak:piecewise} guarantees that $S\subset V$.

 \section{Weak solution of elliptic PDEs}

 In the first two parts, we consider \emph{boundary value problems} of the form
 \begin{equation}\label{eq:weak:bvp}
     -\sum_{j,k=1}^n \partial_j(a_{jk}(x)\partial_k u) + \sum_{j=1}^n b_j(x)\partial_ju + c (x) u = f
 \end{equation}
 on a bounded open set $\Omega\subset\R^n$, where $a_{jk}$, $b_j$, $c$ and $f$ are given functions on $\Omega$. We do not fix boundary conditions at this time. This problem is called \emph{elliptic} if there exists a constant $\alpha >0$ such that
 \begin{equation}\label{eq:weak:elliptic}
     \sum_{j,k=1}^n a_{jk}(x)\xi_j\xi_k \geq \alpha \sum_{j=1}^n \xi_j^2 \quad\text{ for all }\xi\in\R^n, x\in\Omega.
 \end{equation}

 Assuming all functions and the domain are sufficiently smooth, we can multiply by a smooth function $v$, integrate over $x\in\Omega$ and integrate by parts to obtain
 \begin{equation}\label{eq:weak:weakbvp}
     \sum_{j,k=1}^n \inner{a_{jk}\partial_j u,\partial_k v} + \sum_{j=1}^n \inner{b_j \partial_j u,v} + \inner{c u,v} - \sum_{j,k=1}^n \inner{a_{jk} \partial_k u \nu_j, v}_{\partial\Omega}= \inner{f,v},
 \end{equation}
 where $\nu:=(\nu_1,\dots,\nu_n)^T$ is the outward unit normal on $\partial\Omega$ and
 \begin{equation*}
     \inner{f,g}_{\partial\Omega} :=   \int_{\partial\Omega} f(x)g(x)\,dx,
 \end{equation*}
 where $g$ should be understood in the sense of traces, i.e., as $Tg$.
 Note that this formulation only requires $a_{jk},b_j,c\in\m{Linf}$ and $f\in\m{L2}$ in order to be well-defined.
 We then search for $u\in V$ -- for a suitably chosen function space $V$ -- satisfying \eqref{eq:weak:weakbvp} for all $v\in V$ including boundary conditions which we will discuss next. We will consider the following three conditions:

 \paragraph{Dirichlet conditions} We require $u = g$ on $\partial\Omega$ (in the sense of traces) for given $g\in\m{l2}(\partial\Omega)$. If $g=0$ (a \emph{homogeneous} Dirichlet condition), we take $V=\m{H10}$, in which case the boundary integrals in \eqref{eq:weak:weakbvp} vanish since $v=0$ on $\partial\Omega$. The weak formulation is thus: Find $u\in\m{H10}$ satisfying
 \begin{equation}\label{eq:weak:dirichlet}
     a(u,v):= \sum_{j,k=1}^n \inner{a_{jk}\partial_j u,\partial_k v} + \sum_{j=1}^n \inner{b_j \partial_j u,v} + \inner{c u,v}  = \inner{f,v}
 \end{equation}
 for all $v\in\m{H10}$.

 If $g\neq 0$, and $g$ and $\partial\Omega$ are sufficiently smooth (e.g., $g\in\m{h1}(\partial\Omega)$ with $\partial\Omega$ of class ${C}^1$),\footnote{\cite[Theorem 7.40]{Renardy:2004}} we can find a function $u_g\in\m{H1}$ such that $Tu_g=g$. We then set $u = \tilde u + u_g$, where $\tilde u\in\m{H10}$ satisfies
 \begin{equation*}
     a(\tilde u,v) = \inner{f,v} - a(u_g,v)
 \end{equation*}
 for all $v\in\m{H10}$.

 \paragraph{Neumann conditions} We require $\sum_{j,k=1}^n a_{jk} \partial_k u \nu_j  = g$ on $\partial\Omega$ for given $g\in\m{l2}(\partial\Omega)$. In this case, we can substitute this equation in the boundary integral in \eqref{eq:weak:weakbvp} and take $V=\m{H1}$. We then look for $u\in\m{H1}$ satisfying
 \begin{equation}\label{eq:weak:neumann}
     a(u,v) = \inner{f,v} + \inner{g,v}_{\partial\Omega}
 \end{equation}
 for all $v\in\m{H1}$.

 \paragraph{Robin conditions}  We require $d u + \sum_{j,k=1}^n a_{jk} \partial_k u \nu_j  = g$ on $\partial\Omega$ for given $g\in\m{l2}(\partial\Omega)$ and $d\in{L}^\infty(\partial\Omega)$. Again we can substitute this in the boundary integral and take $V=\m{H1}$. The weak form is then: Find $u\in\m{H1}$ satisfying
 \begin{equation}\label{eq:weak:robin}
     a_R(u,v):= a(u,v) + \inner{d u,v}_{\partial\Omega} = \inner{f,v} + \inner{g,v}_{\partial\Omega}
 \end{equation}
 for all $v\in\m{H1}$.

 \bigskip

 These problems have a common form: For a given Hilbert space $V$, a bilinear form $a:V\times V\to \R$ and a linear functional $F:V\to \R$ (e.g., $F:v\mapsto\inner{f,v}$ in the case of Dirichlet conditions), find $u\in V$ such that
 \begin{equation}\label{eq:weak:variational}
     a(u,v) = F(v),\qquad \text{for all } v\in V.
 \end{equation}
 The existence and uniqueness of a solution can be guaranteed by the Lax--Milgram theorem, which is a generalization of the Riesz representation theorem (note that $a$ is in general not symmetric).
 \begin{theorem}[Lax--Milgram theorem]\label{thm:weak:laxmilgram}
     Let a Hilbert space $V$,  a bilinear form $a:V\times V\to \R$ and a linear functional $F:V\to\R$ be given satisfying the following conditions:
     \begin{enumerate}[(i)]
         \item \emph{Coercivity}: There exists $c_1>0$ such that
             \begin{equation*}
                 a(v,v) \geq c_1 \norm*{V}{v}^2
             \end{equation*}
             for all $v\in V$.
         \item \emph{Continuity}: There exist $c_2,c_3>0$ such that
             \begin{align*}
                 a(v,w) &\leq c_2 \norm*{V}{v}\norm*{V}{w},\\
                 F(v) &\leq c_3 \norm*{V}{v}
             \end{align*}
             for all $v,w\in V$.
     \end{enumerate}
     Then there exists a unique solution $u\in V$ to \eqref{eq:weak:variational}, and
     \begin{equation}\label{eq:weak:laxmilgram}
         \norm*{V}{u}\leq \frac{1}{c_1} \norm*{V^*}{F}.
     \end{equation}
 \end{theorem}
 \begin{proof}
     For every fixed $u\in V$, the mapping $v\mapsto a(u,v)$ is a linear functional on $V$, which is continuous by assumption (ii), and so is $F$. By the Riesz--Fr\'echet representation theorem,\footnote{e.g., \cite[Theorem~2.E]{Zeidler:1995a}} there exist unique $\phi_u,\phi_F \in V$ such that
     \begin{equation*}
         \scalprod*{V}{\phi_u}{v} = a(u,v) \quad \text{ and } \quad \scalprod*{V}{\phi_F}{v} = F(v)
     \end{equation*}
     for all $v\in V$. We recall that  $w\mapsto \phi_w$ is a continuous linear mapping from $V^*$ to $V$ with operator norm $1$. Thus, a solution $u\in V$ satisfies
     \begin{equation*}
         0 = a(u,v) - F(v) = \scalprod*{V}{\phi_u-\phi_F}{v}
     \end{equation*}
     for all $v\in V$, which holds if and only if $\phi_u = \phi_F$ in $V$.

     We now wish to solve this equation using the Banach fixed point theorem.\footnote{e.g., \cite[Theorem~1.A]{Zeidler:1995a}} For $\delta >0$, consider the mapping
     \begin{equation*}
         T_\delta:V\to V,\qquad        T_\delta (v) = v - \delta (\phi_v -\phi_F).
     \end{equation*}
     If $T_\delta$ is a contraction, then there exists a unique fixed point $u$ such that $T_\delta(u) = u$ and hence $\phi_u - \phi_F =0$. It remains to show that there exists a $\delta >0$ such that $T_\delta$ is a contraction, i.e., there exists $0<L<1$ with $\norm*{V}{T_\delta v_1-T_\delta v_2} \leq L \norm*{V}{v_1-v_2}$. Let $v_1,v_2 \in V$ be arbitrary and set $v = v_1-v_2$. Then we have
     \begin{equation*}
         \begin{aligned}
             \norm*{V}{T_\delta v_1-T_\delta v_2}^2 &= \norm*{V}{v_1-v_2 - \delta (\phi_{v_1}-\phi_{v_2})}^2\\
             &= \norm*{V}{v - \delta \phi_v}^2 \\
             &= \norm*{V}{v}^2 - 2 \delta \scalprod*{V}{v}{\phi_v} + \delta^2 \scalprod*{V}{\phi_v}{\phi_v}\\
             &= \norm*{V}{v}^2 - 2\delta a(v,v) + \delta^2 a(v,\phi_v)\\
             &\leq \norm*{V}{v}^2 - 2\delta c_1 \norm*{V}{v}^2 + \delta^2 c_2 \norm*{V}{v}\norm*{V}{\phi_v}\\
             &\leq (1-2\delta c_1 +\delta^2 c_2)\norm*{V}{v_1-v_2}^2.
         \end{aligned}
     \end{equation*}
     We can thus choose $0<\delta<2\frac{c_1}{c_2}$ such that $L^2:=(1-2\delta c_1 +\delta^2 c_2)<1$, and the Banach fixed point theorem yields existence and uniqueness of the solution $u\in V$.

     To show the estimate \eqref{eq:weak:laxmilgram}, assume $u\neq 0$ (otherwise the inequality holds trivially). Note that $F$ is a bounded linear functional by assumption (ii), hence $F\in V^*$. We can then apply the coercivity of $a$ and divide by $\norm*{V}{u}\neq 0$ to obtain
     \begin{equation*}
         c_1 \norm*{V}{u} \leq \frac{a(u,u)}{\norm*{V}{u}} \leq \sup_{v\in V} \frac{a(u,v)}{\norm*{V}{v}} =  \sup_{v\in V} \frac{F(v)}{\norm*{V}{v}} = \norm*{V^*}{F}.
         \qedhere
     \end{equation*}
 \end{proof}

 We can now give sufficient conditions on the coefficients $a_{jk}$, $b_j$, $c$ and $d$ such that the boundary value problems defined above have a unique solution.
 \begin{theorem}[well-posedness]\label{thm:weak:well-posed}
     Let $a_{jk}\in\m{Linf}$ satisfy the ellipticity condition \eqref{eq:weak:elliptic} with constant $\alpha>0$, let $b_j,c\in\m{Linf}$ and $f\in\m{L2}$ and $g\in\m{l2}(\partial\Omega)$ be given, and set $\beta =  \alpha^{-1} \sum_{j=1}^n \norm{Linf}{b_j}^2$.
     \begin{enumerate}[a)]
         \item The homogeneous Dirichlet problem has a unique solution $u\in\m{H10}$ if
             \begin{equation*}
                 c(x) - \frac{\beta}{2}\geq 0 \quad\text{ for almost all }x \in\Omega.
             \end{equation*}
             In this case, there exists a $C>0$ such that
             \begin{equation*}
                 \norm{H1}{u} \leq C \norm{L2}{f}.
             \end{equation*}
             Consequently, the inhomogeneous Dirichlet problem for $g\in{H}^1(\partial\Omega)$ has a unique solution satisfying
             \begin{equation*}
                 \norm{H1}{u} \leq C (\norm{L2}{f}+\norm*{{H}^1(\partial\Omega)}{g}).
             \end{equation*}
         \item The Neumann problem for $g\in\m{l2}(\partial\Omega)$ has a unique solution $u\in\m{H1}$ if
             \begin{equation*}
                 c(x)- \frac{\beta}{2} \geq \gamma >0 \quad\text{ for almost all }x \in\Omega.
             \end{equation*}
             In this case, there exists a $C>0$ such that
             \begin{equation*}
                 \norm{H1}{u} \leq C (\norm{L2}{f}+\norm*{\m{l2}(\partial\Omega)}{g}).
             \end{equation*}
         \item The Robin problem for $g\in\m{l2}(\partial\Omega)$ and $d\in{L}^\infty(\partial\Omega)$ has a unique solution if
             \begin{align*}
                 c(x) - \frac{\beta}{2}&\geq \gamma \geq 0 \quad\text{ for almost all }x \in\Omega,\\
                 d(x) &\geq \delta \geq 0\quad\text{ for almost all }x \in\partial\Omega,
             \end{align*}
             and either $\gamma>0$ or $\delta>0$. In this case, there exists a $C>0$ such that
             \begin{equation*}
                 \norm{H1}{u} \leq C (\norm{L2}{f}+\norm*{\m{l2}(\partial\Omega)}{g}).
             \end{equation*}
     \end{enumerate}
 \end{theorem}
 \begin{proof}
     We apply the \nameref{thm:weak:laxmilgram}. Continuity of $a$ and $F$ follow by the Hölder inequality and the boundedness of the coefficients. It thus remains to verify the coercivity of $a$, which we only do for the case of homogeneous Dirichlet conditions (the other cases being similar). Let $v\in\m{H10}$ be given. First, the ellipticity of $a_{jk}$ implies that
     \begin{equation*}
         \int_\Omega \sum_{j,k=1}^n a_{jk} \partial_j v(x) \partial_k v(x)\,dx \geq \alpha \int_\Omega \sum_{j=1}^n \partial_j v(x)^2 \,dx  = \alpha\sum_{j=1}^n \norm{L2}{\partial_j v}^2 = \alpha |v|_{\m{H1}}^2.
     \end{equation*}
     We then have by Young's inequality $ab\leq \frac{\alpha}2 a^2+\frac1{2\alpha} b^2$ for $a=|v|_{\m{H1}}$, $b=\norm{L2}{v}$ and $\alpha>0$ as well as repeated application of Hölder's inequality that
     \begin{equation*}
         \begin{aligned}
             a(v,v) &\geq \alpha |v|_{\m{H1}}^2 - \left(\sum_{j=1}^n \norm{Linf}{b_j}^2\right)^{\frac12} |v|_{\m{H1}} \norm{L2}{v} + \int_\Omega c(x) v(x)^2\,dx\\
             & \geq \frac{\alpha}2 |v|_{\m{H1}}^2 + \int_\Omega \left(c(x) - \frac{1}{2\alpha}\sum_{j=1}^n \norm{Linf}{b_j}^2\right) |v|^2\,dx.
         \end{aligned}
     \end{equation*}
     Under the assumption that $c-\frac\beta{2}\geq 0$, the second term is non-negative and we deduce using Poincaré's inequality that
     \begin{equation*}
         a(v,v) \geq \frac\alpha{2} |v|_{\m{H1}}^2 \geq  \frac\alpha{4} |v|_{\m{H1}}^2 + \frac{\alpha}{4c_\Omega^2}\norm{L2}{v}^2 \geq C \norm{H1}{v}^2
     \end{equation*}
     for $C:= \alpha / (4+4c_\Omega^2)$, where $c_\Omega$ is the constant from Poincaré's inequality.
 \end{proof}
 Note that these conditions are not sharp; different ways of estimating the first-order terms in $a$ give different conditions. For example, if $b_j\in{W}^{1,\infty}(\Omega)$, we can take $\beta = \sum_{j=1}^n \norm{Linf}{\partial_j b_j}$.

 Naturally, if the data has higher regularity, we can expect more regularity of the solution as well. The corresponding theory is quite involved, and we give only two results which will be relevant in the following.
 \begin{theorem}[higher regularity\footnote{\cite[Theorem~2.24]{Troianiello:1987a}}]\label{thm:weak:regularity:smooth}
     Let $\Omega\subset\R^n$ be a bounded domain with ${C}^{k+1}$ boundary, $k\geq 0$, $a_{jk}\in{C}^k(\bar\Omega)$ and $b_j,c\in{W}^{k,\infty}(\Omega)$. Then for any $f\in{H}^k(\Omega)$, the solution of the homogeneous Dirichlet problem is in ${H}^{k+2}(\Omega)\cap\m{H10}$, and there exists a $C>0$ such that
     \begin{equation*}
         \norm*{{H}^{k+2}(\Omega)}{u} \leq C (\norm*{{H}^k(\Omega)}{f}+\norm{H1}{u}).
     \end{equation*}
 \end{theorem}

 \begin{theorem}[higher regularity\footnote{\cite[Theorem~5.2.2]{Grisvard:1985}, \cite[pp.~169--189]{Ladyzhenskaya:1968}}]\label{thm:weak:regularity:convex}
     Let $\Omega$ be a convex polygon in $\R^2$ or a parallelepiped in $\R^3$, $a_{jk}\in{C}^1(\bar\Omega)$ and $b_j,c\in\m{c}(\bar\Omega)$. If $f\in L^(\Omega)$, then the solution of the homogeneous Dirichlet problem is in $\m{H2}$, and there exists a $C>0$ such that
     \begin{equation*}
         \norm{H2}{u} \leq C \norm{L2}{f}.
     \end{equation*}
 \end{theorem}

 For non-convex polygons, $u\in{H}^2(\Omega)$ is not possible. This is due to the presence of so-called \emph{corner singularities} at reentrant corners, which severely limits the accuracy of finite element approximations. This requires special treatment, and is a topic of extensive current research.

\part{Conforming Finite Elements}\label{part:conforming}

\chapter{Galerkin approach for elliptic problems}\label{chap:galerkin}

We have seen that elliptic partial differential equations can be cast into the following form: Given a Hilbert space $V$, a bilinear form $a:V\times V\to \R$ and a continuous linear functional $F:V\to\R$, find $u\in V$ satisfying
\begin{equation}\label{eq:galerkin:variational}
    a(u,v) = F(v) \quad \text{ for all } v\in V.\tag{W}
\end{equation}
According to the \nameref{thm:weak:laxmilgram}, this problem has a unique solution if there exist $c_1,c_2>0$ such that
\begin{align}
    a(v,v) &\geq c_1 \norm*{V}{v}^2,\label{eq:galerkin:coercive}\\
    a(u,v) &\leq c_2 \norm*{V}{u}\norm*{V}{v},\label{eq:galerkin:continuous}
\end{align}
hold for all $u,v\in V$ (which we will assume from here on).

The \emph{conforming Galerkin approach} consists in choosing a (finite-dimensional) closed subspace $V_h\subset V$ and looking for $u_h\in V_h$ satisfying\footnote{The subscript $h$ stands for a \emph{discretization parameter}, and indicates that we expect convergence of $u_h$ to the solution of \eqref{eq:galerkin:variational} as $h\to 0$.}
\begin{equation}\label{eq:galerkin:galerkin}
    a(u_h,v_h) = F(v_h) \quad \text{ for all } v_h\in V_h.\tag{W$_h$}
\end{equation}
Since we have chosen a closed $V_h\subset V$, the subspace $V_h$ is a Hilbert space with inner product $\scalprod*{V}{\cdot}{\cdot}$ and norm $\norm*{V}{\cdot}$. Furthermore, the conditions \eqref{eq:galerkin:coercive} and \eqref{eq:galerkin:continuous} are satisfied for all $u_h,v_h\in V_h$ as well. The \nameref{thm:weak:laxmilgram} thus immediately yields the well-posedness of \eqref{eq:galerkin:galerkin}.
\begin{theorem}
    Under the assumptions of \cref{thm:weak:laxmilgram}, for any closed subspace $V_h\subset V$, there exists a unique solution $u_h\in V_h$ of \eqref{eq:galerkin:galerkin} satisfying
    \begin{equation*}
        \norm*{V}{u_h}\leq \frac{1}{c_1}\norm*{V^*}{F}.
    \end{equation*}
\end{theorem}

The following result is essential for all error estimates of Galerkin approximations.
\begin{lemma}[Céa's lemma]\label{thm:galerkin:cea}
    Let $u_h$ be the solution of \eqref{eq:galerkin:galerkin} for given $V_h\subset V$ and $u$ be the solution of \eqref{eq:galerkin:variational}. Then,
    \begin{equation*}
        \norm*{V}{u-u_h} \leq \frac{c_2}{c_1} \inf_{v_h\in V_h}\norm*{V}{u-v_h},
    \end{equation*}
    where $c_1$ and $c_2$ are the constants from \eqref{eq:galerkin:coercive} and \eqref{eq:galerkin:continuous}.
\end{lemma}
\begin{proof}
    Since $V_h\subset V$, we deduce (by subtracting \eqref{eq:galerkin:variational} and \eqref{eq:galerkin:galerkin} with the same $v=v_h\in V_h$) the \emph{Galerkin orthogonality}
    \begin{equation}\label{eq:galerkin:orthogonal}
        a(u-u_h,v_h) = 0 \quad \text{ for all } v_h\in V_h.
    \end{equation}
    Hence, for arbitrary $v_h\in V_h$, we have $v_h-u_h\in V_h$ and therefore $a(u-u_h,v_h-u_h)=0$. Using \eqref{eq:galerkin:coercive} and \eqref{eq:galerkin:continuous}, we obtain
    \begin{equation*}
        \begin{aligned}
            c_1 \norm*{V}{u-u_h}^2& \leq a(u-u_h,u-u_h)\\
            &=a(u-u_h,u-v_h) + a(u-u_h,v_h-u_h)\\
            &\leq c_2 \norm*{V}{u-u_h}\norm*{V}{u-v_h}.
        \end{aligned}
    \end{equation*}
    Dividing by $\norm*{V}{u-u_h}$, rearranging, and taking the infimum over all $v_h\in V_h$ yields the desired estimate.
\end{proof}
This implies that the error of any (conforming) Galerkin approach is determined by the approximation error of the exact solution in $V_h$. The derivation of such error estimates will be the topic of the next chapters.

\paragraph{The symmetric case}

The estimate in \nameref{thm:galerkin:cea} is weaker than the corresponding estimate \eqref{eq:bvp1d:cea} for the model problem in \cref{chap:overview}. This is due to the symmetry of the bilinear form in the latter case, which allows characterizing solutions of \eqref{eq:galerkin:variational} as minimizers of a functional.
\begin{theorem}\label{thm:galerkin:symmetric}
    If $a$ is coercive and symmetric, $u\in V$ satisfies \eqref{eq:galerkin:variational} if and only if
    $u$ is the minimizer of
    \begin{equation*}
        J(v) := \tfrac12 a(v,v)-F(v)
    \end{equation*}
    over all $v\in V$.
\end{theorem}
\begin{proof}
    For any $u,v\in V$ and $t\in \R$,
    \begin{equation*}
        J(u+tv) = J(u) + t (a(u,v)-F(v)) + \frac{t^2}{2} a(v,v)
    \end{equation*}
    due to the bilinearity and symmetry of $a$. Assume now that $u$ satisfies $a(u,v)-F(v) = 0$ for all $v\in V$. Then setting $t=1$, we deduce that for all $v\neq 0$,
    \begin{equation*}
        J(u+v) = J(u) + \tfrac12 a(v,v) \geq J(u) + \frac{c_1}{2} \norm*{V}{v}^2 > J(u).
    \end{equation*}
    Hence, $u$ is the unique minimizer of $J$. Conversely, if $u$ is the (unique) minimizer of $J$, every directional derivative of $J$ at $u$ must vanish, which implies that
    \begin{equation*}
        0 = \frac{d}{dt} J(u+tv)|_{t=0} = a(u,v) - F(v)
    \end{equation*}
    for all $v\in V$.
\end{proof}

Together with coercivity and continuity, the symmetry of $a$ implies that $a(u,v)$ is an inner product on $V$ that induces an \emph{energy norm} $\norm*{a}{u} := a(u,u)^\frac12$. (In fact, in many applications, the functional $J$ represents an energy which is minimized in a physical system. For example in continuum mechanics, $\tfrac12\norm*{a}{u}^2 = \tfrac12 a(u,u)$ represents the elastic deformation energy of a body, and $-F(v)$ represents its potential energy under external load.)

Arguing as in \cref{sec:overview:galerkin}, we see that the solution $u_h \in V_h$ of \eqref{eq:galerkin:galerkin} -- which is called \emph{Ritz--Galerkin approximation} in this context -- satisfies
\begin{equation*}
    \norm*{a}{u-u_h} = \min_{v_h\in V_h}\norm*{a}{u-v_h},
\end{equation*}
i.e., $u_h$ is the best approximation of $u$ in $V_h$ in the energy norm. Using the equivalence of norms, this implies that the infimum in \cref{thm:galerkin:cea} is attained for symmetric bilinear forms.
Equivalently, one can say that the error $u-u_h$ is orthogonal to $V_h$ in the inner product defined by $a$.

\bigskip

Often it is more useful to estimate the error in a weaker norm. This requires a \emph{duality argument}. Let $H$ be a Hilbert space with inner product $(\cdot,\cdot)_H$ and $V$ be a closed subspace satisfying the conditions of the \nameref{thm:weak:laxmilgram} theorem such that the embedding $V\hookrightarrow H$ is continuous (e.g., $V=\m{H1} \hookrightarrow \m{L2} = H$). Then we have the following estimate.
\begin{lemma}[Aubin--Nitsche lemma]\label{thm:galerkin:aubin}
    Let $u_h$ be the solution of \eqref{eq:galerkin:galerkin} for given $V_h\subset V$ and $u$ be the solution of \eqref{eq:galerkin:variational}.
    For any $g\in H$, let $\phi_g$ be the unique solution of the \emph{adjoint problem}
    \begin{equation}\label{eq:galerkin:adjoint}
        a(w,\phi_g) = (g,w)_H \quad \text{ for all } w \in V.
    \end{equation}
    Then there exists a $C>0$ such that
    \begin{equation*}
        \norm*{H}{u-u_h} \leq C \norm*{V}{u-u_h} \sup_{g\in H\setminus\{0\}}\left(\frac{1}{\norm*{H}{g}} \inf_{v_h\in V_h} \norm*{V}{\phi_g-v_h}\right).
    \end{equation*}
\end{lemma}
\begin{proof}
    We make use of the dual representation of the norm in any Hilbert space,
    \begin{equation}\label{eq:galerkin:dualnorm}
        \norm*{H}{w} = \sup_{g\in H}\frac{(g,w)_H}{\norm*{H}{g}}.
    \end{equation}

    Now, inserting $w=u-u_h$ in the adjoint problem, we obtain for any $v_h\in V_h$ using the Galerkin orthogonality and continuity of $a$ that
    \begin{equation*}
        \begin{aligned}
            (g,u-u_h)_H &= a(u-u_h,\phi_g) \\
            &=  a(u-u_h,\phi_g - v_h) \\
            &\leq  c_2\norm*{V}{u-u_h}\norm*{V}{\phi_g-v_h}.
        \end{aligned}
    \end{equation*}
    Inserting $w=u-u_h$ into \eqref{eq:galerkin:dualnorm}, we thus obtain
    \begin{equation*}
        \begin{aligned}
            \norm*{H}{u-u_h} &= \sup_{g\in H}\frac{(g,u-u_h)_H}{\norm*{H}{g}}\\
            &\leq  c_2\norm*{V}{u-u_h} \sup_{g\in H\setminus\{0\}}\frac{\norm*{V}{\phi_g-v_h}}{\norm*{H}{g}}
        \end{aligned}
    \end{equation*}
    for arbitrary $v_h\in V_h$, and taking the infimum over all $v_h$ yields the desired estimate.
\end{proof}
Note that the existence of a unique solution of the adjoint problem is an assumption here that needs to be verified.
If $a$ is symmetric, this is guaranteed by the \nameref{thm:weak:laxmilgram}.
Otherwise, both the original and the adjoint problem need to satisfy the conditions of the Lax--Milgram theorem (which is the case, e.g., for constant coefficients $b_j$).

\chapter{Finite element spaces}\label{chap:elements}

Finite element methods are a special case of Galerkin methods, where the finite-dimensional subspace consists of piecewise polynomials. To construct these subspaces, we proceed in two steps:
\begin{enumerate}
    \item We define a \emph{reference element} and study polynomial interpolation on this element.
    \item We use suitably transformed copies of the reference element to partition the given domain and discuss how to construct a global interpolant from local interpolants on each element.
\end{enumerate}
We then follow the same steps in proving interpolation error estimates for functions in Sobolev spaces.

\section{Construction of finite element spaces}

To allow a unified study of the zoo of finite elements proposed in the literature,\footnote{For a -- far from complete -- list of elements, see, e.g., \cite[Chapter 3]{Brenner:2008}, \cite[Section 2.2]{Ciarlet:2002}} we define a finite element in an abstract way.

\begin{defn}
    A \emph{finite element} is a triple  $(K,\calP,\calN)$ where
    \begin{enumerate}[(i)]
        \item $K\subset \R^n$ is a simply connected bounded open set with piecewise smooth boundary (the \emph{element domain}, or simply \emph{element} if there is no possibility of confusion);
        \item $\calP$ is a finite-dimensional space of functions defined on $K$ (the \emph{space of shape functions});
        \item $\calN = \{N_1,\dots,N_d\}$ is a basis of $\calP^*$ (the \emph{set of nodal variables} or \emph{degrees of freedom}).
    \end{enumerate}
\end{defn}
Here $\calP^*$ denotes the algebraic dual of $\calP$, i.e., the space of linear functionals on $\calP$.
As we will see, condition (iii) guarantees that the interpolation problem on $K$ using functions in $\calP$ -- and hence the Galerkin approximation -- is well-posed. The nodal variables will play the role of interpolation conditions. This is a somewhat backwards definition compared to our introduction in Chapter 1 (where we have directly specified a basis for the shape functions). However, it leads to an equivalent characterization that allows much greater freedom in defining finite elements. The connection is given in the next definition.
\begin{defn}
    Let $(K,\calP,\calN)$ be a finite element. A basis $\{\psi_1,\dots,\psi_d\}$ of $\calP$ is called \emph{dual basis} or \emph{nodal basis} to $\calN$ if $N_i(\psi_j) = \delta_{ij}$.
\end{defn}
For example, for the linear finite elements in one dimension, $K=(0,1)$, $\calP=P_1$ is the space of linear polynomials, and $\calN=\{N_1,N_2\}$ are the \emph{point evaluations} $\calN_1(v) = v(0)$, $\calN_2(v)=v(1)$ for every $v\in\calP$. The nodal basis is given by $\psi_1(x) = 1-x$ and $\psi_2(x) = x$.

Condition (iii) is the only one that is difficult to verify. The following lemma simplifies this task.
\begin{lemma}\label{thm:elements:basis}
    Let $\calP$ be a $d$-dimensional vector space and let $\{N_1,\dots,N_d\}$ be a subset of $\calP^*$. Then the following statements are equivalent:
    \begin{enumerate}[a)]
        \item $\{N_1,\dots,N_d\}$ is a basis of $\calP^*$;
        \item if  $v\in\calP$ satisfies $N_i(v)=0$ for all $1\leq i \leq d$, then $v=0$.
    \end{enumerate}
\end{lemma}
\begin{proof}
    Let $\{\psi_1,\dots,\psi_d\}$ be a basis of $\calP$. Then $\{N_1,\dots,N_d\}$ is a basis of $\calP^*$ if and only if for any $L\in\calP^*$, there exist (unique) $\alpha_i$, $1\leq i\leq d$, such that
    \begin{equation*}
        L = \sum_{j=1}^d \alpha_j N_j.
    \end{equation*}
    Using the basis of $\calP$, this is equivalent to $L(\psi_i)= \sum_{j=1}^d\alpha_j N_j(\psi_i)$ for all $1\leq i \leq d$. Define the (square) \emph{Vandermonde matrix} $\mathbf{B}=(N_j(\psi_i))_{i,j=1}^d$ and the vectors
    \begin{equation*}
        \mathbf{L}= (L(\psi_1),\dots,L(\psi_d))^T,\qquad \mathbf{a} = (\alpha_1,\dots,\alpha_d)^T.
    \end{equation*}
    Then (a) is equivalent to $\mathbf{Ba}=\mathbf{L}$ being uniquely solvable, i.e., $\mathbf{B}$ being invertible.

    On the other hand, given any $v\in\calP$, we can write $v =\sum_{j=1}^d\beta_j \psi_j$. The condition (b) can be expressed as
    \begin{equation*}
        \sum_{j=1}^n \beta_j N_i(\psi_j) = N_i(v)=0\quad\text{ for all } 1\leq i\leq d
    \end{equation*}
    implying $v=0$, or, in matrix form, that $\mathbf{B}^T\mathbf{b} = 0$ implies $0=\mathbf{b} := (\beta_1,\dots,\beta_d)^T$, i.e., that $B^T$ is injective.  But this too is equivalent to the fact that $\mathbf{B}$ is invertible (since any square matrix is invertible if and only if it is surjective).
\end{proof}
Note that (b) in particular implies that the interpolation problem using functions in $\calP$ with interpolation conditions $\calN$ is uniquely solvable.
To construct a finite element, one usually proceeds in the following way:
\begin{enumerate}
    \item choose an element domain $K$ (e.g., a triangle),
    \item choose a polynomial space $\calP$ of a given degree $k$ (e.g., linear functions),
    \item choose $d$ degrees of freedom $\calN = \{N_1,\dots,N_d\}$, where $d$ is the dimension of $\calP$, such that the corresponding interpolation problem has a unique solution,
    \item compute the nodal basis of $\calP$ with respect to $\calN$.
\end{enumerate}
The last step amounts to solving for $1\leq j\leq d$ the concrete interpolation problems $N_i(\psi_j)=\delta_{ij}$, e.g., using the Vandermonde matrix.
A useful tool to verify the unique solvability of the interpolation problem for polynomials is the following lemma, which is a multidimensional form of polynomial division. Recall that for multivariate polynomials, the \emph{(total) degree} is the maximal sum of all occuring powers in a term (e.g., $p(x) = x_1 x_2^2$ has degree $3$). It is convenient to write such a polynomial $p$ of degree $k$ on $\R^n$ as $p(x) = \sum_{|\alpha|\leq k} c_\alpha x^\alpha$ using a \emph{multi-index} $\alpha\in \N_0^{n-1}$ with the convention that $x^\alpha:=x_1^{\alpha_1}\cdot x_n^{\alpha_n}$ and $|\alpha| := \sum_{i=1}^n\alpha_i$.
\begin{lemma}\label{thm:elements:factor}
    Let $L\neq 0$ be a linear-affine functional on $\R^n$ and $P$ be a polynomial of total degree $d\geq 1$ with $P(x) = 0$ for all $x$ with $L(x)=0$. Then there exists a polynomial $Q$ of total degree $d-1$ such that $P=LQ$.
\end{lemma}
\begin{proof}
    First, we note that affine transformations map the space of polynomials of degree $d$ to itself. Thus, we can assume without loss of generality that $P$ vanishes on the hyperplane orthogonal to the $x_n$ axis, i.e. $L(x)=x_n$ and  $P(\hat x,0) = 0$, where $\hat x = (x_1,\dots,x_{n-1})$. Since the degree of $P$ is $d$, we can write
    \begin{equation*}
        P(\hat x,x_n) = \sum_{j=0}^d \left[ \sum_{|\alpha|\leq d-j} c_{\alpha,j}\hat x^\alpha\right] x^j_n.
    \end{equation*}
    For $x_n = 0$, this implies that
    \begin{equation*}
        0=P(\hat x,0) = \sum_{|\alpha|\leq d} c_{\alpha,0}\hat x^\alpha,
    \end{equation*}
    and therefore $c_{\alpha,0}=0$ for all $|\alpha|\leq d$. Hence,
    \begin{equation*}
        \begin{aligned}
            P(\hat x,x_n) &= \sum_{j=1}^d\left[ \sum_{|\alpha|\leq d-j} c_{\alpha,j}\hat x^\alpha\right] x^j_n\\
            &=x_n \sum_{j=1}^d \sum_{|\alpha|\leq d-j} c_{\alpha,j}\hat x^\alpha x_n^{j-1}\\
            &=: x_nQ = LQ,
        \end{aligned}
    \end{equation*}
    where $Q$ is of degree $d-1$.
\end{proof}

\section{Examples of finite elements}

We restrict ourselves to the case $n=2$ (higher dimensions being similar) and the most common examples.

\paragraph{Triangular elements}
Let $K$ be a triangle and
\begin{equation*}\textstyle
    P_k=\setof{\sum_{|\alpha|\leq k}c_\alpha x^\alpha}{c_\alpha\in\R}
\end{equation*}
denote the space of all bivariate polynomials of total degree less than or equal $k$, e.g., $P_2=\mathrm{span}\,\{1,x_1,x_2,x_1^2,x_2^2,x_1x_2\}$. It is straightforward to verify that $P_k$ (and hence $P_k^*$) is a vector space of dimension $\half(k+1)(k+2)$. We consider two types of interpolation conditions: function values (\emph{Lagrange interpolation}) and gradient values (\emph{Hermite interpolation}). The following examples define valid finite elements. Note that the argumentation is essentially the same as for the well-posedness of the corresponding one-dimensional polynomial interpolation problems.
\begin{figure}
    \centering
    \begin{subfigure}{0.32\linewidth}
        \centering
        \begin{tikzpicture}[scale=1.8]
            \coordinate (z1) at (0,0);
            \coordinate (z2) at (2,0);
            \coordinate (z3) at (1,1.73);
            \draw (z1) -- node[below]{$L_3$} (z2);
            \draw (z2) -- node[right]{$L_1$} (z3);
            \draw (z3) -- node[left]{$L_2$} (z1);
            \node[below=2pt] at (z1) {$z_1$}; \node[below=2pt] at (z2) {$z_2$}; \node[above=2pt] at (z3) {$z_3$};
            \foreach \x in {1,2,3} {\fill (z\x) circle (2pt);}
            \end{tikzpicture}
            \caption{linear Lagrange element\label{fig:elements:tri:lin}}
        \end{subfigure}
        \hfill
        \begin{subfigure}{0.32\linewidth}
            \centering
            \begin{tikzpicture}[scale=1.8]
                \coordinate (z1) at (0,0);
                \coordinate (z2) at (2,0);
                \coordinate (z3) at (1,1.73);
                \coordinate (z4) at ($1/2*(z2)+1/2*(z3)$);
                \coordinate (z5) at ($1/2*(z1)+1/2*(z3)$);
                \coordinate (z6) at ($1/2*(z2)+1/2*(z1)$);
                \draw (z1) -- (z2);
                \draw (z2) -- (z3);
                \draw (z3) -- (z1);
                \node[below=2pt] at (z1) {$z_1$}; \node[below=2pt] at (z2) {$z_2$}; \node[above=2pt] at (z3) {$z_3$};
                \node[right=2pt] at (z4) {$z_4$}; \node[left=2pt] at (z5) {$z_5$}; \node[below=2pt] at (z6) {$z_6$};
                \foreach \x in {1,...,6} {\fill (z\x) circle (2pt);}
                \end{tikzpicture}
                \caption{quadratic Lagrange element\label{fig:elements:tri:quad}}
            \end{subfigure}
            \hfill
            \begin{subfigure}{0.32\linewidth}
                \centering
                \begin{tikzpicture}[scale=1.8]
                    \coordinate (z1) at (0,0);
                    \coordinate (z2) at (2,0);
                    \coordinate (z3) at (1,1.73);
                    \coordinate (z4) at ($1/3*(z1)+1/3*(z2)+1/3*(z3)$);
                    \draw (z1) -- (z2);
                    \draw (z2) -- (z3);
                    \draw (z3) -- (z1);
                    \node[below=5pt] at (z1) {$z_1$}; \node[below=5pt] at (z2) {$z_2$}; \node[above=5pt] at (z3) {$z_3$};
                    \node[above] at (z4) {$z_4$};
                    \foreach \x in {1,...,4} {\fill (z\x) circle (2pt);}
                        \foreach \x in {1,...,3} {\draw (z\x) circle (4pt);}
                        \end{tikzpicture}
                        \caption{cubic Hermite element\label{fig:elements:tri:cub}}
                    \end{subfigure}
                    \caption{Triangular finite elements. Filled circles denote point evaluation, open circles gradient evaluations.}\label{fig:elements:tri}
                \end{figure}
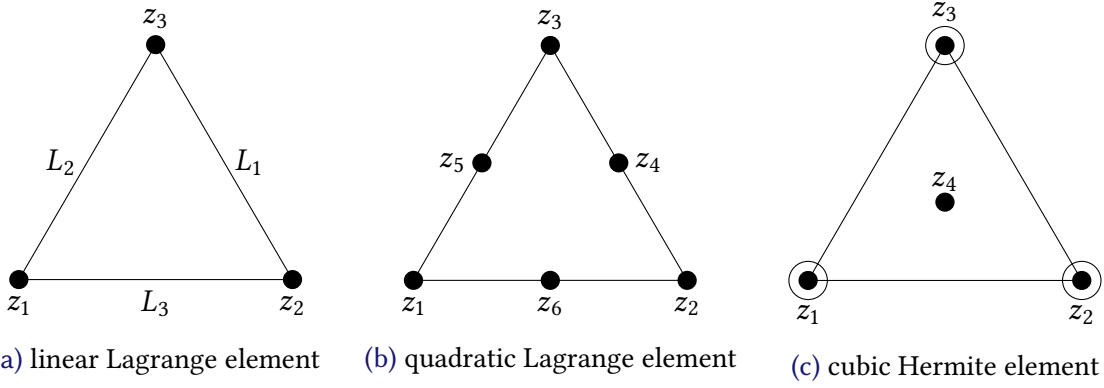
                \begin{itemize}
                    \item\emph{Linear Lagrange elements}: Let $k=1$ and take $\calP = P_1$ (hence the dimension of $\calP$ and $\calP^*$ is $3$) and $\calN = \{N_1,N_2,N_3\}$ with $N_i(v) = v(z_i)$, where $z_1,z_2,z_3$ are the vertices of $K$ (see \cref{fig:elements:tri:lin}). We need to show that condition (iii) holds, which we will do by way of \cref{thm:elements:basis}. Suppose that $v\in P_1$ satisfies $v(z_1)=v(z_2)=v(z_3)=0$. Since $v$ is linear, it must also vanish on each line connecting the vertices, which can be defined as the zero-sets of the (non-constant) linear functions $L_1,L_2,L_3$. Hence, by \cref{thm:elements:factor}, there exists a constant (i.e., polynomial of degree $0$) $c$ such that, e.g., $v=cL_1$. Now let $z_1$ be the vertex not on the edge defined by $L_1$. Then
                        \begin{equation*}
                            0 = v(z_1) = cL_1(z_1).
                        \end{equation*}
                        Since $L_1(z_1)\neq 0$ (otherwise the linear functional $L_1$ would be identically zero), this implies $c=0$ and thus $v=0$.

                    \item \emph{Quadratic Lagrange elements}: Let $k=2$ and take $\calP =  P_2$ (hence the dimension of $\calP$ and $\calP^*$ is $6$). Set $\calN = \{N_1,N_2,N_3,N_4,N_5,N_6\}$ with $N_i(v) = v(z_i)$, where $z_1,z_2,z_3$ are again the vertices of $K$ and $z_4,z_5,z_6$ are the midpoints of the edges described by the linear functions $L_1,L_2,L_3$, respectively (see \cref{fig:elements:tri:quad}). To show that condition (iii) holds, we argue as above. Let $v\in P_2$ vanish at $z_i$, $1\leq i \leq 6$. On each edge, $v$ is a quadratic function that vanishes at three points (say, $z_2, z_3, z_4$) and thus must be identically zero. If $L_1$ is the functional vanishing on the edge containing $z_2, z_3, z_4$, then by \cref{thm:elements:factor}, there exists a linear polynomial $Q_1$ such that $v=L_1Q_1$. Now consider one of the remaining edges with corresponding functional, e.g., $L_2$. Since $v(z_5) = v(z_6) = 0$ by assumption and $L_2$ cannot be zero there (otherwise it would be constant), we have that $Q_1(z_5) = Q_1(z_6)=0$, i.e., $Q_1$ is a linear polynomial  on this edge with two roots and hence vanishes. Applying \cref{thm:elements:factor} to $Q_1$, we thus obtain a constant $c$ such that $v=L_1Q_1=cL_1L_2$. Taking the midpoint of the remaining edge, $z_6$, we have
                        \begin{equation*}
                            0 = v(z_6) = cL_1(z_6)L_2(z_6),
                        \end{equation*}
                        and since neither $L_1$ nor $L_2$ are zero in $z_6$, we deduce $c=0$ and hence $v=0$.

                    \item\emph{Cubic Hermite elements}: Let $k=3$ and take $\calP =  P_3$ (hence the dimension of $\calP$ and $\calP^*$ is $10$). Instead of taking $\calN$ as function evaluations at ten suitable points, we take $N_i$, $1\leq i\leq 4$ as the point evaluation at the vertices $z_1,z_2,z_3$ and the barycenter $z_4 = \frac13(z_1+z_2+z_3)$ (see \cref{fig:elements:tri:cub}) and take the remaining nodal variables as gradient evaluations:
                        \begin{equation*}
                            N_{i+4}(v) = \partial_1 v(z_i), \qquad N_{i+7} = \partial_2 v(z_i),\quad 1\leq i \leq 3.
                        \end{equation*}
                        Now we again consider $v\in P_3$ with $N_i(v)=0$ for all $1\leq i\leq 10$. On each edge, $v$ is a cubic polynomial with double roots at each vertex, and hence must vanish. By considering successively each edge, we find that $v=cL_1L_2L_3$ which implies that
                        \begin{equation*}
                            0=v(z_4) =cL_1(z_4)L_2(z_4)L_3(z_4)
                        \end{equation*}
                        and hence $c=0$ since the barycenter $z_4$ lies on neither of the edges. Therefore, $v=0$.
                \end{itemize}
                The interpolation points $z_i$ are called \emph{nodes} (not to be confused with the \emph{vertices} defining the element domain).
                Both types of elements can be defined for arbitrary degree $k$.
                It should be clear from the above that our definition of finite elements gives us a blueprint for constructing elements with desired properties. This should be contrasted with, e.g., the choice of finite difference stencils.

                \paragraph{Rectangular elements}
                For rectangular elements, we can follow a tensor-product approach. We consider the vector space
                \begin{equation*}
                    Q_k = \setof{\sum_j c_jp_j(x_1)q_j(x_2)}{c_j\in\R, p_j,q_j\in P_k}
                \end{equation*}
                of products of univariate polynomials of degree up to $k$, which has dimension $(k+1)^2$ (e.g., $Q_2 = \mathrm{span}\,\{1,x_1,x_2,x_1x_2,x_1^2x_2,x_1x_2^2,x_1^2,x_2^2\}$). By the same arguments as in the triangular case, we can show that the following examples are finite elements:
                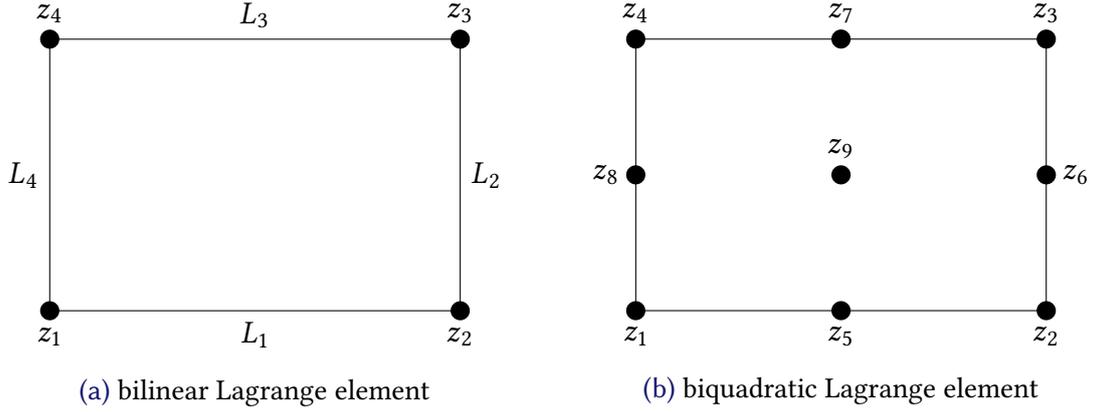
\begin{figure}
                    \centering
                    \begin{subfigure}{0.495\textwidth}
                        \centering
                        \begin{tikzpicture}[scale=1.8]
                            \coordinate (z1) at (0,0);
                            \coordinate (z2) at (3,0);
                            \coordinate (z3) at (3,2);
                            \coordinate (z4) at (0,2);
                            \draw (z1) -- node[below]{$L_1$} (z2);
                            \draw (z2) -- node[right]{$L_2$} (z3);
                            \draw (z3) -- node[above]{$L_3$} (z4);
                            \draw (z4) -- node[left]{$L_4$} (z1);
                            \node[below=2pt] at (z1) {$z_1$}; \node[below=2pt] at (z2) {$z_2$}; \node[above=2pt] at (z3) {$z_3$}; \node[above=2pt] at (z4) {$z_4$};
                            \foreach \x in {1,...,4} {\fill (z\x) circle (2pt);}
                            \end{tikzpicture}
                            \caption{bilinear Lagrange element\label{fig:elements:quad:lin}}
                        \end{subfigure}
                        \hfill
                        \begin{subfigure}{0.495\textwidth}
                            \centering
                            \begin{tikzpicture}[scale=1.8]
                                \coordinate (z1) at (0,0);
                                \coordinate (z2) at (3,0);
                                \coordinate (z3) at (3,2);
                                \coordinate (z4) at (0,2);
                                \coordinate (z5) at ($1/2*(z1)+1/2*(z2)$);
                                \coordinate (z6) at ($1/2*(z2)+1/2*(z3)$);
                                \coordinate (z7) at ($1/2*(z3)+1/2*(z4)$);
                                \coordinate (z8) at ($1/2*(z4)+1/2*(z1)$);
                                \coordinate (z9) at ($1/4*(z1)+1/4*(z2)+1/4*(z3)+1/4*(z4)$);
                                \draw (z1) --  (z2);
                                \draw (z2) --  (z3);
                                \draw (z3) --  (z4);
                                \draw (z4) --  (z1);

                                \node[below=2pt] at (z1) {$z_1$}; \node[below=2pt] at (z2) {$z_2$}; \node[above=2pt] at (z3) {$z_3$}; \node[above=2pt] at (z4) {$z_4$};
                                \node[below=2pt] at (z5) {$z_5$}; \node[right=2pt] at (z6) {$z_6$}; \node[above=2pt] at (z7) {$z_7$}; \node[left=2pt] at (z8) {$z_8$};
                                \node[above=2pt] at (z9) {$z_9$};

                                \foreach \x in {1,...,9} {\fill (z\x) circle (2pt);}
                                \end{tikzpicture}
                                \caption{biquadratic Lagrange element\label{fig:elements:quad:quad}}
                            \end{subfigure}
                            \caption{Rectangular finite elements. Filled circles denote point evaluation.}\label{fig:elements:quad}
                        \end{figure}
                        \begin{itemize}
                            \item\emph{Bilinear Lagrange elements}: Let $k=1$ and take $\calP =  Q_1$  (hence the dimension of $\calP$ and $\calP^*$ is $4$) and  $\calN = \{N_1,N_2,N_3,N_4\}$ with $N_i(v) = v(z_i)$, where $z_1,z_2,z_3,z_4$ are the vertices of $K$ (see \cref{fig:elements:quad:lin}).

                            \item\emph{Biquadratic Lagrange elements}: Let $k=2$ and take $\calP =  Q_2$  (hence the dimension of $\calP$ and $\calP^*$ is $9$) and  $\calN = \{N_1,\dots,N_9\}$ with $N_i(v) = v(z_i)$, where $z_1,z_2,z_3,z_4$ are the vertices of $K$, $z_5,z_6,z_7,z_8$ are the edge midpoints and $z_9$ is the centroid of $K$ (see \cref{fig:elements:quad:quad}).
                        \end{itemize}

                        The above construction is easy to generalize for arbitrary $k$ and $n$: Let $t_1,\dots,t_{k+1}$ be distinct points on (say) $[0,1]$ with $t_1=0$ and $t_{k+1}=1$. Then the nodes $z_1,\dots,z_d$ for the rectangular Lagrange element on $K=[0,1]^n$ are given by the \emph{tensor product}
                        \begin{equation*}
                            \setof{(t_{i_1},\dots,t_{i_n})}{i_j=1,\dots,k+1 \text{ for }j=1,\dots,n}.
                        \end{equation*}
                        This straightforward construction is the main advantage of rectangular elements; on the other hand, triangular elements give more flexibility for handling complicated domains.

                        \section{The interpolant}

                        We wish to estimate the error of the best approximation of a function in a finite element space. An upper bound for this approximation is given by stitching together interpolating polynomials on each element.
                        \begin{defn}
                            Let $(K,\calP,\calN)$ be a finite element and let $\{\psi_1,\dots,\psi_d\}$ be the corresponding nodal basis of $\calP$. For a given function $v$ such that $N_i(v)$ is defined for all $1\leq i\leq d$, the \emph{local interpolant} of $v$ is defined as
                            \begin{equation*}
                                \calI_Kv = \sum_{i=1}^d N_i(v)\psi_i.
                            \end{equation*}
                        \end{defn}
                        The local interpolant can be explicitly constructed once the nodal basis is known. This can be simplified significantly if the reference element domain is chosen as, e.g., the unit simplex.

                        Useful properties of the local interpolant are given next.
                        \begin{lemma}\label{thm:elements:localinterp}
                            Let $(K,\calP,\calN)$ be a finite element and $\calI_K$ the local interpolant. Then
                            \begin{enumerate}
                                \item the mapping $v\mapsto \calI_K$ is linear;
                                \item $N_i(\calI_Kv)) = N_i(v)$, $1\leq i \leq d$;
                                \item $\calI_K(v) = v$ for all $v\in\calP$, i.e., $\calI_K$ is a projection.
                            \end{enumerate}
                        \end{lemma}
                        \begin{proof}
                            The claim (i) follows directly from the linearity of the $N_i$. For (ii), we use the definition of $\calI_K$ and $\psi_i$ to obtain
                            \begin{equation*}
                                \begin{aligned}
                                    N_i(\calI_K v) &= N_i\left(\sum_{j=1}^dN_j(v)\psi_j\right)
                                    = \sum_{j=1}^d N_j(v)N_i(\psi_j)
                                    =\sum_{j=1}^d N_j(v)\delta_{ij}\\ &=  N_i(v)
                                \end{aligned}
                            \end{equation*}
                            for all $1\le i\leq d$ and arbitrary $v$. This implies that $N_i(v-\calI_Kv)=0$ for all $1\leq i\leq d$, and hence by \cref{thm:elements:basis} that $\calI_Kv = v$ and hence (iii) holds.
                        \end{proof}

                        We now use the local interpolant on each element to define a global interpolant on a union of elements.
                        \begin{defn}
                            A \emph{subdivision} of a bounded open set $\Omega\subset\R^n$ is a finite collection $\calT$ of open sets $K_i$ such that
                            \begin{enumerate}[(i)]
                                \item $K_i\cap K_j  = \emptyset $ if $i\neq j$;
                                \item $\bigcup_i \bar K_i = \bar\Omega$.
                            \end{enumerate}
                        \end{defn}
                        \begin{defn}
                            Let $\calT$ be a subdivision of $\Omega$ such that for each $K_i$ there is a finite element $(K_i,\calP_i,\calN_i)$ with local interpolant $\calI_{K_i}$, and let $m$ be the order of the highest partial derivative appearing in any nodal variable. Then the \emph{global interpolant} $\calI_\calT v$ of $v\in{C}^m(\bar\Omega)$ on $\calT$ is defined by
                            \begin{equation*}
                                (\calI_\calT v)|_{K_i} = \calI_{K_i} v \quad \text{ for all } K_i\in\calT.
                            \end{equation*}
                        \end{defn}

                        To obtain some regularity of the global interpolant, we need additional assumptions on the subdivision. Roughly speaking, where two elements meet, the corresponding nodal variables have to match as well. For triangular elements, this can be expressed concisely.
                        \begin{defn}
                            A \emph{triangulation}  of a bounded open set $\Omega\subset\R^2$ is a subdivision $\calT$ of $\Omega$ such that
                            \begin{enumerate}[(i)]
                                \item every $K_i\in\calT$ is a triangle;
                                \item no vertex of any triangle lies on an edge of another triangle (i.e., no \emph{hanging nodes}).
                            \end{enumerate}
                        \end{defn}
                        Similar conditions can be given for $n\geq 3$ (tetrahedra, simplices), in which case one usually also speaks of triangulations. Note that this supposes that $\Omega$ is polyhedral itself. (For non-polyhedral domains, it is possible to use curved elements near the boundary.)

                        \begin{defn}
                            A global interpolant $\calI_\calT$ has \emph{continuity order} $m$ (in short, \enquote{is ${C}^m$}) if $\calI_\calT v\in{C}^m(\bar\Omega)$ for all $v\in{C}^m(\bar\Omega)$ (for which the interpolation is well-defined). In this case, the space
                            \begin{equation*}
                                V_\calT = \setof{\calI_\calT v}{v\in{C}^m(\bar\Omega)}
                            \end{equation*}
                            is called a \emph{${C}^m$ finite element space}.
                        \end{defn}

                        In particular, to obtain global continuity of the interpolant, we need to make sure that the local interpolants coincide where two element domains meet. This requires that the corresponding nodal variables are compatible. For Lagrange and Hermite elements, where each nodal variable is taken as the evaluation of a function or its derivative at a point $z_i$, this reduces to a geometric condition on the placement of nodes on edges.
                        \begin{figure}
                            \centering
                            \begin{subfigure}{0.33\textwidth}
                                \centering
                                \begin{tikzpicture}[scale=1.8]
                                    \coordinate (z1) at (0,0);
                                    \coordinate (z2) at (2,0);
                                    \coordinate (z3) at (1,1.73);
                                    \coordinate (z4) at ($1/2*(z2)+1/2*(z3)$);
                                    \coordinate (z5) at ($1/2*(z1)+1/2*(z3)$);
                                    \coordinate (z6) at ($1/2*(z2)+1/2*(z1)$);
                                    \draw (z1) -- (z2) node(L3) {};
                                    \draw (z2) -- (z3) node(L1) {};
                                    \draw (z3) -- (z1);
                                    \node[below=8pt] at (z1) {$z_1$}; \node[below=8pt] at (z2) {$z_2$}; \node[above=8pt] at (z3) {$z_3$};
                                    \node[left=2pt] at (z4) {$z_4$}; \node[right=2pt] at (z5) {$z_5$}; \node[above=2pt] at (z6) {$z_6$};
                                    \foreach \x in {1,...,3} {
                                        \fill (z\x) circle (2pt);
                                        \draw (z\x) circle (4pt);
                                        \draw (z\x) circle (6pt);
                                    }
                                    \draw [->] (z4) -- ($ (z4)!0.25!270:(z3) $);
                                    \draw [->] (z5) -- ($ (z5)!0.25!270:(z1) $);
                                    \draw [->] (z6) -- ($ (z6)!0.25!270:(z2) $);
                                \end{tikzpicture}
                                \caption{Argyris triangle\label{fig:elements:argyris}}
                            \end{subfigure}
                            \hfill
                            \begin{subfigure}{0.66\textwidth}
                                \centering
                                \begin{tikzpicture}[scale=1.8]
                                    \coordinate (z1) at (0,0);
                                    \coordinate (z2) at (3,0);
                                    \coordinate (z3) at (3,1.7);
                                    \coordinate (z4) at (0,1.7);
                                    \draw (z1) --  (z2);
                                    \draw (z2) --  (z3);
                                    \draw (z3) --  (z4);
                                    \draw (z4) --  (z1);
                                    \node[below=5pt] at (z1) {$z_1$}; \node[below=5pt] at (z2) {$z_2$}; \node[above=5pt] at (z3) {$z_3$}; \node[above=5pt] at (z4) {$z_4$};
                                    \foreach \x in {1,...,4} {%
                                        \fill (z\x) circle (2pt);
                                        \draw (z\x) circle (4pt);
                                        \draw [double,->] (z\x) -- ($(z\x)+(0.25,0.25)$);
                                    }
                                \end{tikzpicture}
                                \caption{Bogner--Fox--Schmit rectangle\label{fig:elements:fox}}
                            \end{subfigure}
                            \caption{${C}^1$ elements. Filled circles denote point evaluation, double circles evaluation of gradients up to total order $2$, and arrows evaluation of normal derivatives. The double arrow stands for evaluation of the second mixed derivative $\partial^2_{12}$.}
                        \end{figure}
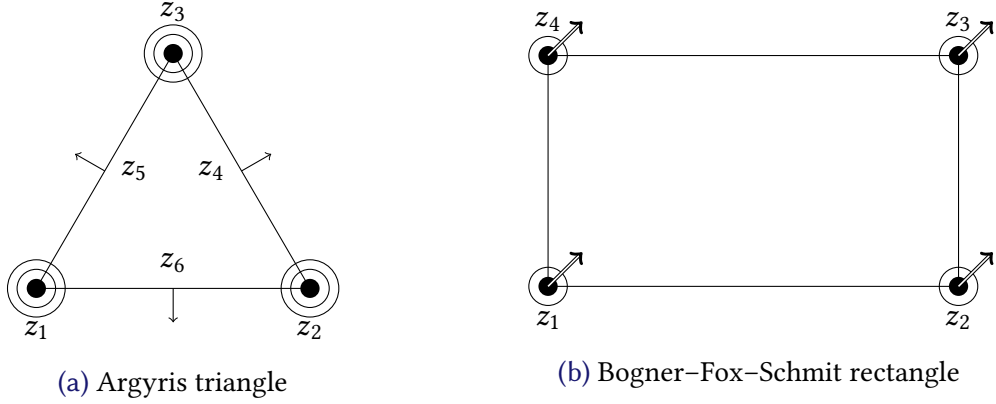
                        \begin{theorem}
                            The triangular Lagrange and Hermite elements of fixed degree are all ${C}^0$ elements (i.e., lead to $C^0$ finite element space). More precisely, given a triangulation $\calT$ of $\Omega$, it is possible to choose edge nodes for the corresponding elements $(K_i,\calP_i,\calN_i)$, $K_i\in\calT$, such that $\calI_\calT v\in{C}^0(\bar\Omega)$ for all $v\in{C}^m(\bar\Omega)$, where $m=0$ for Lagrange and $m=1$ for Hermite elements.
                        \end{theorem}
                        \begin{proof}
                            It suffices to show that the global interpolant is continuous across each edge. Let $K_1$ and $K_2$ be two triangles sharing an edge $e$. Assume that the nodes on this edge are placed symmetrically with respect to rotation (i.e., the placement of the nodes should \enquote{look the same} from $K_1$ and $K_2$), and that $\calP_1$ and $\calP_2$ consist of polynomials of degree $k$.

                            Let $v\in{C}^m(\bar\Omega)$ be given and set $w:=\calI_{K_1}v - \calI_{K_2}v$, where we extend both local interpolants as polynomials outside $K_1$ and $K_2$, respectively. Hence, $w$ is a polynomial of degree $k$ whose restriction $w|_e$ to $e$ is a one-dimensional polynomial having $k+1$ roots (counted by multiplicity). This implies that $w|_e=0$, and thus the interpolant is continuous across $e$.
                        \end{proof}
                        A similar argument shows that the bilinear and biquadratic Lagrange elements are ${C}^0$ as well. Examples of ${C}^1$ elements are the Argyris triangle (of degree $5$ and  $21$ nodal variables, including normal derivatives across edges at their midpoints, \cref{fig:elements:argyris}) and the Bogner--Fox--Schmit rectangle (a bicubic Hermite element of dimension $16$, \cref{fig:elements:fox}). It is one of the strengths of the abstract formulation described here that such exotic elements can be treated by the same tools as simple Lagrange elements.

                        In order to obtain global interpolation error estimates, we need uniform bounds on the local interpolation errors. For this, we need to be able to compare the local interpolation operators on different elements. This can be done with the following notion of equivalence of elements.
                        \begin{defn}
                            Let $(\hat K, \hat \calP,\hat \calN)$ be a finite element and $T:\R^n\to\R^n$ be an affine transformation, i.e., $T:\hat x\mapsto  A\hat x+b$ for $A\in\Rnn$ invertible and $b\in\R^n$. The finite element $(K,\calP,\calN)$ is called \emph{affine equivalent} to $(\hat K,\hat \calP,\hat\calN)$ if
                            \begin{enumerate}[(i)]
                                \item $K = \setof{A\hat x+b}{\hat x\in\hat K}$,
                                \item $\calP = \setof{\hat p \circ T^{-1}}{\hat p\in\hat\calP}$,
                                \item $\calN = \setof{N_i}{N_i(p) = \hat N_i(p\circ T) \text{ for all } p\in\calP}$.
                            \end{enumerate}
                            A triangulation $\calT$ consisting of affine equivalent elements is also called \emph{affine}.
                        \end{defn}
                        It is a straightforward exercise to show that the nodal bases of $\hat\calP$ and $\calP$ are related by $\hat \psi_i = \psi_i\circ T$.
                        Hence, if the nodal variables on edges are placed symmetrically, triangular Lagrange elements of the same order are affine equivalent, as are triangular Hermite elements. The same holds true for rectangular elements. Non-affine equivalent elements (such as \emph{isoparametric elements}\footnote{see, e.g., \cite[\S\,III.2]{Braess:2007}}) are useful in treating elements with curved boundaries (for non-polyhedral domains)

                        The advantage of this construction is that affine equivalent elements are also interpolation equivalent in the following sense.
                        \begin{lemma}\label{thm:elements:equivalent}
                            Let $(\hat K,\hat \calP,\hat \calN)$ and $(K,\calP,\calN)$ be two affine equivalent finite elements related by the transformation $T_K$. Then,
                            \begin{equation*}
                                \calI_{\hat K} (v\circ T_K) = (\calI_K v)\circ T_K.
                            \end{equation*}
                        \end{lemma}
                        \begin{proof}
                            Let $\hat\psi_i$ and $\psi_i$ be the nodal basis of $\hat\calP$ and $\calP$, respectively. By definition,
                            \begin{equation*}
                                \calI_{\hat K} (v\circ T_K) = \sum_{i=1}^d \hat N_i(v\circ T_K) \hat\psi_i  = \sum_{i=1}^d N_i(v) (\psi_i\circ T_K) = (\calI_K v)\circ T_K.
                                \qedhere
                            \end{equation*}
                        \end{proof}
                        Given a reference element $(\hat K,\hat \calP,\hat \calN)$, we can thus generate a triangulation $\calT$ using affine equivalent elements.

\chapter{Polynomial interpolation in Sobolev spaces}

We now come to the heart of the mathematical theory of finite element methods. As we have seen, the distance of the finite element solution to the true solution is determined by the distance to the best approximation by piecewise polynomials, which in turn is bounded by the distance to the corresponding interpolant. It thus remains to derive estimates for the (local and global) interpolation error.

\section{The Bramble--Hilbert lemma}

We start with the error for the local interpolant. The key for deriving error estimates is the \emph{Bramble--Hilbert lemma} \cite{Bramble:1970}. The derivation here follows the original functional-analytic arguments (by way of several results which may be of independent interest); there are also constructive approaches which allow more explicit computation of the constants.\footnote{see, e.g., \cite[\S\,3.2]{Suli}, \cite[Chapter~4]{Brenner:2008}}

The first lemma characterizes the kernel of differentiation operators.
\begin{lemma}\label{lem:bramble:kernel}
    If $v\in\m{Wkp}$ satisfies $D^\alpha v = 0$ for all $|\alpha| = k$, then $v$ is almost everywhere equal to a polynomial of degree $k-1$.
\end{lemma}
\begin{proof}
    If $D^\alpha v = 0$ holds for all $|\alpha| = k$, then also $D^\beta D^\alpha v = 0\in L^p(\Omega)$ for any multi-index $\beta$. Hence, $v\in\bigcap_{k=1}^\infty \m{Wkp}$. The Sobolev embedding \cref{thm:weak:sobolev} thus guarantees that $v\in{C}^k(\Omega)$ for all $k\in \N$. The claim then follows using classical (pointwise) arguments, e.g., by Taylor series expansion.
\end{proof}

The next result concerns moment interpolation of Sobolev functions on polynomials.
\begin{lemma}\label{lem:bramble:projection}
    For every $v\in\m{Wkp}$ there is a unique polynomial $q\in P_{k-1}$  such that
    \begin{equation}\label{eq:bramble:projection}
        \int_\Omega D^\alpha (v-q) \,dx = 0 \qquad\text{ for all } |\alpha| \leq k-1.
    \end{equation}
\end{lemma}
\begin{proof}
    Writing $q=\sum_{|\beta|\leq k-1}\xi_\beta x^\beta\in P_{k-1}$ as a linear combination of monomials, the condition \eqref{eq:bramble:projection} is equivalent to the linear system
    \begin{equation*}
        \sum_{|\beta|\leq k-1} \xi_\beta \int_\Omega D^\alpha x^\beta\, dx = \int_\Omega D^\alpha v\, dx, \qquad |\alpha|\leq k-1.
    \end{equation*}
    It thus remains to show that the quadratic matrix
    \begin{equation*}
        \mathbf{M} = \left(\int_\Omega D^\alpha x^\beta\, dx\right)_{|\alpha|,|\beta|\leq k-1}
    \end{equation*}
    is non-singular, which we do by showing injectivity. Consider $\mathbf{\xi}= (\xi_\beta)_{|\beta|\leq k-1}$ such that $\mathbf{M}\mathbf{\xi}=0$. This implies that the corresponding polynomial $q$ satisfies
    \begin{equation*}
        \int_\Omega  D^\alpha q\, dx = 0 \qquad\text{ for all } |\alpha|\leq k-1.
    \end{equation*}
    Inserting in turn for $\alpha$ all possible multi-indices in descending (lexicographical) order (such that $D^\alpha x^\beta$ is constant) yields $\xi_\beta = 0$ for all $|\beta|\leq k-1$. Thus, $\mathbf{M}\mathbf{\xi}=0$ implies $\mathbf{\xi}=0$, and therefore $\mathbf{M}$ is invertible.
\end{proof}

The last lemma is a generalization of Poincaré's inequality.
\begin{lemma}\label{lem:bramble:poincare}
    Let $v\in\m{Wkp}$ such that
    \begin{equation}\label{eq:bramble:poincare_mean}
        \int_\Omega D^\alpha v \,dx = 0 \qquad\text{ for all } |\alpha|\leq k-1.
    \end{equation}
    Then
    \begin{equation}\label{eq:bramble:poincare}
        \norm{Wkp}{v}\leq c_0 |v|_{\m{Wkp}},
    \end{equation}
    where the constant $c_0>0$ depends only on $\Omega$, $k$ and $p$.
\end{lemma}
\begin{proof}
    We argue by contradiction. Assume the claim does not hold. Then there exists a sequence $\{v_n\}_{n\in\N}\subset\m{Wkp}$ of functions satisfying \eqref{eq:bramble:poincare_mean} and
    \begin{equation}\label{eq:bramble:bound}
        |v_n|_{\m{Wkp}}\to 0  \qquad\text{ but }\qquad \norm{Wkp}{v_n} = 1 \quad\text{ as } n\to\infty.
    \end{equation}
    Since the embedding $\m{Wkp}\hookrightarrow{W}^{k-1,p}(\Omega)$ is compact by \cref{thm:weak:sobolev}, there exists a subsequence (also denoted by $\{v_n\}_{n\in\N}$) converging in ${W}^{k-1,p}(\Omega)
    $ to a $v\in{W}^{k-1,p}(\Omega)
    $, i.e.,
    \begin{equation}\label{eq:bramble:poincare1}
        \norm*{{W}^{k-1,p}(\Omega)}{v-v_n} \to 0 \quad\text{ as } n\to \infty.
    \end{equation}
    Since in addition $|v_n|_{\m{Wkp}}\to 0$ by assumption \eqref{eq:bramble:bound}, $\{v_n\}_{n\in\N}$ is a Cauchy sequence in $\m{Wkp}$ as well and thus converges in $\m{Wkp}$ to a $\tilde v\in\m{Wkp}$ which must satisfy $\tilde v = v$ (otherwise we would have a contradiction to \eqref{eq:bramble:poincare1}). By continuity, we then obtain that $|v|_{\m{Wkp}}=0$, and \cref{lem:bramble:kernel} yields that $v\in P_{k-1}$. Furthermore, $v$ satisfies
    \begin{equation*}
        \int_\Omega D^\alpha v \,dx = \lim_{n\to\infty}  \int_\Omega D^\alpha v_n \,dx = 0\quad \text{ for all } |\alpha|\leq k-1
    \end{equation*}
    by assumption \eqref{eq:bramble:poincare_mean}, which as in the proof of  \cref{lem:bramble:projection} implies that $v=0$. But this is a contradiction to
    \begin{equation*}
        \norm{Wkp}{v} = \lim_{n\to\infty}\norm{Wkp}{v_n} = 1.
        \qedhere
    \end{equation*}
\end{proof}

We are now in a position to prove our central result.
\begin{theorem}[Bramble--Hilbert lemma]\label{thm:bramble}
    Let $F:\m{Wkp}\to\R$ satisfy
    \begin{enumerate}[(i)]
        \item $|F(v)|\leq c_1 \norm{Wkp}{v}$ for all $v\in \m{Wkp}$ (boundedness),
        \item $|F(u+v)|\leq c_2 (|F(u)|+|F(v)|)$ for all $u,v\in \m{Wkp}$ (sublinearity),
        \item $F(q) = 0$ for all $q\in P_{k-1}$ (annihilation).
    \end{enumerate}
    Then there exists a constant $c>0$ such that for all $v\in\m{Wkp}$,
    \begin{equation*}
        |F(v)|\leq c|v|_{\m{Wkp}}.
    \end{equation*}
\end{theorem}
\begin{proof}
    For arbitrary $v\in\m{Wkp}$ and $q\in P_{k-1}$, we have
    \begin{equation*}
        |F(v)| = |F(v-q+q)| \leq c_2 (|F(v-q)|+|F(q)|) \leq c_1c_2 \norm{Wkp}{v-q}.
    \end{equation*}
    Given $v$, we now choose $q\in P_{k-1}$ as the polynomial from \cref{lem:bramble:projection} and apply \cref{lem:bramble:poincare} to $v-q\in\m{Wkp}$ to obtain
    \begin{equation*}
        \norm{Wkp}{v-q} \leq c_0 |v-q|_{\m{Wkp}} = c_0 |v|_{\m{Wkp}},
    \end{equation*}
    where $c_0$ is the constant appearing in \eqref{eq:bramble:poincare} and we have used that $D^\alpha q=0$ for $q\in P_{k-1}$ and all $|\alpha|=k$. This proves the claim with $c:=c_0c_1c_2$.
\end{proof}

\section{Interpolation error estimates}

We wish to apply the Bramble--Hilbert lemma to the interpolation error. We start with the error on the reference element.
\begin{theorem}\label{thm:interp:reference}
    Let $(K,\calP,\calN)$ be a finite element with $P_{k-1}\subset \calP$ for some $k\geq 1$ and all $N\in\calN$ bounded on $W^{k,p}(K)$, $1\leq p\leq \infty$. Then for any $v\in{W}^{k,p}(K)$,
    \begin{equation}\label{eq:interp:reference}
        |v-\calI_K v|_{{W}^{l,p}(K)} \leq c|v|_{{W}^{k,p}(K)} \quad\text{for all}\quad 0\leq l\leq k
    \end{equation}
    where the constant $c>0$ depends only on $n,k,p,l$ and $(K,\calP,\calN)$.
\end{theorem}
\begin{proof}
    It is straightforward to verify that $F:v\mapsto |v-\calI_K v|_{{W}^{l,p}(K)}$ defines a sublinear functional on ${W}^{k,p}(K)$ for all $l\leq k$. Let $\psi_1,\dots,\psi_d$ be the nodal basis of $\calP$ to $\calN$. Since the $N_i$ in $\calN$ are bounded on ${W}^{k,p}(K)$, we have that
    \begin{equation*}
        \begin{aligned}
            |F(v)| &\leq |v|_{{W}^{l,p}(K)} +|\calI_K v|_{{W}^{l,p}(K)} \\
            &\leq \norm*{{W}^{k,p}(K)}{v} +\sum_{i=1}^d|N_i(v)| |\psi_i|_{{W}^{l,p}(K)}\\
            &\leq \norm*{{W}^{k,p}(K)}{v} +\sum_{i=1}^d C_i\norm*{{W}^{k,p}(K)}{v} |\psi_i|_{{W}^{l,p}(K)}\\
            &\leq (1 +  C \max_{1\leq i\leq d} |\psi_i|_{{W}^{l,p}(K)}) \norm*{{W}^{k,p}(K)}{v}
        \end{aligned}
    \end{equation*}
    and hence that $F$ is bounded. In addition, $\calI_K q =q $ for all $q\in\calP$ and therefore $F(q)=0$. We can now apply the Bramble--Hilbert lemma to $F$, which proves the claim.
\end{proof}

To estimate the interpolation error on an arbitrary finite element $(K,\calP,\calN)$, we assume that it is generated by the affine transformation
\begin{equation}\label{eq:interp:affine}
    T_K: \hat K \to K,\qquad \hat x\mapsto A_K\hat x+b_K
\end{equation}
from the reference element $(\hat K,\hat\calP,\hat\calN)$, i.e., $\hat v:=v\circ T_K$ is the function $v$ on $K$ expressed in local coordinates on $\hat K$. We then need to consider how the estimate \eqref{eq:interp:reference} transforms under $T_K$. For this, we recall that for sufficiently smooth $v$, the chain rule for weak derivatives is given by
\begin{equation}
    \frac{\partial \hat v}{\partial{\hat x_i}} = \sum_{j=1}^n \frac{\partial v}{\partial x_j}\frac{\partial x_j}{\partial \hat x_i}
    = \sum_{j=1}^n (A_K)_{ij} \frac{\partial v}{\partial x_j},
\end{equation}
and the transformation rule for integrals by
\begin{equation}\label{eq:interp:transform}
    \int_{T_K(\hat K)} v\,dx = \int_{\hat K}(v\circ T_K)|\det (A_K)| \,d\hat x.
\end{equation}

\begin{lemma}\label{lem:interp:transform}
    Let $k\geq 0$ and $1\leq p\leq \infty$. There exists $c>0$ such that for all $K$ and $v\in{W}^{k,p}(K)$, the function $\hat v = v \circ T_K$ satisfies
    \begin{align}
        |\hat v|_{{W}^{k,p}(\hat K)}&\leq c \norm*{}{A_K}^k |\det(A_K)|^{-\frac1p} |v|_{{W}^{k,p}(K)},\label{eq:interp:transform1}\\
        |v|_{{W}^{k,p}(K)}&\leq c \norm*{}{A_K^{-1}}^k |\det(A_K)|^{\frac1p} |\hat v|_{{W}^{k,p}(\hat K)}.\label{eq:interp:transform2}
    \end{align}
\end{lemma}
\begin{proof}
    First, we have by \cref{thm:weak:transform} that $\hat v \in {W}^{k,p}(\hat K)$. Let now $\alpha$ be a multi-index with $|\alpha|=k$, and let $\hat D^\alpha$ denote the corresponding weak derivative with respect to $\hat x$. Applying the chain and transformation rule, we obtain with a constant $c$ depending only on $n$, $k$, and $p$ that
    \begin{equation*}
        \begin{aligned}
            \norm*[]{{L}^p(\hat K)}{\hat D^\alpha \hat v} &\leq c  \norm*{}{A_K}^k \sum_{|\beta|= k} \norm*{{L}^p(\hat K)} {D^\beta v \circ T_K}\\
            &\leq c  \norm*{}{A_K}^k |\det(A_K)|^{-\frac1p} | v|_{{W}^{k,p}( K)}.
        \end{aligned}
    \end{equation*}
    Summing over all $|\alpha| = k$ yields \eqref{eq:interp:transform1}. Arguing similarly using $T^{-1}_K$ yields  \eqref{eq:interp:transform2}.
\end{proof}

We now derive a geometrical estimate of the quantities appearing in the right-hand side of \eqref{eq:interp:transform1} and \eqref{eq:interp:transform2}. For a given element domain $K$, we define
\begin{itemize}
    \item the \emph{diameter} $h_K := \max_{x_1,x_2\in K} \norm*{}{x_1-x_2}$,
    \item the \emph{insphere diameter} $\rho_K :=2\max \{\rho>0: B_\rho(x)\subset K \text{ for some }x\in K\}$ (i.e., the diameter of the largest ball contained in $K$).
    \item the \emph{condition number} $\sigma_K := \frac{h_K}{\rho_K}$.
\end{itemize}
\begin{lemma}\label{lem:interp:estim}
    Let $T_K$ be an affine mapping defined as in \eqref{eq:interp:affine} such that $K=T_K(\hat K)$. Then
    \begin{equation*}
        |\det (A_K)| = \frac{\mathrm{vol}(K)}{\mathrm{vol}(\hat K)},\qquad
        \norm*{}{A_K}\leq \frac{h_K}{\rho_{\hat K}},\qquad
        \norm*{}{A_K^{-1}} \leq \frac{h_{\hat K}}{\rho_K}.
    \end{equation*}
\end{lemma}
\begin{proof}
    The first property follows from the transformation rule \eqref{eq:interp:transform} applied to the constant function $v\equiv 1$.
    For the second property, recall that the matrix norm of $A_K$ is given by
    \begin{equation*}
        \norm*{}{A_K} = \sup_{\norm*{}{\hat x}=1}\norm*{}{A_K \hat x} = \frac1{\rho_{\hat K}} \sup_{\norm*{}{\hat x}=\rho_{\hat K}} \norm*{}{A_K\hat x}.
    \end{equation*}
    Now for any $\hat x$ with $\norm*{}{\hat x}=\rho_{\hat K}$, there exists $\hat x_1,\hat x_2\in\hat K$ with $\hat x = \hat x_1 - \hat x_2$ (e.g., choose a suitable $\hat x_1$ on the insphere and $\hat x_2$ as its antipodal point). Then
    \begin{equation*}
        A_K\hat x = T_K \hat x_1 - T_K\hat x_2 = x_1 -x_2\quad\text{for some }x_1,x_2\in K,
    \end{equation*}
    which implies $ \norm*{}{A_K\hat x}\leq h_K$ and thus the desired inequality. The last property is obtained by exchanging the roles of $K$ and $\hat K$.
\end{proof}
Note that since the insphere of diameter $\rho_K$ is contained in $K$, which in turn is contained in the surrounding sphere of diameter $h_K$, we can further estimate (with a constant $c$ depending only on $n$)
\begin{equation*}
    c h_K^n \geq \mathrm{vol}(K) \geq c \rho_K^n = c  \frac{h_K^n}{\sigma_K^n}.
\end{equation*}

The local interpolation error can then be estimated by transforming to the reference element, bounding the error there, and transforming back (a so-called \emph{scaling argument}).
\begin{theorem}[local interpolation error]\label{thm:interp:local}
    Let $(\hat K,\hat \calP,\hat \calN)$ be a finite element with $P_{k-1}\subset\hat\calP$ for some $k\geq 1$ and $\hat \calN$ bounded on ${W}^{k,p}(\hat K)$,  $1\leq p\leq \infty$. For any element $(K,\calP,\calN)$ affine equivalent to $(\hat K,\hat \calP,\hat \calN)$ by the affine transformation $T_K$, there exists a constant $c>0$ independent of $K$ such that for any $v\in{W}^{k,p}(K)$,
    \begin{equation}\label{eq:interp:local}
        |v-\calI_K v|_{{W}^{l,p}(K)} \leq ch_K^{k-l}\sigma_K^l|v|_{{W}^{k,p}(K)}\quad\text{for all }0\leq l\leq k.
    \end{equation}
\end{theorem}
\begin{proof}
    Let $\hat v := v\circ T_K$. By \cref{thm:elements:equivalent}, $\calI_{\hat K}\hat v = (\calI_K v)\circ T_K$ (i.e., interpolating the transformed function is equivalent to transforming the interpolated function). Hence, we can apply \cref{lem:interp:transform} to $(v-\calI_Kv)$ and use \cref{thm:interp:reference} to obtain (with a generic constant $c$ that can change from line to line)
    \begin{equation*}
        \begin{aligned}
            |v-\calI_K v|_{{W}^{l,p}(K)} &\leq c  \norm*{}{A_K^{-1}}^l |\det(A_K)|^{\frac1p} |\hat v-\calI_{\hat K}\hat v |_{{W}^{l,p}(\hat K)}\\
            &\leq c  \norm*{}{A_K^{-1}}^l |\det(A_K)|^{\frac1p} |\hat v|_{{W}^{k,p}(\hat K)}\\
            &\leq c  \norm*{}{A_K^{-1}}^l  \norm*{}{A_K}^k |v|_{{W}^{k,p}(K)}\\
            &\leq c ( \norm*{}{A_K^{-1}}  \norm*{}{A_K} )^l \norm*{}{A_K}^{k-l} |v|_{{W}^{k,p}(K)}.
        \end{aligned}
    \end{equation*}
    The claim now follows from \cref{lem:interp:estim} and the fact that $h_{\hat K}$ and $\rho_{\hat K}$ are independent of~$K$.
\end{proof}

To obtain an estimate for the global interpolation error, which should converge to zero as $h\to 0$, we need to have a uniform bound (independent of $K$ and $h$) of the condition number $\sigma_K$. This requires a further assumption on the triangulation. A triangulation $\calT$ is called \emph{shape regular} if there exists a constant $\kappa$ independent of $h:=\max_{K\in\calT} h_K$ such that
\begin{equation*}
    \sigma_K \leq \kappa \qquad \text{ for all } K\in\calT.
\end{equation*}
(For triangular elements, e.g., this holds if all interior angles are bounded from below.)

Using this upper bound and summing over all elements, we obtain an estimate for the global interpolation error.
\begin{theorem}[global interpolation error]\label{thm:interp:global}
    Let $\calT$ be a shape regular affine triangulation of $\Omega\subset\R^n$ with the reference element $(\hat K,\hat \calP,\hat \calN)$ satisfying the requirements of \cref{thm:interp:local} for some $k\geq 1$. Then there exists a constant $c>0$ independent of $h$ such that for all $v\in\m{Wkp}$,
    \begin{align*}
        \norm{Lp}{v-\calI_{\calT}v} + \sum_{l=1}^k h^l \left(\sum_{K\in\calT}|v-\calI_K v|_{{W}^{l,p}(K)}^p\right)^\frac1p &\leq c h^{k} |v|_{\m{Wkp}}, \quad 1\leq p<\infty,\\[0.5ex]
        \norm{Linf}{v-\calI_{\calT}v} + \sum_{l=1}^k h^l \max_{K\in\calT}|v-\calI_K v|_{{W}^{l,\infty}(K)} &\leq c h^{k} |v|_{{W}^{k,\infty}(\Omega)}.
    \end{align*}
\end{theorem}
\enlargethispage{1cm}
Similar estimates can be obtained for elements based on the tensor product spaces $Q_k$.\footnote{e.g., \cite[Chapter~4.6]{Brenner:2008}}

\section{Inverse estimates}

The above theorems estimated the interpolation error in a coarser norm (i.e., $l\leq k$) than than the given function to be interpolated. In general, the converse (estimating a finer norm by a coarser one) is not possible; however, for the discrete approximations $v_h\in V_h$, such so-called \emph{inverse estimates} can be established.

Local estimates follow as above from a scaling argument, using the equivalence of norms on the finite dimensional space $\hat \calP$ in place of the Bramble--Hilbert lemma.
\begin{theorem}[local inverse estimate\footnote{e.g., \cite[Lemma 1.138]{Ern:2004}}]\label{thm:interp:local_inverse}
    Let $(\hat K,\hat \calP,\hat \calN)$ be a finite element with $\hat\calP\subset W^{l,p}(\hat K)$ for an $l\geq 0$ and $1\leq p\leq \infty$. For any element $(K,\calP,\calN)$  with $h_K\leq 1$ affine equivalent to $(\hat K,\hat \calP,\hat \calN)$ by the affine transformation $T_K$, there exists a constant $c>0$ independent of $K$ such that for any $v_h\in \calP$,
    \begin{equation*}
        \|v_h\|_{{W}^{l,p}(K)} \leq ch_K^{k-l}\|v_h\|_{{W}^{k,p}(K)}
    \end{equation*}
    for all $0\leq k\leq l$.
\end{theorem}

For uniform global estimates, we need a lower bound on $h_K^{-1}$. A triangulation $\calT$ is called \emph{quasi-uniform} if it is shape regular and there exists a $\tau\in(0,1]$ such that $h_K \geq \tau h$ for all $K\in\calT$. By summing over the local estimates, we obtain the following global estimate.
\begin{theorem}[global inverse estimate\footnote{e.g., \cite[Corollary 1.141]{Ern:2004}}]\label{thm:interp:global_inverse}
    Let $\calT$ be a quasi-uniform affine triangulation of $\Omega\subset\R^n$ with the reference element $(\hat K,\hat \calP,\hat \calN)$ satisfying the requirements of \cref{thm:interp:local_inverse} for an $l\geq 0$. Then there exists a constant $c>0$ independent of $h$ such that for all $v_h\in V_h:=\setof{v\in\m{Lp}}{v|_K \in \calP, K\in\calT}$,
    \begin{align*}
        \left(\sum_{K\in\calT}\|v_h\|_{{W}^{l,p}(K)}^p\right)^\frac1p &\leq c h^{k-l} \left(\sum_{K\in\calT}\|v_h\|_{{W}^{k,p}(K)}^p\right)^\frac1p, \quad 1\leq p<\infty,\\[0.5ex]
        \max_{K\in\calT}\|v_h\|_{{W}^{l,\infty}(K)} &\leq c h^{k-l}\left(\max_{K\in\calT}\|v_h\|_{{W}^{k,\infty}(K)}\right) ,
    \end{align*}
    for all $0\leq k\leq l$.
\end{theorem}

\chapter{Error estimates for the finite element approximation}

We can now give error estimates for the conforming finite element approximation of elliptic boundary value problems using Lagrange elements.
Let a reference element $(\hat K,\hat \calP,\hat \calN)$ and a triangulation $\calT$ using affine equivalent elements be given. Denoting the affine transformation from the reference element to the element $(K,\calP,\calN)$ by $T_K:\hat x\mapsto A_K\hat x +b_K$, we can define the corresponding ${C}^0$ finite element space by
\begin{equation}\label{eq:elements:femspace}
    V_h := \setof{v_h\in{C}^0(\bar\Omega)}{(v_h|_K\circ T_K) \in\hat\calP \text{ for all }K\in \calT}\cap V
\end{equation}
(the intersection being necessary in case of Dirichlet conditions).

\section{A priori error estimates}\label{sec:error:apriori}

By \nameref{thm:galerkin:cea}, the discretization error is bounded by the best-approximation error, which in turn can be bounded by the interpolation error. The results of the preceding chapters therefore yield the following a priori error estimates.

\begin{theorem}\label{thm:interp:apriori}
    Let $u\in\m{H1}$ be the solution of the boundary value problem \eqref{eq:weak:bvp} together with appropriate boundary conditions. Let $\calT$ be a shape regular affine triangulation of $\Omega\subset\R^n$ with the reference element $(\hat K,\hat \calP,\hat \calN)$ satisfying $P_{k-1}\in\hat\calP$ for some $k\geq 1$, and let $u_h\in V_h$ be the corresponding Galerkin approximation. If $u\in{H}^m(\Omega)$ for $\frac{n}2< m <k$, then there exists $c>0$ independent of $h$ and $u$ such that
    \begin{equation*}
        \norm{H1}{u-u_h}\leq c h^{m-1}|u|_{{H}^m(\Omega)}.
    \end{equation*}
\end{theorem}
\begin{proof}
    Since $m > \frac{n}2$, the Sobolev embedding \cref{thm:weak:sobolev} implies that $u\in\m{c}(\bar\Omega)$ and hence that the local (pointwise) interpolant is well defined. In addition, the nodal interpolation preserves homogeneous Dirichlet boundary conditions. Hence $\calI_\calT u\in V_h$, and \nameref{thm:galerkin:cea} yields
    \begin{equation*}
        \norm{H1}{u-u_h} \leq c \inf_{v_h\in V_h}\norm{H1}{u-v_h}\leq c \norm{H1}{u-\calI_\calT u}.
    \end{equation*}
    \Cref{thm:interp:global} for $p=2$, $l=1$, and $k=m$ further implies
    \begin{equation*}
        \norm{H1}{u-\calI_\calT u}\leq c  h^{m-1}|u|_{{H}^m(\Omega)},
    \end{equation*}
    and the claim follows by combining these estimates.
\end{proof}

If the bilinear form $a$ is symmetric, or if the adjoint problem to \eqref{eq:weak:bvp} is well-posed, we can apply the \nameref{thm:galerkin:aubin} to obtain better estimates in the $\m{l2}$ norm.
\begin{theorem}\label{thm:interp:aprioril2}
    Under the assumptions of \cref{thm:interp:apriori}, there exists $c>0$ such that
    \begin{equation*}
        \norm{L2}{u-u_h}\leq c h^{m}|u|_{{H}^m(\Omega)}.
    \end{equation*}
\end{theorem}
\begin{proof}
    By the Sobolev embedding \cref{thm:weak:sobolev}, the embedding $\m{H1}\hookrightarrow\m{L2}$ is continuous. Thus, the \nameref{thm:galerkin:aubin} yields
    \begin{equation*}
        \norm{L2}{u-u_h}\leq c \norm{H1}{u-u_h} \sup_{g\in \m{L2}}\left(\frac{1}{\norm{L2}{g}} \inf_{v_h\in V_h} \norm{H1}{\phi_g-v_h}\right),
    \end{equation*}
    where $\phi_g$ is the solution of the adjoint problem with right-hand side $g$. Estimating the best approximation in $V_h$ by the interpolant and using \cref{thm:interp:global}, we obtain
    \begin{equation*}
        \inf_{v_h\in V_h}\norm{H1}{\phi_g-v_h} \leq \norm{H1}{\phi_g - \calI_\calT \phi_g} \leq c h|\phi_g|_{{H}^2(\Omega)} \leq ch \norm{L2}{g}
    \end{equation*}
    by the well-posedness of the adjoint problem. Combining this inequality with the one from \cref{thm:interp:apriori} yields the claimed estimate.
\end{proof}

Using duality arguments based on different adjoint problems, one can derive estimates in other $\m{Lp}$ spaces, including $\m{Linf}$.\footnote{e.g., \cite[Chapter 8]{Brenner:2008}}

\section{A posteriori error estimates}

It is often the case that the regularity of the solution varies over the domain $\Omega$ (for example, near corners or jumps in the right-hand side or coefficients). It is then advantageous to make the element size $h_K$ small only where it is actually needed. Such information can be obtained using \emph{a posteriori error estimates}, which can be evaluated for a computed solution $u_h$ to decide where the mesh needs to be refined. Here, we will only sketch \emph{residual-based} error estimates and simple \emph{duality-based} estimates, and refer to the literature for details.\footnote{e.g., \cite[Chapter 9]{Brenner:2008}, \cite[Chapter 10]{Ern:2004}}

For the sake of presentation, we consider a simplified boundary value problem. Let $f\in\m{L2}$ and $\alpha\in\m{Linf}$ with $\alpha_1 \geq \alpha(x) \geq\alpha_0 >0 $ for almost all $x\in\Omega$ be given. Then we search for $u\in\m{H10}$ satisfying
\begin{equation}\label{eq:adaptivity:weak}
    a(u,v) := \inner{\alpha\nabla u,\nabla v} = \inner{f,v} \qquad\text{for all }v\in\m{H10}.
\end{equation}
(The same arguments can be carried out for the general boundary value problem \eqref{eq:weak:bvp} with homogeneous Dirichlet or Neumann conditions). Let $V_h\subset \m{H10}$ be a finite element space and let $u_h\in V_h$ be the corresponding Ritz--Galerkin approximation.

\paragraph{Residual-based error estimates}\label{sec:error:residual}

Residual-based estimates give an error estimate in the $\m{h1}$ norm. We first note that the bilinear form $a$ is coercive with constant $\alpha_0$, and hence we have
\begin{equation*}
    \begin{aligned}
        \alpha_0\norm{H1}{u-u_h} &\leq \frac{a(u-u_h,u-u_h)}{\norm{H1}{u-u_h}}\\
        &\leq \sup_{w\in\m{H10}}  \frac{a(u-u_h,w)}{\norm{H1}{w}}\\
        &= \sup_{w\in\m{H10}}  \frac{a(u,w)-(\alpha \nabla u_h,\nabla w)}{\norm{H1}{w}}\\
        &= \sup_{w\in\m{H10}}  \frac{(f,w)-\scalprod*{{H}^{-1},\m{h1}}{-\nabla\cdot(\alpha \nabla u_h)}{w}}{\norm{H1}{w}}\\
        &= \sup_{w\in\m{H10}}  \frac{\scalprod*{{H}^{-1},\m{h1}}{f+\nabla\cdot(\alpha \nabla u_h)}{w}}{\norm{H1}{w}}\\
        &= \norm*{{H}^{-1}(\Omega)}{f+\nabla\cdot(\alpha \nabla u_h)}
    \end{aligned}
\end{equation*}
using integration by parts and the definition of the dual norm. For brevity, we have written $\nabla\cdot w = \sum_{j=1}^n \partial_j w_j$ for the (distributional) divergence of $w\in\m{L2}^n$. Since all terms on the right-hand side are known, this is in principle already an a posteriori estimate. However, the ${H}^{-1}$ norm cannot be localized, so we will perform the integration by parts on each element separately and insert an interpolation error to eliminate the $\m{h1}$ norm of $w$ (and hence the supremum).

This requires some notation. Let $\calT_h$ be the triangulation corresponding to $V_h$ and $\partial\calT_h$ the set of faces of all $K\in\calT_h$. The set of all interior faces will be denoted by $\Gamma_h$, i.e.,
\begin{equation*}
    \Gamma_h = \setof{F\in \partial\calT_h}{F\cap\partial\Omega = \emptyset}.
\end{equation*}
For $F\in\Gamma_h$ with $F=\bar K_1\cap \bar K_2$, let $\nu_1$ and $\nu_2$ denote the unit outward normal to $K_1$ and $K_2$, respectively. We define the jump in normal derivative for $w_h\in V_h$ across $F$ (noting that $\nu_1=-\nu_2$) as
\begin{equation*}
    \jump{\alpha\nabla w_h}_F := (\alpha\nabla w_h)|_{K_1}\cdot \nu_1 + (\alpha\nabla w_h)|_{K_2}\cdot \nu_2\in L^2(F).
\end{equation*}

Assume now that $\alpha$ is piecewise smooth and that the triangulation is chosen such that $\alpha|_K \in C^1(\overline K)$ for every $K\in \calT_h$.
We can then perform the above integration by parts elementwise to obtain for $w\in\m{H10}$
\begin{equation*}
    \begin{aligned}
        a(u-u_h,w) &= (f,w) - a(u_h,w)\\
        &=(f,w) - \sum_{K\in\calT_h}\int_K  \alpha \nabla (u-u_h) \cdot \nabla w \,dx\\
        &= \sum_{K\in\calT_h} \left(\int_K (f + \nabla\cdot(\alpha\nabla u_h))\, w\,dx - \sum_{F\in\partial K} \int_F(\alpha\nabla u_h\cdot \nu)\, w\,ds\right)\\
        &=  \sum_{K\in\calT_h} \int_K (f + \nabla\cdot(\alpha\nabla u_h))\, w\,dx -\sum_{F\in\Gamma_h}\int_F \jump{\alpha\nabla u_h}_F w\,ds
    \end{aligned}
\end{equation*}
since $w\in\m{H10}$ is continuous almost everywhere and $u_h$ is a polynomial on $K$ and hence $(\alpha\nabla u_h)|_K$ is in fact differentiable.

Our next task is to get rid of $w$ by canceling $\norm{H1}{w}$ in the definition of the dual norm. We do this by inserting (via Galerkin orthogonality) the interpolant of $w$ and applying an interpolation error estimate. The difficulty here is that $w\in\m{H10}$ is not sufficiently smooth to allow Lagrange interpolation, since pointwise evaluation is not well-defined. To circumvent this, we combine interpolation with projection.
For $K\in\calT_h$, let $\omega_K$ be the union of all elements touching $K$, i.e.,
\begin{equation*}
    \omega_K = \bigcup \setof{\bar K'\in\calT_h}{\bar K' \cap \bar K \neq \emptyset}.
\end{equation*}
Furthermore, for every node $z$ of $K$ (i.e., there is $N\in\calN$ such that $N(v) = v(z)$), denote
\begin{equation*}
    \omega_z = \bigcup \setof{\bar K' \in\calT_h}{z \in\bar K'}\subset\omega_K.
\end{equation*}
The $\m{l2}(\omega_z)$ projection of $v\in \m{H1}$ onto $P_0$ is then defined as the unique $\pi_z(v)\in P_0$ satisfying
\begin{equation*}
    \int_{\omega_z} (\pi_z(v)-v) q\,dx = 0 \quad \text{ for all } q \in P_0,
\end{equation*}
see \cref{lem:bramble:projection} for $k=0$.
For $z\in\partial\Omega$, we set $\pi_z(v) = 0$ to respect the homogeneous Dirichlet conditions.
The local \emph{Clément interpolant} $\calI_C v\in V_h$ of $v\in\m{H10}$ is then given by
\begin{equation*}
    \calI_C v  = \sum_{i=1}^d N_i(\pi_{z_i}(v))\psi_i.
\end{equation*}
Since the $L^2(\omega_z)$ projection is continuous on $H^1(\omega_K)$ for every $z\in \overline K$, we can apply the \nameref{thm:bramble} together with a scaling argument to obtain as for the standard interpolation the error estimates\footnote{e.g., \cite[Theorem II.6.9]{Braess:2007}}
\begin{align*}
    \norm*{\m{l2}(K)}{v-\calI_C v} & \leq c h_K \norm*{\m{h1}(\omega_K)}{v},\\
    \norm*{\m{l2}(F)}{v-\calI_C v} & \leq c h_K^{1/2} \norm*{\m{h1}(\omega_K)}{v},
\end{align*}
for all $v\in\m{H10}$, $K\in\calT_h$ and $F\subset\partial K$. (Note $\mathrm{vol}(F)$ scales with $h_K$, while $\mathrm{vol}(K)$ scales with $h_K^2$.)

Using the Galerkin orthogonality for the global Clément interpolant $\calI_C w\in V_h$, we can proceed as before, use the Cauchy--Schwarz inequality, these interpolation error estimates, and the fact that every $K$ appears only in a finite number of $\omega_K$, to arrive at the estimate
\begin{equation*}
    \begin{aligned}
        \norm{H1}{u-u_h} &\leq \frac{1}{\alpha_0} \sup_{w\in\m{H10}}  \frac{a(u-u_h,w-\calI_C w)}{\norm{H1}{w}}\\
        &\leq \frac{1}{\alpha_0} \sup_{w\in\m{H10}}\frac{1}{\norm{H1}{w}}
        \left(\sum_{K\in\calT_h}\norm*{\m{l2}(K)}{f + \nabla\cdot(\alpha\nabla u_h)} \norm*{\m{l2}(K)}{w-\calI_C w} \right.\\
        \MoveEqLeft[-11] \left.+ \sum_{F\in\Gamma_h} \norm*{\m{l2}(F)}{ \jump{\alpha\nabla u_h}_F} \norm*{\m{l2}(F)}{w-\calI_C w}\right)\\
        &\leq C\sup_{w\in\m{H10}}\frac{1}{\norm{H1}{w}}
        \left(\sum_{K\in\calT_h}h_K\norm*{\m{l2}(K)}{f + \nabla\cdot(\alpha\nabla u_h)} \norm*{\m{H1}}{w} \right.\\
        \MoveEqLeft[-11] \left.+ \sum_{F\in\Gamma_h} h_K^{1/2} \norm*{\m{l2}(F)}{ \jump{\alpha\nabla u_h}_F} \norm*{\m{H1}}{w}\right)\\
        &\leq C\left(\sum_{K\in\calT_h} h_K \norm*{\m{l2}(K)}{f+\nabla\cdot(\alpha\nabla u_h)} + \sum_{F\in\Gamma_h}h_K^{1/2}\norm*{\m{l2}(F)}{\jump{\alpha\nabla u_h}_F}\right).
    \end{aligned}
\end{equation*}
All terms on the right-hand side are now fully computable given a discrete solution $u_h$. To obtain a minimal upper bound (as a proxy for minimizing the error itself), we are thus lead to make $h_K$ small(er) (by subdividing $K$ into a number of smaller elements) where the \emph{finite element residual} is large or the normal derivative has a large jump.

\paragraph{Duality-based error estimates}

The use of Clément interpolation can be avoided if we are satisfied with an a posteriori error estimate in the $\m{l2}$ norm as we can then apply the Aubin--Nitsche trick. Let $w\in\m{H10}\cap\m{H2}$ solve the adjoint problem
\begin{equation*}
    a(v,w) = (u-u_h,v) \qquad \text{for all }v\in \m{H10}.
\end{equation*}
Inserting $v= u-u_h\in\m{H10}$ and applying the Galerkin orthogonality $a(u-u_h,w_h)=0$ for the global interpolant $w_h:=\calI_\calT w$ then yields
\begin{equation*}
    \begin{aligned}
        \norm{L2}{u-u_h}^2 &= (u-u_h,u-u_h) = a(u-u_h,w- w_h)\\
        &=  (f,w-w_h) - a(u_h,w- w_h).
    \end{aligned}
\end{equation*}
Now we again integrate by parts on each element and apply the Cauchy--Schwarz inequality to obtain
\begin{equation*}
    \begin{aligned}
        \norm{L2}{u-u_h}^2 &\leq \sum_{K\in\calT_h}  \norm*{\m{l2}(K)}{f+\nabla\cdot(\alpha\nabla u_h)}\norm*{\m{l2}(K)}{w-w_h}\\
        \MoveEqLeft[-1] + \sum_{F\in\Gamma_h}\norm*{\m{l2}(F)}{\jump{\alpha\nabla u_h}_F}\norm*{\m{l2}(F)}{w- w_h}.
    \end{aligned}
\end{equation*}

By the symmetry of $a$ and the well-posedness of \eqref{eq:adaptivity:weak}, we have $w\in\m{H2}$ due to \cref{thm:weak:regularity:convex}. We can thus estimate the local interpolation error for $w$ using \cref{thm:interp:local} for $k=2$,  $l=0$ and  $p=2$ to obtain
\begin{equation*}
    \norm*{\m{l2}(K)}{w-w_h}\leq c h_K^2 \norm{H2}{w}.
\end{equation*}
Similarly, using the \nameref{thm:bramble} and a scaling argument yields
\begin{equation*}
    \norm*{\m{l2}(F)}{w-w_h}\leq c h_K^{3/2} \norm{H2}{w}.
\end{equation*}
Finally, we have from \cref{thm:weak:regularity:convex} the estimate
\begin{equation*}
    \norm{H2}{w}\leq C \norm{L2}{u-u_h}.
\end{equation*}
Combining these inequalities, we obtain the desired a posteriori error estimate
\begin{equation*}
    \norm{L2}{u-u_h} \leq C\left(\sum_{K\in\calT_h} h_K^2 \norm*{\m{l2}(K)}{f+\nabla\cdot(\alpha\nabla u_h)} + \sum_{F\in\Gamma_h}h_K^{3/2}\norm*{\m{l2}(F)}{\jump{\alpha\nabla u_h}_F}\right).
\end{equation*}

\bigskip

Such a posteriori estimates can be used to locally decrease the mesh size in order to reduce the discretization error. This leads to \emph{adaptive finite element methods}, which is a very active area of current research. For details, we refer to, e.g., \cite[Chapter 9]{Brenner:2008}, \cite{Verfuerth:2013}.

\chapter{Implementation}\label{chap:implementation}

This chapter discusses some of the issues involved in the implementation of the finite element method on a computer. It should only serve as a guide for solving model problems and understanding the structure of professional software packages; due to the availability of high-quality free and open source frameworks such as \texttt{deal.II}\footnote{\cite{dealII}, \url{http://www.dealii.org}} and \texttt{FEniCS}\footnote{\cite{fenics}, \url{http://fenicsproject.org}}, there is usually no need to write a finite element solver from scratch.

In the following, we focus on triangular Lagrange and Hermite elements on polygonal domains; the extension to higher-dimensional and quadrilateral elements is fairly straightforward.

\section{Triangulation}

The geometric information on a triangulation is described by a \emph{mesh}, a cloud of connected points in $\R^2$. This information is usually stored in a collection of two-dimensional arrays, the most fundamental of which are
\begin{itemize}
    \item \emph{the list of nodes}, which contains the coordinates $z_i=(x_i,y_i)$ of each node corresponding to a degree of freedom:
        \begin{center}
            \verb!nodes(i) = (x_i,y_i)!;
        \end{center}
    \item \emph{the list of elements}, which contains for every element in the triangulation the corresponding entries in \texttt{nodes} of the nodal variables:
        \begin{center}
            \verb!elements(i) = (i_1,i_2,i_3)!,
        \end{center}
        where $z_{i_1} = $\verb!nodes(i_1)!.
        Care must be taken that the ordering is consistent for each element. Points for which both function and gradient evaluation are given appear twice and are discerned by position in the list (usually function values first, then gradient).
\end{itemize}
The array \verb!elements! serves as the \emph{local-to-global index}.
Depending on the boundary conditions, the following are also required:
\begin{itemize}
    \item for Dirichlet conditions, a \emph{list of boundary points} \verb!bdy_nodes!;
    \item for Neumann conditions, a \emph{list of boundary faces} \verb!bdy_faces! which contain the (consistently ordered) entries in \texttt{nodes} of the nodes on each face.
\end{itemize}

The generation of a good (quasi-uniform) mesh for a given complicated domain is an active research area in itself. For uniform meshes on simple geometries (such as rectangles), it is possible to create the needed data structures by hand. An alternative are \emph{Delaunay triangulations}, which can be constructed (e.g., by the \textsc{matlab} command \verb!delaunay!) given a list of nodes. More complicated generators can create meshes from a geometric description of the boundary; an example is the \textsc{matlab} package \verb!distmesh!.\footnote{\url{http://persson.berkeley.edu/distmesh}; an almost exhaustive list of mesh generators can be found at \url{http://www.robertschneiders.de/meshgeneration/software.html}.}

\section{Assembly}\label{sec:implementation:assembly}

The main effort in implementing lies in assembling the stiffness matrix $\mathbf{K}$, i.e., computing its entries $K_{ij} = a(\phi_i,\phi_j)$ for all basis elements $\phi_i$, $\phi_j$. This is most efficiently done element-wise, where the computation is performed by transformation to a reference element.

\paragraph{The reference element}

We consider the reference element domain
\begin{equation*}
    \hat K = \setof{(\xi_1,\xi_2)\in\R^2}{0\leq \xi_1,\xi_2 \leq 1, \text{ and } \xi_1+\xi_2 \leq 1},
\end{equation*}
with the vertices $z_1 = (0,0)$, $z_2 = (1,0)$, $z_3 = (0,1)$ (in this order).
For any triangle $K$ defined by the ordered set of vertices $((x_1,y_1),(x_2,y_2),(x_3,y_3))$, the affine transformation $T_K$ from $\hat K$ to $K$ is given by $T_K(\xi) = A_K\xi + b_K$ with
\begin{equation*}
    \quad
    A_K = \begin{pmatrix} x_2-x_1 & x_3 - x_1 \\ y_2 -y_1 & y_3-y_1\end{pmatrix},\qquad
    b_K = \begin{pmatrix} x_1\\y_1 \end{pmatrix}.
\end{equation*}

Given a set of nodal variables $\hat \calN=(\hat N_1,\dots,\hat N_d)$, it is straightforward (if tedious) to compute the corresponding nodal basis functions $\hat \psi_i$ from the conditions $\hat N_i(\hat\psi_j)=\delta_{ij}$, $1\leq i,j\leq d$. (For example, the nodal basis for the linear Lagrange element is $\{1-\xi_1-\xi_2, \xi_1, \xi_2\}$.)

If the coefficients in the bilinear form $a$ are constant, one can then compute the integrals on the reference element exactly, noting that due to the affine transformation, the partial derivatives of the basis functions change according to
\begin{equation*}
    \nabla \psi(x) = A_K^{-T} \nabla\hat\psi(\xi).
\end{equation*}

\paragraph{Quadrature}
If the coefficients are not given analytically, it is necessary to evaluate the integrals using numerical quadrature, i.e., to compute
\begin{equation*}
    \int_K v(x)\,dx \approx \sum_{k=1}^r w_k v(x_k)
\end{equation*}
using appropriate \emph{quadrature weights} $w_k$ and \emph{quadrature nodes} $x_k$.
Since this amounts to replacing the bilinear form $a$ by $a_h$ (a \emph{variational crime}\footnote{\cite{Strang:1972}}), care must be taken that the discrete problem is still well-posed and that the quadrature error is negligible compared to the approximation error. It is possible to show that this can be ensured if the quadrature is sufficiently exact and the weights are positive (see \cref{chap:gen_galerkin}).
\begin{theorem}[effect of quadrature\footnote{e.g., \cite[Theorems~4.1.2, 4.1.6]{Ciarlet:2002}}]\label{thm:implementation:quadrature}
    Let $\calT_h$ be a shape regular affine triangulation with $P_1\subset\hat\calP\subset P_k$ for $k\geq 1$. If the quadrature on $\hat K$ is of order $2k-2$, all weights are positive, and $h$ is small enough, then the discrete problem is well-posed.

    If in addition the surface integrals are approximated by a quadrature rule of order $2k-1$ and the conditions of \cref{thm:interp:apriori} hold, there exists a $c>0$ such that for $f\in{H}^{k-1}(\Omega)$ and $g\in{H}^k(\partial\Omega)$ and sufficiently small $h$,
    \begin{equation*}
        \norm{H1}{u-u_h}\leq c h^{k-1}(\norm*{{H}^k(\Omega)}{u}+\norm*{{H}^{k-1}(\Omega)}{f}+\norm*{{H}^k(\partial\Omega)}{g}).
    \end{equation*}
\end{theorem}
The rule of thumb is that the quadrature should be exact for the integrals involving second-order derivatives if the coefficients were constant. For linear elements (where the gradients are constant), order $0$ (i.e., the midpoint rule) is therefore sufficient to obtain an error estimate of order $h$.

For higher order elements, Gauß quadrature is usually employed. This is simplified by using \emph{barycentric coordinates}: If the vertices of $K$ are $(x_1,y_1)$, $(x_2,y_2)$, and $(x_3,y_3)$, the barycentric coordinates $(\zeta_1,\zeta_2,\zeta_3)$ of $(x,y)\in K$ are defined by
\begin{enumerate}[(i)]
    \item $\zeta_1,\zeta_2,\zeta_3\in[0,1]$,
    \item $\zeta_1+\zeta_2+\zeta_3=1$,
    \item $(x,y) = \zeta_1(x_1,y_1)+\zeta_2(x_2,y_2)+\zeta_3(x_3,y_3)$.
\end{enumerate}
Barycentric coordinates are invariant under affine transformations: If $\xi\in\hat K$ has the barycentric coordinates $(\zeta_1,\zeta_2,\zeta_3)$ with respect to the vertices of $\hat K$, then $x=T_K\xi$ has the same coordinates with respect to the vertices of $K$.
The Gauß nodes in barycentric coordinates and the corresponding weights for quadrature of order up to $5$ are given in~\cref{tab:implement:gauss}. The element contributions of the local basis functions can then be computed as, e.g., in
\begin{equation*}
    \int_K \scalprod*{}{A(x)\nabla \phi_i(x)}{\nabla\phi_j(x)} \,dx \approx
    \det(A_K) \sum_{k=1}^{n_l} w_k \scalprod*{}{A(x_k) A_K^{-T}\nabla\hat\psi_i(\xi_k)}{A_K^{-T}\nabla\hat\psi_j(\xi_k)},
\end{equation*}
where $A(x)=(a_{ij}(x))_{i,j=1}^2$ is the matrix of coefficients for the second-order derivatives, $n_l$ is the number of Gauss nodes, $x_k$ and $\xi_k$ are the Gauß nodes on the element and reference element, respectively, and $\hat\psi_i$, $\hat\psi_j$ are the basis functions on the reference element corresponding to $\phi_i$, $\phi_j$. The other integrals in $a$ and $F$ are calculated similarly.

\begin{SCtable}
    \caption{Gauß nodes $x_k$ (in barycentric coordinates) and weights $w_k$ on the reference triangle. The quadrature is exact up to order $l$ and uses $n_l$ nodes. For starred nodes, all possible permutations appear with identical weights.}\label{tab:implement:gauss}
    \centering
    \renewcommand{\arraystretch}{1.3}
    \begin{tabular}{cccc}
        \toprule
        $l$ & $n_l$ & $x_k$ & $w_k$\\
        \midrule
        $1$ & $1$ & $(\tfrac13,\tfrac13,\tfrac13)$ & $\frac12$\\
        $2$ & $3$ & $(\tfrac16,\tfrac16,\tfrac23)^\star$ & $\frac16$\\
        $3$ & $7$ & $(\tfrac13,\tfrac13,\tfrac13)$& $\frac{9}{40}$\\
        && $(\tfrac12,\tfrac12,0)^\star$ & $\frac{2}{30}$\\
        &&$(0,0,1)^\star$ &  $\frac{1}{40}$\\
        $5$ & $7$ & $(\tfrac13,\tfrac13,\tfrac13)$ &$\frac{9}{80}$\\
        && $(\frac{6-\sqrt{15}}{21},\frac{6-\sqrt{15}}{21},\frac{9+2\sqrt{15}}{21})^\star$& $\frac{155-\sqrt{15}}{2400}$\\
        &&$(\frac{6+\sqrt{15}}{21},\frac{6+\sqrt{15}}{21},\frac{9-2\sqrt{15}}{21})^\star$ & $\frac{155+\sqrt{15}}{2400}$\\
        \bottomrule
    \end{tabular}
\end{SCtable}

The complete procedure for the assembly of the stiffness matrix $\mathbf{K}$ and right-hand side $\mathbf{F}$ is sketched in \cref{alg:implementation}.

\paragraph{Boundary conditions}

It remains to incorporate the boundary conditions. For Dirichlet conditions $u=g$ on $\partial\Omega$, it is most efficient to assemble the stiffness matrices and right-hand side as above, and replace each row in $\mathbf{K}$ and entry in $\mathbf{F}$ corresponding to a node in \verb!bdy_nodes! with the equation for the prescribed nodal value:
\begin{algorithmic}[1]
    \For{$i=1,\dots,{\texttt{length(bdy\_nodes)}}$}
    \State Set $k = $ \verb!bdy_nodes!$(i)$
    \State Set $K_{k,j} = 0$ for all $j$
    \State Set $K_{k,k} = 1$,  ${F}_k = g(\texttt{nodes}(k))$
    \EndFor
\end{algorithmic}

For inhomogeneous Neumann or for Robin boundary conditions, one assembles the contributions to the boundary integrals from each face similarly to \cref{alg:implementation}, where the loop over \texttt{elements} is replaced by a loop over \verb!bdy_faces! (and one-dimensional Gauß quadrature is used).
\begin{algorithm}[ht!]
    \caption{Finite element method for Lagrange triangles}\label{alg:implementation}
    \begin{algorithmic}[1]
        \Require{mesh \texttt{nodes}, \texttt{elements}, data $a_{ij}$,$b_j$,$c$,$f$}
        \State Compute Gauß nodes $\xi_l$ and weights $w_l$ on reference element
        \State Compute values of nodal basis elements and their gradients at Gauß nodes on reference element
        \State Set $K_{ij} = F_{j} = 0$
        \For{$k=1,\dots,\texttt{length(elements)}$}
        \State Compute transformation $T_K$, Jacobian $\det(A_K)$ for element $K=\texttt{elements}(k)$
        \State Evaluate coefficients and right-hand side at transformed Gauß nodes $T_K(\xi_l)$
        \State Compute $a(\phi_i,\phi_j)$, $\inner{f,\phi_j}$ for all nodal basis elements $\phi_i,\phi_j$ using transformation rule and Gauß quadrature on reference element
        \For{$i,j=1,\dots,d$}
        \State Set $r=\texttt{elements}(k,i)$, $s=\texttt{elements}(k,j)$
        \State Set $K_{r,s} \gets  K_{r,s} + a(\phi_i,\phi_j)$, $F_{s} \gets  F_{s} + \inner{f,\phi_j}$
        \EndFor
        \EndFor
        \Ensure{$K_{ij},F_{j}$}
    \end{algorithmic}
\end{algorithm}

\part{Nonconforming Finite Elements}\label{part:nonconforming}

\chapter{Generalized Galerkin approach}\label{chap:gen_galerkin}

The results of the preceding chapters depended on the conformity of the Galerkin approach: the discrete problem is obtained by restricting the continuous problem to suitable subspaces. This is too restrictive for many applications beyond standard second order elliptic problems, where it would be necessary to consider
\begin{itemize}
    \item \emph{Petrov--Galerkin} approaches, where the function $u$ satisfying $a(u,v)$ for all $v\in V$ is an element of $U\neq V$;
    \item \emph{non-conforming} approaches, where the discrete spaces $U_h$ and $V_h$ are not subspaces of $U$ and $V$, respectively; and
    \item \emph{non-consistent} approaches, where the discrete problem involves a bilinear form $a_h \neq a$ (and $a_h$ might not be well-defined for all $u\in U$).
\end{itemize}
We thus need a more general framework that covers these cases as well. Let $U$, $V$ be Banach spaces, where $V$ is reflexive, and let $U^*$, $V^*$ denote their topological duals. Given a bilinear form $a:U\times V\to\R$ and a continuous linear functional $F\in V^*$, we are looking for $u\in U$ satisfying
\begin{equation}\label{eq:gen_galerkin}
    a(u,v) = F(v) \quad \text{ for all } v\in V.\tag{$\mathcal{W}$}
\end{equation}
The following generalization of the Lax--Milgram theorem gives sufficient (and, as can be shown, necessary) conditions for the well-posedness of \eqref{eq:gen_galerkin}.
\begin{theorem}[Banach--Ne\v{c}as--Babu\v{s}ka]\label{thm:BNB}
    Let $U$ and $V$ be Banach spaces and $V$ be reflexive. Let a bilinear form $a: U\times V\to\R$ and a linear functional $F:V\to\R$ be given satisfying the following assumptions:
    \begin{enumerate}[(i)]
        \item\emph{Inf--sup condition}: there exists a $c_1>0$ such that
            \begin{equation*}
                \inf_{u\in U}\sup_{v\in V} \frac{a(u,v)}{\norm*{U}{u}\norm*{V}{v}} \geq c_1 .
            \end{equation*}
        \item\emph{Continuity}: there exist $c_2,c_3$ such that
            \begin{align*}
                |a(u,v)|&\leq c_2\norm*{U}{u}\norm*{V}{v},\\
                |F(v)|&\leq c_3 \norm*{V}{v}
            \end{align*}
            for all $u\in U$, $v\in V$.
        \item\emph{Injectivity}: for any $v\in V$,
            \begin{equation*}
                a(u,v) = 0 \text{ for all } u\in U \quad\text{implies} \quad v=0.
            \end{equation*}
    \end{enumerate}
    Then there exists a unique solution $u\in U$ to \eqref{eq:gen_galerkin}, which satisfies
    \begin{equation*}
        \norm*{U}{u}\leq \frac{1}{c_1} \norm*{V^*}{F}.
    \end{equation*}
\end{theorem}
\begin{proof}
    The proof is essentially an application of the closed range theorem:\footnote{e.g., \cite[Theorem 3.E]{Zeidler:1995b}}
    For a bounded linear operator $A$ between two Banach spaces $X$ and $Y$, the range $\ran A$ of $A$ is closed in $Y$ if and only if $\ran A =(\ker A^*)^\bot$, where $A^*:Y^*\to X^*$ is the adjoint of $A$, $\ker A :=\setof{x\in X}{Ax = 0}$ is the null space of an operator $A:X\to Y$, and for $V\subset X$,
    \begin{equation*}
        V^\bot :=\setof{x\in X^*}{\dual{X}{x}{v}=0 \text{ for all } v\in V}
    \end{equation*}
    is the polar of $V$. We apply this theorem to the operator $A: U \to V^*$ defined by
    \begin{equation*}
        \dual{V}{Au}{v} = a(u,v) \quad \text{ for all } v\in V
    \end{equation*}
    to show that $A$ is an isomorphism (i.e., that $A$ is bijective and $A$ and $A^{-1}$ are continuous), which is equivalent to the claim since \eqref{eq:gen_galerkin} can be expressed as $Au=F$.

    Continuity of $A$ easily follows from continuity of $a$ and the definition of the norm on $V^*$.
    We next show injectivity of $A$. Let $u_1,u_2\in U$ be given with $Au_1 = Au_2$. By definition, this implies $a(u_1,v)=a(u_2,v)$ and hence $a(u_1-u_2,v) =0$ for all $v\in V$. Hence, the inf--sup condition implies that
    \begin{equation*}
        c_1\norm*{U}{u_1-u_2} \leq \sup_{v\in V}\frac{a(u_1-u_2,v)}{\norm*{V}{v}}=0
    \end{equation*}
    and therefore $u_1 = u_2$.

    Due to the injectivity of $A$, for any $v^*\in \ran A\subset V^*$ we have a unique $u=:A^{-1}v^*\in U$, and the inf--sup condition yields
    \begin{equation}\label{eq:bnb:continuity}
        c_1\norm*{U}{u}\leq \sup_{v\in V}\frac{a(u,v)}{\norm*{V}{v}}=  \sup_{v\in V}\frac{\dual{V}{Au}{v}}{\norm*{V}{v}} = \sup_{v\in V}\frac{\dual{V}{v^*}{v}}{\norm*{V}{v}} = \norm*{V^*}{v^*}.
    \end{equation}
    Any preimage thus satisfies the claimed inequality; it remains to show that every $v^*\in V^*$ has a preimage.
    We next show that $\ran A$ is closed. Let $\{v^*_n\}_{n\in\N}\subset \ran A\subset V^*$ be a sequence converging to a $v^*\in V^*$, i.e., there exists $u_n\in U$ such that $v^*_n=Au_n$, and the $v^*_n$ form a Cauchy sequence. From \eqref{eq:bnb:continuity}, we deduce for all $n,m\in\N$ that
    \begin{equation*}
        \norm*{U}{u_n-u_m} \leq \frac1{c_1}\norm*{V^*}{A(u_n - u_m)} = \frac1{c_1}\norm*{V^*}{v^*_n-v^*_m},
    \end{equation*}
    which implies that $\{u_n\}_{n\in\N}$ is a Cauchy sequence as well and thus converges to a $u\in U$. The continuity of $A$ then yields
    \begin{equation*}
        v^* = \lim_{n\to\infty} v^*_n = \lim_{n\to\infty} A u_n = Au,
    \end{equation*}
    and we obtain $v^*\in \ran A$.
    We can therefore apply the closed range theorem. By the reflexivity of $V$, we have $A^*:V\to U^*$ and
    \begin{equation*}
        \begin{aligned}
            \ker A^*  &= \setof{v\in V}{A^*v = 0}\\
            &= \setof{v\in V}{\dual{U}{A^*v}{u}=0 \text{ for all }u\in U}\\
            &=\setof{v\in V}{\dual{V}{Au}{v}=0  \text{ for all }u\in U}\\
            &=\setof{v\in V}{a(u,v)=0 \text{ for all }u\in U}\\
            &=\{0\}
        \end{aligned}
    \end{equation*}
    due to the injectivity condition (iii). Hence the closed range theorem and the reflexivity of $V$ yields
    \begin{equation*}
        \ran A = (\{0\})^\bot = \setof{v^*\in V^*}{\dual{V}{v^*}{0}=0} = V^*,
    \end{equation*}
    and therefore surjectivity of $A$. Thus, $A$ is an isomorphism and the claimed estimate follows from \eqref{eq:bnb:continuity} applied to $v^*=F\in V^*$.
\end{proof}

The term \enquote{injectivity condition} is due to the fact that it implies injectivity of the adjoint operator $A^*$ and hence (due to the closed range of $A$) surjectivity of $A$. Note that in the symmetric case where $U=V$ is a Hilbert space, coercivity of $a$ implies both the inf--sup condition and (via contraposition) the injectivity condition, and we recover the Lax--Milgram lemma.

\bigskip

For the \emph{non-conforming} Galerkin approach, we replace $U$ by $U_h$ and $V$ by $V_h$, where $U_h$ and $V_h$ are finite-dimensional spaces, and introduce a bilinear form $a_h:U_h\times V_h\to\R$ and a linear functional $F_h:V_h\to\R$. We then search for $u_h\in U_h$ satisfying
\begin{equation}\label{eq:gen_galerkin_h}
    a_h(u_h,v_h) = F_h(v_h) \quad \text{ for all } v_h\in V_h.\tag{$\mathcal{W}_h$}
\end{equation}
In contrast to the conforming setting, the well-posedness of \eqref{eq:gen_galerkin_h} cannot be deduced from the well-posedness of \eqref{eq:gen_galerkin} but needs to be proved independently. This is somewhat simpler due to the finite-dimensionality of the spaces, where injectivity of a square matrix already implies surjectivity.
\begin{theorem}\label{thm:bnb_h}
    Let $U_h$ and $V_h$ be finite-dimensional vector spaces with norms $\|\cdot\|_{U_h}$ and $\|\cdot\|_{V_h}$, respectively.
    Let a bilinear form $a_h: U_h\times V_h\to\R$ and a linear functional $F_h:V_h\to\R$ be given satisfying the following assumptions:
    \begin{enumerate}[(i)]
        \item\emph{discrete Inf--sup condition}: there exists a $c_1>0$ such that
            \begin{equation*}
                \inf_{u_h\in U_h}\sup_{v_h\in V_h} \frac{a_h(u_h,v_h)}{\norm*{U_h}{u_h}\norm*{V_h}{v_h}} \geq c_1 .
            \end{equation*}
        \item\emph{Continuity}: there exist $c_2,c_3$ such that
            \begin{align*}
                |a_h(u_h,v_h)|&\leq c_2\norm*{U_h}{u_h}\norm*{V_h}{v_h},\\
                |F_h(v_h)|&\leq c_3 \norm*{V_h}{v_h}
            \end{align*}
            for all $u_h\in U_h$, $v_h\in V_h$.
    \end{enumerate}
    Assume further that $\dim U_h = \dim V_h$.
    Then there exists a unique solution $u_h\in U_h$ to \eqref{eq:gen_galerkin_h}, which satisfies
    \begin{equation*}
        \norm*{U_h}{u_h}\leq \frac{1}{c_1} \norm*{V_h^*}{F_h}.
    \end{equation*}
\end{theorem}
\begin{proof}
    Consider a basis $\{\phi_1,\dots,\phi_n\}$ of $U_h$ and $\{\psi_1,\dots,\psi_n\}$ of $V_h$ and define the matrix $\mathbf{K}\in\Rnn$, $K_{ij} = a_h(\phi_i,\psi_j)$. Then the claim is equivalent to the invertibility of $\mathbf{K}$. From the inf--sup condition, we obtain injectivity of $\mathbf{K}$ by arguing as in the continuous case. By the rank theorem and the condition $\dim U_h = \dim V_h$, this implies surjectivity of $\mathbf{K}$ and hence invertibility. The estimate follows again from the inf--sup condition.
\end{proof}
Note that since $U_h$ and $V_h$ are no longer assumed to be a subspaces of $U$ and $V$, respectively, they can't simply inherit the norms of the latter, and we thus have to choose new one. These \emph{discrete norms} will have to be specifically adapted to the discrete bilinear form $a_h$ in order to ensure that both the inf--sup and the continuity condition are satisfied.
Note also the difference between \cref{thm:bnb_h} and the \nameref{thm:weak:laxmilgram} in the discrete case: In the latter, the coercivity condition amounts to the assumption that the matrix $\mathbf{K}$ is positive definite, while the inf-sup- and injectivity condition only amounts to requiring injectivity and surjectivity.

\bigskip

The error estimates for non-conforming methods are based on the following two generalization of Céa's lemma.
Although we do not require $U_h\subset U$ and $V_h\subset V$, we need to have some way of comparing elements of $U$ and $U_h$ in order to obtain error estimates for the solution $u_h$. We therefore assume that there exists a subspace $U_* \subset U$ containing the exact solution such that
\begin{equation*}
    U(h) := U_* + U_h = \setof{w+w_h}{w\in U_*, w_h\in U_h}
\end{equation*}
can be endowed with an ``error norm'' $\norm*{U(h)}{u}$ satisfying
\begin{enumerate}[(i)]
    \item $\norm*{U(h)}{u_h} = \norm*{U_h}{u_h}$ for all $u_h\in U_h$,
    \item $\norm*{U(h)}{u}\leq c \norm*{U}{u}$ for all $u\in U_*$.
\end{enumerate}

The first result concerns non-consistent but conforming approaches and can be used to prove estimates for the error arising from numerical integration; see \cref{thm:implementation:quadrature}. Note that here we do not assume that the discrete bilinear form $a_h$ is well-defined for the continuous solution $u\in U$.
\begin{theorem}[first Strang lemma]\label{thm:strang1}
    Assume that the conditions of \cref{thm:bnb_h} hold and that
    \begin{enumerate}[(i)]
        \item $U_h\subset U=U(h)$ and $V_h\subset V$;
        \item there exists a constant $c_4>0$ independent of $h$ such that
            \begin{equation*}
                |a(u,v_h)| \leq c_4 \norm*{U(h)}{u}\norm*{V_h}{v_h} \qquad\text{for all }u\in U,v_h\in V_h.
            \end{equation*}
    \end{enumerate}
    Then the solutions $u$ and $u_h$ to \eqref{eq:gen_galerkin} and \eqref{eq:gen_galerkin_h}, respectively, satisfy
    \begin{equation*}
        \begin{aligned}
            \norm*{U(h)}{u-u_h}&\leq \frac{1}{c_1} \sup_{v_h\in V_h}\frac{|F(v_h)-F_h(v_h)|}{\norm*{V_h}{v_h}} \\
            \MoveEqLeft[-1]    + \inf_{w_h\in U_h}\left[ \left(1+\frac{c_4}{c_1}\right)\norm*{U(h)}{u-w_h} + \frac{1}{c_1}\sup_{v_h\in V_h}\frac{|a(w_h,v_h)-a_h(w_h,v_h)|}{\norm*{V_h}{v_h}}\right].
        \end{aligned}
    \end{equation*}
\end{theorem}
\begin{proof}
    Let $w_h\in U_h$ be arbitrary. By the triangle inequality and the assumption on the error norm, we have
    \begin{equation*}
        \norm*{U(h)}{u-u_h}\leq  \norm*{U(h)}{u-w_h} +  \norm*{U_h}{u_h-w_h}
    \end{equation*}
    For the second term, we can apply the discrete inf--sup condition to obtain
    \begin{equation*}
        c_1 \norm*{U_h}{u_h-w_h} \leq \sup_{v_h\in V_h} \frac{a_h(u_h-w_h,v_h)}{\norm*{V_h}{v_h}}.
    \end{equation*}
    Using \eqref{eq:gen_galerkin} and \eqref{eq:gen_galerkin_h}, by assumption (i) we can write
    \begin{equation*}
        a_h(u_h-w_h,v_h) = a(u-w_h,v_h) + a(w_h,v_h) - a_h(w_h,v_h) + F_h(v_h) - F(v_h).
    \end{equation*}
    Inserting this into the last estimate and applying the assumption (ii) yields
    \begin{equation*}
        \begin{aligned}
            c_1 \norm*{U(h)}{u_h-w_h} \leq c_4 \norm*{U(h)}{u-w_h} &+ \sup_{v_h\in V_h}\frac{|a(w_h,v_h)-a_h(w_h,v_h)|}{\norm*{V_h}{v_h}}\\
            &+ \sup_{v_h\in V_h}\frac{|F(v_h)-F_h(v_h)|}{\norm*{V_h}{v_h}}.
        \end{aligned}
    \end{equation*}
    The claim now follows by taking the infimum over all $w_h\in U_h$.
\end{proof}

If the bilinear form $a_h$ can be extended to $U(h)\times V_h$ (such that $a_h(u,v_h)$ makes sense), we can dispense with the assumption of conformity.
\begin{theorem}[second Strang lemma]\label{thm:strang2}
    Assume that the conditions of \cref{thm:bnb_h} hold and that
    there exists a constant $c_4>0$ independent of $h$ such that
    \begin{equation*}
        |a_h(u,v_h)| \leq c_4 \norm*{U(h)}{u}\norm*{V_h}{v_h}\qquad\text{for all }u\in U(h), v_h\in V_h.
    \end{equation*}
    Then the solutions $u$ and $u_h$ to \eqref{eq:gen_galerkin} and \eqref{eq:gen_galerkin_h}, respectively, satisfy
    \begin{equation*}
        \norm*{U(h)}{u-u_h}\leq \left(1+\frac{c_4}{c_1}\right)\inf_{w_h\in U_h} \norm*{U(h)}{u-w_h}+ \frac{1}{c_1} \sup_{v_h\in V_h}\frac{|F_h(v_h)-a_h(u,v_h)|}{\norm*{V_h}{v_h}}.
    \end{equation*}
\end{theorem}
\begin{proof}
    We proceed as before. Let $w_h\in U_h$ be given. Then
    \begin{equation*}
        \begin{aligned}
            a_h(u_h-w_h,v_h) &= a_h(u_h-u,v_h) + a_h(u-w_h,v_h)\\
            &= F_h(v_h) - a_h(u,v_h) + a_h(u-w_h,v_h).
        \end{aligned}
    \end{equation*}
    The discrete inf--sup condition and the assumption on $a_h$ then imply
    \begin{equation*}
        c_1 \norm*{U_h}{u_h-w_h} \leq \sup_{v_h\in V_h} \frac{|F_h(v_h)-a_h(u-w_h,v_h)|}{\norm*{V_h}{v_h}} + c_4 \norm*{U(h)}{u-w_h},
    \end{equation*}
    and we again conclude using the triangle inequality and taking the infimum over all $w_h\in U_h$.
\end{proof}

\bigskip

To illustrate the application of the \nameref{thm:strang1}, we consider the effect of quadrature on the Galerkin approximation. For simplicity, we consider for $u,v\in\m{H10}=U=V$ the continuous bilinear form
\begin{equation*}
    a(u,v) = (\alpha\nabla u,\nabla v)
\end{equation*}
with $\alpha\in{W}^{1,\infty}(\Omega)\hookrightarrow\m{c}(\overline\Omega)$, $\alpha_1\geq \alpha(x)\geq \alpha_0 >0$. Let $U_h=V_h\subset \m{H10}=U(h)$ be constructed from triangular Lagrange elements of degree $m$ on an affine-equivalent triangulation $\calT_h$. The discrete bilinear form is then
\begin{equation*}
    a_h(u_h,v_h) = \sum_{K\in\calT_h} \sum_{k=1}^{l_m} w_k \alpha(x_k)\nabla u_h(x_k)\cdot\nabla v_h(x_k),
\end{equation*}
where $w_k>0$ and $x_k$ are Gauß quadrature weights and nodes on each element and $l_m$ is chosen sufficiently large that the quadrature is exact for polynomials of degree up to $2m-1$. Since $\nabla u_h$ is a vector of polynomials of degree $m-1$, this implies
\begin{equation*}
    \begin{aligned}
        \left(\sum_{k=1}^{l_m} w_k \alpha(x_k)\nabla u_h(x_k)\cdot\nabla v_h(x_k)\right)^2 &\leq  \alpha_1^2\left(\sum_{k=1}^{l_m} w_k| \nabla u_h(x_k)|^2\right)\left( \sum_{k=1}^{l_m} w_k |\nabla v_h(x_k)|^2 \right)\\
        &= \alpha_1^2 |\nabla u_h|^2_{\m{h1}(K)}|\nabla v_h|^2_{\m{h1}(K)}
    \end{aligned}
\end{equation*}
since the quadrature is exact for $|\nabla u_h|^2,|\nabla u_h|^2\in P_{2m-2}$. Hence, $a_h$ is continuous on $U_h\times V_h$, since
\begin{equation*}
    |a_h(u_h,v_h)| \leq C \norm{H1}{u_h}\norm{H1}{v_h}.
\end{equation*}
Similarly, $a_h$ is coercive since
\begin{equation*}
    \begin{aligned}
        a_h(u_h,u_h) &\geq \alpha_0  \sum_{K\in\calT_h} \sum_{k=1}^{l_m} w_k |\nabla u_h(x_k)|^2=\alpha_0 |u_h|^2_{\m{H1}}\\ &\geq C \norm{H1}{u_h}^2
    \end{aligned}
\end{equation*}
by positivity of the weights and Poincaré's inequality (\cref{thm:weak:poincare}). As coercivity implies the inf--sup condition, the discrete problem is well-posed by \cref{thm:bnb_h}.

We next derive error estimates for $m=1$ (linear Lagrange elements). Using the \nameref{thm:strang1}, we find that the discretization error is bounded by the approximation error and the quadrature error. For the former, \cref{thm:interp:global} yields
\begin{equation*}
    \inf_{w_h\in V_h}\norm{H1}{u-w_h} \leq C h |u|_{\m{H2}}.
\end{equation*}
For the quadrature error in the bilinear form, we use that for $w_h,v_h\in V_h$, the gradients $\nabla w_h$ and $\nabla v_h$ are constant on each element to write
\begin{equation*}
    \begin{aligned}
        a(w_h,v_h)-a_h(w_h,v_h) &= \sum_{K\in\calT_h} \left(\int_K \alpha\nabla w_h\cdot \nabla v_h\,dx - \sum_{k=1}^{l_m} w_k \alpha(x_k)\nabla w_h(x_k)\cdot \nabla v_h(x_k)\right)\\
        &=  \sum_{K\in\calT_h} \nabla w_h\cdot \nabla v_h\left(\int_K \alpha\,dx - \sum_{k=1}^{l_m}w_k\alpha(x_k)\right).
    \end{aligned}
\end{equation*}
Since for any $m\geq 1$,
\begin{equation*}
    E_K(v) := \int_K v(x)\,dx - \sum_{k=1}^{l_m} w_k v(x_k)
\end{equation*}
is a bounded, sublinear functional on ${W}^{m,\infty}(K)$ which vanishes for all $v\in P_{m-1}\subset P_{2m-1}$, we can apply the \nameref{thm:bramble} on the reference element $\hat K$ to obtain
\begin{equation*}
    |E_{\hat K}(\hat v)| \leq C |\hat v|_{{W}^{m,\infty}(\hat K)}.
\end{equation*}
A scaling argument then yields (noting that the right-hand norm involves the essential supremum over $K$ and thus doesn't scale with $\mathop{\mathrm{vol}}(K)$)
\begin{equation*}
    |E_K(v)| \leq C h_K^{m} \,\mathop{\mathrm{vol}}(K)\, |v|_{{W}^{m,\infty}(K)}.
\end{equation*}
Inserting this for $m=1$ and using again that $\nabla u_h,\nabla v_h$ are constant on each element, we obtain
\begin{equation*}
    \begin{aligned}
        |a(w_h,v_h)-a_h(w_h,v_h)| &\leq \sum_{K\in\calT_h} |\nabla w_h\cdot \nabla v_h|\, |E_K(\alpha)|\\
        &\leq C \sum_{K\in\calT_h} h_K  |\alpha|_{{W}^{1,\infty}(K)} (\mathrm{vol}(K) |\nabla w_h\cdot \nabla v_h|)\\
        &= C \sum_{K\in\calT_h} h_K  |\alpha|_{{W}^{1,\infty}(K)} \int_K |\nabla w_h\cdot \nabla v_h|\,dx\\
        &\leq C h  |\alpha|_{{W}^{1,\infty}(\Omega)} \norm{H1}{w_h}\norm{H1}{v_h}.
    \end{aligned}
\end{equation*}

For the quadrature error on the right-hand side $F_h(v_h)=\sum_{k=1}^{l_m}w_k f(x_k)v_h(x_k)$ for given $f\in{W}^{1,\infty}(\Omega)$, we can proceed similarly (applying the Bramble--Hilbert lemma to $E_K(fv_h)$ and using the product rule and equivalence of norms on $V_h$, followed by a scaling argument) to obtain
\begin{equation*}
    |F(v_h) - F_h(v_h)| \leq C h \, \|f\|_{{W}^{1,\infty}(\Omega)} \norm{H1}{v_h}.
\end{equation*}
Combining these estimates with the \nameref{thm:strang1} yields (with a generic constant $C$ independent of $h$ and using $\norm{H1}{w_h}\leq \norm{H1}{u-w_h}+\norm{H1}{u}$) that
\begin{equation*}
    \begin{aligned}
        \norm{H1}{u-u_h} &\leq C h  \|f\|_{{W}^{1,\infty}(\Omega)} + \inf_{w_h\in V_h}
        \left(C \norm{H1}{u-w_h} + C h  |\alpha|_{{W}^{1,\infty}(\Omega)} \norm{H1}{w_h}\right)\\
        &\leq  C h  \|f\|_{{W}^{1,\infty}(\Omega)} + C \inf_{w_h\in V_h}
        \left(C\norm{H1}{u-w_h} + Ch   \norm{H1}{u-w_h}\right)\\
        \MoveEqLeft[-8] + C h\norm{H1}{u}\\
        &\leq  C h  \|f\|_{{W}^{1,\infty}(\Omega)} + C h^2 |u|_{\m{H2}} + C h \norm{H2}{u}\\
        &\leq C h \left(\|f\|_{{W}^{1,\infty}(\Omega)}+ \|u\|_{\m{H2}}\right),
    \end{aligned}
\end{equation*}
for $h<1$, as claimed in \cref{thm:implementation:quadrature}.

\chapter{Mixed methods}

We now consider variational problems with constraints. Such problems arise, e.g., in the variational formulation of incompressible flow problems (where incompressibility of the solution $u$ can be expressed as the condition $\nabla\cdot u = 0$) or when explicitly enforcing boundary conditions in the weak formulation. To motivate the general problem we will study in this chapter, consider two reflexive Banach spaces $V$ and $M$ and the symmetric and coercive bilinear form $a:V\times V\to \R$. We know (cf. \cref{thm:galerkin:symmetric}) that the solution $u\in V$ to $a(u,v) = F(v)$ for all $v\in V$ is the unique minimizer of $J(v) = \frac12 a(v,v)-F(v)$. If we want $u$ to satisfy the additional condition $b(u,\mu) = 0$ for all $\mu \in M$ and a bilinear form $b:V\times M\to \R$ (e.g., $b(u,\mu) = (\nabla\cdot u,\mu)$), we can introduce the Lagrangian
\begin{equation*}
    L(u,\lambda) = J(u) + b(u,\lambda)
\end{equation*}
and consider the saddle point problem
\begin{equation*}
    \inf_{v\in V}\sup_{\mu\in M} L(v,\mu).
\end{equation*}
Taking the derivative with respect to $v$ and $\mu$, we obtain the (formal) first-order optimality conditions for the saddle point $(u,\lambda)\in V\times M$:
\begin{equation*}
    \left\{\begin{aligned}
            a(u,v) + b(v,\lambda) &= F(v) && \text{ for all } v\in V,\\
            b(u,\mu) &=0  && \text{ for all } \mu\in M.\\
    \end{aligned}\right.
\end{equation*}
This can be made rigorous; the existence of a Lagrange multiplier $\lambda$ however requires some assumptions on $b$. In the next section, we will see that these can be expressed in the form of an inf--sup condition.

\section{Abstract saddle point problems}\label{sec:mixed:saddle}

Let $V$ and $M$ be two reflexive Banach spaces,
\begin{equation*}
    a:V\times V\to \R, \qquad b:V\times M\to\R
\end{equation*}
be two continuous (not necessarily symmetric) bilinear forms, and $f\in V^*$ and $g\in M^*$ be given. Then we search for $(u,\lambda)\in V\times M$ satisfying the saddle point conditions
\begin{equation}\label{eq:mixed:saddle}
    \left\{\begin{aligned}
            a(u,v) + b(v,\lambda) &= \dual{V}{f}{v} && \text{ for all } v\in V,\\
            b(u,\mu) &= \dual{M}{g}{\mu}  && \text{ for all } \mu\in M.\\
    \end{aligned}\right.\tag{$\mathcal{S}$}
\end{equation}

In principle, we can obtain existence and uniqueness of $(u,\lambda)$ by considering \eqref{eq:mixed:saddle} as a variational problem for a bilinear form $c:(V\times M)\times(V\times M)\to \R$ and verifying a suitable inf--sup condition. It is, however, more convenient to express this condition in terms of the original bilinear forms $a$ and $b$. For this purpose, we first reformulate \eqref{eq:mixed:saddle} as an operator equation by introducing the operators
\begin{equation*}
    \begin{aligned}
        &A: V\to V^*,\qquad & \dual{V}{Au}{v} &= a(u,v) \quad \text{ for all } v\in V,\\
        &B: V\to M^*,\qquad & \dual{M}{Bu}{\mu} &= b(u,\mu) \quad \text{ for all } \mu\in M,\\
        &B^*: M\to V^*,\qquad & \dual{V}{B^*\lambda}{v} &= b(v,\lambda) \quad \text{ for all } v\in V.
    \end{aligned}
\end{equation*}
Then \eqref{eq:mixed:saddle} is equivalent to
\begin{equation}\label{eq:mixed:saddle_op}
    \left\{    \begin{aligned}
            Au + B^*\lambda &= f \quad\text{ in } V^*,\\
            Bu &= g \quad\text{ in } M^*.
    \end{aligned}\right.
\end{equation}
From this, we can see the following: If $B$ were invertible, the existence and uniqueness first of $u$ and then of $\lambda$ would follow immediately. In the (more realistic) case that $B$ has a nontrivial null space
\begin{equation*}
    \ker B = \setof{x\in V}{b(x,\mu) = 0 \text{ for all } \mu\in M}
\end{equation*}
(e.g., constant functions in the case $Bu = \nabla\cdot u$), we have to require that $A$ is injective on it to obtain a unique $u$. Existence of $\lambda$ then follows if $B^*$ is surjective. To verify these conditions, we follow the general approach of the \nameref{thm:BNB} theorem.

\begin{theorem}[Brezzi splitting theorem]\label{thm:brezzi}
    Assume that
    \begin{enumerate}[(i)]
        \item $a:V\times V\to\R$ satisfies the conditions of \cref{thm:BNB} for $U=V=\ker B$ and
        \item $b:V\times M\to\R$ satisfies for $\beta >0$ the condition
            \begin{equation}\label{eq:LBB}
                \inf_{\mu\in M}\sup_{v\in V} \frac{b(v,\mu)}{\norm*{V}{v}\norm*{M}{\mu}} \geq \beta.
            \end{equation}
    \end{enumerate}
    Then there exists a unique solution $(u,\lambda)\in V\times M$ to \eqref{eq:mixed:saddle} satisfying
    \begin{equation*}
        \norm*{V}{u}+\norm*{M}{\lambda}\leq C(\norm*{V^*}{f}+\norm*{M^*}{g}).
    \end{equation*}
\end{theorem}
Condition (ii) is an inf--sup condition for $B^*$ (since the infimum is taken over the test functions $\mu$) and is known as the \emph{Ladyzhenskaya--Babu\v{s}ka--Brezzi} (LBB) condition.
Note that $a$ only has to satisfy an inf--sup condition on the null space of $B$, not on all of $V$, which is crucial in many applications.
\begin{proof}
    First, by following the proof of \cref{thm:BNB}, we deduce that the LBB condition implies that $B^*$ has closed range and therefore
    \begin{equation*}
        \ran B^* = (\ker B)^\bot = \setof{v^*\in V^*}{\dual{V}{v^*}{v} = 0 \text{ for all } v\in \ker B}.
    \end{equation*}
    In addition, $B^*$ is injective on $M$ with
    \begin{equation}\label{eq:LBB:bound_badj}
        \beta \norm*{M}{\mu} \leq \norm*{V^*}{B^*\mu}
    \end{equation}
    holds for all $\mu\in M$.
    By the closed range theorem, $B$ has closed range as well and hence is surjective on $\ran B = (\ker B^*)^\bot = (\{0\})^\bot = M^*$.
    Thus for any $g\in M^*$ there exists a $\tilde u_g\in V$ satisfying $B\tilde u_g = g$. Since $B$ is not injective, $\tilde u_g$ is not unique, nor can its norm necessarily be bounded by that of $g$ (since one can add to $\tilde u_g$ any element in $\ker B$). However, among the possible solutions, we can find (at least) one that is bounded by applying the Hahn--Banach extension theorem.
    Let $v^*\in (\ker B)^\bot = \ran B^* \subset V^*$ be given. By the above, there then exists a unique $\lambda\in M$ such that $B^*\lambda = v^*$ and $\norm*{M}{\lambda}\leq\frac1\beta \norm*{V^*}{v^*}$.
    Since $V\subset (V^*)^*$, we can write
    \begin{equation*}
        \langle \tilde u_g, v^*\rangle_{(V^*)^*,V^*} = \dual{V}{B^*\lambda}{\tilde u_g} = \dual{M}{g}{\lambda}
        \leq \norm*{M^*}{g}\norm*{M}{\lambda}
        \leq \frac1\beta \norm*{M^*}{g} \norm*{V^*}{v^*}.
    \end{equation*}
    This implies that $\tilde u_g$ is bounded as a linear functional on $(\ker B)^\bot\subset V^*$, and in particular that $\norm*{((\ker B)^\bot)^*}{\tilde u_g}\leq \frac1\beta\norm*{M^*}{g}$. The Hahn--Banach extension theorem thus yields existence of a $u_g\in (V^*)^* = V$ with $u_g = \tilde u_g$ on $(\ker B)^\bot=\ran B^*$ and
    \begin{equation}\label{eq:LBB:est1}
        \norm*{V}{u_g} = \norm*{((\ker B)^\bot)^*}{\tilde u_g} \leq \frac1\beta\norm*{M^*}{g}.
    \end{equation}
    In addition, $Bu_g = g$ as well, since for all $\mu\in M$, we have that $B^*\mu \in \ran B^*=(\ker B)^\bot$ and hence by the extension property that
    \begin{equation*}
        \dual{M}{Bu_g}{\mu} = \dual{V}{B^*\mu}{u_g} = \dual{V}{B^*\mu}{\tilde u_g} =
        \dual{M}{g}{\mu}.
    \end{equation*}

    \bigskip

    Due to condition (i), $A$ is an isomorphism from $\ker B$ to $(\ker B)^*$. Considering $f-Au_g\in V^*$ as a bounded linear form on $\ker B\subset V$, we can thus find a unique $u_f \in \ker B$ satisfying
    \begin{equation}\label{eq:LBB:eq1}
        Au_f = f - Au_g \quad\text{in }(\ker B)^*
    \end{equation}
    (but not necessarily in $V^*$!) and
    \begin{equation}\label{eq:LBB:est2}
        \norm*{V}{u_f} \leq \frac1\alpha \norm*{(\ker B)^*}{f-Au_g} \leq \frac{1}\alpha(\norm*{V^*}{f} + C\norm*{V}{u_g}),
    \end{equation}
    where $\alpha>0$ and $C>0$ are the constants in the inf--sup and continuity conditions for $a$, respectively, and we have used that $f\in V^*$ such that $\norm*{(\ker B)^*}{f}\leq \norm*{V^*}{f}$ by definition of the dual norm and the fact that $\ker B\subset V$ is endowed with the same norm as $V$.

    Now set $u = u_f + u_g \in V$ and consider $f-Au\in V^*$, which due to \eqref{eq:LBB:eq1} satisfies for all $v\in\ker B$ that
    $\dual{(\ker B)}{f - Au}{v} = 0$.
    This implies that $f-Au \in(\ker B)^\bot$, and the surjectivity of $B^*$ on $(\ker B)^\bot=\ran B^*$ yields the existence of a $\lambda\in M$ satisfying
    \begin{equation*}
        B^* \lambda = f - Au  \quad\text{in } V^*
    \end{equation*}
    and
    \begin{equation}\label{eq:LBB:est3}
        \norm*{M}{\lambda} \leq \frac1\beta(\norm*{V^*}{f} + C \norm*{V}{u}).
    \end{equation}
    Since $u_f\in \ker B$ and hence
    \begin{equation*}
        Bu = Bu_g = g \quad\text{in } M^*,
    \end{equation*}
    we have thus found $(u,\lambda)\in V\times M$ satisfying \eqref{eq:mixed:saddle}, and the claimed estimate follows by combining \eqref{eq:LBB:est1}, \eqref{eq:LBB:est2} and \eqref{eq:LBB:est3}.

    To show uniqueness, consider the difference $(\bar u,\bar\lambda)$ of two solutions $(u_1,\lambda_1)$ and $(u_2,\lambda_2)$, which solves the homogeneous problem \eqref{eq:mixed:saddle_op} with $f =0 $ and $g=0$, i.e.,
    \begin{equation*}
        \left\{
            \begin{aligned}
                A\bar u + B^*\bar \lambda &= 0 \quad\text{ in } V^*,\\
                B\bar u &= 0 \quad\text{ in } M^*.
            \end{aligned}
        \right.
    \end{equation*}
    The second equation yields $\bar u\in\ker B$, and the inf--sup condition for $A$ on $\ker B$ implies
    \begin{equation*}
        \alpha \norm*{V}{\bar u }^2 \leq  a(\bar u, \bar u) = a(\bar u,\bar u) + b(\bar u,\bar\lambda) = 0 .
    \end{equation*}
    Since $\bar u = 0$, it follows from the first equation that $B^*\bar\lambda = 0$ and thus from the injectivity of $B^*$ that $\bar \lambda = 0$.
\end{proof}

\section{Galerkin approximation of saddle point problems}

For the Galerkin approximation of \eqref{eq:mixed:saddle}, we again choose subspaces $V_h\subset V$ and $M_h \subset M$ and look for $(u_h,\lambda_h)\in V_h\times M_h$ satisfying
\begin{equation}\label{eq:mixed:saddle_h}
    \left\{\begin{aligned}
            a(u_h,v_h) + b(v_h,\lambda_h) &= \dual{V}{f}{v_h} && \text{ for all } v_h\in V_h,\\
            b(u_h,\mu_h) &= \dual{M}{g}{\mu_h}  && \text{ for all } \mu_h\in M_h.\\
    \end{aligned}\right.\tag{$\mathcal{S}_h$}
\end{equation}
This approach is called a \emph{mixed finite element method}. It is clear that the choice of $V_h$ and of $M_h$ cannot be independent of each other but must satisfy a compatibility condition similar to that in \cref{thm:brezzi}. For its statement, we define the operator $B_h:V_h\to M_h^*$ analogously to $B$.
\begin{theorem}\label{thm:brezzi_h}
    Assume there exist constants $\alpha_h,\beta_h>0$ such that
    \begin{align}
        \inf_{u_h\in \ker B_h}\sup_{v_h\in \ker B_h} \frac{a(u_h,v_h)}{\norm*{V}{u_h}\norm*{V}{v_h}} &\geq \alpha_h\label{eq:LBB_A_h},\\
        \inf_{\mu_h\in M_h}\sup_{v_h\in V_h} \frac{b(v_h,\mu_h)}{\norm*{V}{v_h}\norm*{M}{\mu_h}} &\geq \beta_h.\label{eq:LBB_h}
    \end{align}
    Then there exists a unique solution $(u_h,\lambda_h)\in V_h\times M_h$ to \eqref{eq:mixed:saddle_h} satisfying
    \begin{equation*}
        \norm*{V}{u_h}+\norm*{M}{\lambda_h}\leq C(h)(\norm*{V^*}{f}+\norm*{M^*}{g}).
    \end{equation*}
\end{theorem}
\begin{proof}
    The claim follows immediately from \cref{thm:brezzi} and the fact that in finite dimensions, the inf--sup condition for $a$ is sufficient to apply the discrete BNB \cref{thm:bnb_h}.
\end{proof}
Note that in general, this is a non-conforming approach since even for $V_h\subset V$ and $M_h\subset M$, as we do not necessarily have that $B_h$ is the restriction of $B$ to $V_h$ (i.e., $B(V_h) \not\subset M_h^*$) or that $\ker B_h$ is a subspace of $\ker B$. Hence, the discrete inf--sup conditions do not follow from their continuous counterparts. However, if the subspace $V_h$ is chosen suitably, it is possible to deduce the discrete LBB condition from the continuous one.

\begin{theorem}[Fortin criterion]\label{thm:mixed:fortin}
    Assume that the LBB condition \eqref{eq:LBB} is satisfied. Then the discrete LBB condition \eqref{eq:LBB_h} is satisfied if and only if there exists a linear operator $\Pi_h:V\to V_h$ such that
    \begin{equation*}
        b(\Pi_h v,\mu_h) = b(v,\mu_h) \quad \text{for all }\mu_h\in M_h,
    \end{equation*}
    and there exists a $\gamma_h>0$ such that
    \begin{equation*}
        \norm*{V}{\Pi_h v}\leq \gamma_h \norm*{V}{v} \quad \text{for all } v\in V.
    \end{equation*}
\end{theorem}
\begin{proof}
    Assume that such a $\Pi_h$ exists. Since $\ran \Pi_h\subset V_h$, we have for all $\mu_h\in M_h\subset M$ that
    \begin{equation*}
        \sup_{v_h\in V_h} \frac{b(v_h,\mu_h)}{\norm*{V}{v_h}} \geq
        \sup_{v\in V} \frac{b(\Pi_h v,\mu_h)}{\norm*{V}{\Pi_h v}} \geq
        \sup_{v\in V} \frac{b(v,\mu_h)}{\gamma_h\norm*{V}{v}} \geq
        \frac\beta{\gamma_h}\norm*{M}{\mu_h},
    \end{equation*}
    which implies the discrete LBB condition. Conversely, if the discrete LBB condition holds, the operator $B_h:V_h\to M_h^*$ as defined above is surjective and has a continuous right inverse. Furthermore, for any $v\in V$ we can consider $Bv\in M^*$ as a linear functional on $M_h\subset M$ only. Hence for any $v\in V$, there exists a $\pi_h\in V_h$ such that $B_h (\pi_h) = B v|_{M_h} \in M_h^*$, i.e., $b(\pi_h,\mu_h) = b(v,\mu_h)$ for all $\mu_h\in M_h\subset M$, and
    \begin{equation*}
        \beta_h\norm*{V}{\pi_h} \leq \norm*{M_h^*}{B v} \leq C \norm*{V}{v}.
    \end{equation*}
    We thus obtain the desired operator by defining $\Pi_h$ as the (linear) mapping $v\mapsto \pi_h$.
\end{proof}
The operator $\Pi_h$ is called \emph{Fortin projector}. From the proof, we can see that the discrete LBB condition holds with a constant independent of $h$ if and only if the Fortin projector is uniformly bounded in $h$ (i.e., if $\gamma_h \equiv \gamma$).

A priori error estimates can be obtained using the following variant of Céa's lemma.
\begin{theorem}\label{thm:mixed:cea}
    Assume the conditions of \cref{thm:brezzi_h} are satisfied. Let $(u,\lambda)\in V\times M$ and $(u_h,\lambda_h)\in V_h\times M_h$ be the solutions to \eqref{eq:mixed:saddle} and \eqref{eq:mixed:saddle_h}, respectively. Then there exists a constant $C(h)>0$ such that
    \begin{equation*}
        \norm*{V}{u-u_h}+\norm*{M}{\lambda-\lambda_h}\leq C(h) \left(\inf_{w_h\in V_h}\norm*{V}{u-w_h}+\inf_{\mu_h\in M_h}\norm*{M}{\lambda-\mu_h}\right).
    \end{equation*}
\end{theorem}
\begin{proof}
    For arbitrary $w_h\in V_h$, consider the restriction of $B(u-w_h)\in M^*$ to $M_h$
    Due to the discrete LBB condition, the operator $B_h:V_h\to M_h^*$ is surjective and has a continuous right inverse. Hence, there exists $r_h\in V_h$ satisfying $B_h r_h = B(u-w_h)|_{M_h}\in M_h^*$, i.e.,
    \begin{equation*}
        b(r_h,\mu_h) = b(u-w_h,\mu_h) \quad \text{ for all } \mu_h \in M_h \subset M
    \end{equation*}
    and
    \begin{equation*}
        \beta_h\norm*{V}{r_h}\leq C \norm*{V}{u-w_h}.
    \end{equation*}
    Furthermore, $z_h:=r_h+w_h$ satisfies
    \begin{equation*}
        b(z_h,\mu_h) = b(u,\mu_h) = \dual{M}{g}{\mu_h} = b(u_h,\mu_h)\quad \text{ for all } \mu_h \in M_h\subset M,
    \end{equation*}
    which implies that $u_h - z_h\in \ker B_h$. The discrete inf--sup condition \eqref{eq:LBB_A_h} thus yields
    \begin{equation}\label{eq:mixed:cea0}
        \begin{aligned}[t]
            \alpha_h\norm*{V}{u_h-z_h}&\leq \sup_{v_h\in \ker B_h} \frac{a(u_h-z_h,v_h)}{\norm*{V}{v_h}}\\
            &= \sup_{v_h\in \ker B_h} \frac{a(u_h-u,v_h)+a(u-z_h,v_h)}{\norm*{V}{v_h}}\\
            &= \sup_{v_h\in \ker B_h} \frac{b(v_h,\lambda-\lambda_h)+a(u-z_h,v_h)}{\norm*{V}{v_h}},
        \end{aligned}
    \end{equation}
    by taking the difference of the first equations of \eqref{eq:mixed:saddle} and \eqref{eq:mixed:saddle_h}. For any $v_h\in \ker B_h$ and $\mu_h\in M_h$, we have
    \begin{equation*}
        b(v_h,\lambda_h) = 0 = b(v_h,\mu_h)
    \end{equation*}
    and hence from the continuity of $a$ and $b$ that
    \begin{equation*}
        \alpha_h\norm*{V}{u_h-z_h} \leq C(\norm*{V}{u-z_h} + \norm*{M}{\lambda-\mu_h})
    \end{equation*}
    for arbitrary $\mu_h\in M_h$. Using the triangle inequality, we thus obtain
    \begin{align}
        \norm*{V}{u-u_h} &\begin{aligned}[t]
            &\leq \norm*{V}{u-z_h}+\norm*{V}{z_h-u_h} \\
            &\leq \left(1+\frac{C}{\alpha_h}\right)\norm*{V}{u-z_h} + \frac{C}{\alpha_h}\norm*{M}{\lambda-\mu_h}
        \end{aligned}\label{eq:mixed:cea1}
        \intertext{and, by definition of $z_h$}
        \norm*{V}{u-z_h} &\leq  \norm*{V}{u-w_h}+\norm*{V}{r_h} \leq  \left(1+\frac{C}{\beta_h}\right)\norm*{V}{u-w_h}.\label{eq:mixed:cea2}
    \end{align}

    To estimate $\norm*{M}{\lambda-\lambda_h}$, we again use that for all $w_h\in V_h$ and $\mu_h\in M_h$,
    \begin{equation*}
        a(u-u_h,w_h) = b(w_h,\lambda-\lambda_h) = b(w_h,\lambda-\mu_h) + b(w_h,\mu_h-\lambda_h).
    \end{equation*}
    After rearrangement, the discrete LBB condition and the continuity of $a$ and $b$ thus yield
    \begin{equation*}
        \beta_h \norm*{M}{\lambda_h - \mu_h} \leq C(\norm*{V}{u-u_h} + \norm*{M}{\lambda-\mu_h}).
    \end{equation*}
    Applying the triangle inequality again, we obtain
    \begin{equation}\label{eq:mixed:cea3}
        \begin{aligned}[t]
            \norm*{M}{\lambda-\lambda_h} &\leq \norm*{M}{\lambda-\mu_h}+\norm*{M}{\lambda_h-\mu_h} \\
            &\leq \left(1+\frac{C}{\beta_h}\right)\norm*{M}{\lambda-\mu_h} + \frac{C}{\beta_h}\norm*{V}{u-u_h}.
        \end{aligned}
    \end{equation}
    Combining \eqref{eq:mixed:cea1}, \eqref{eq:mixed:cea2}, and \eqref{eq:mixed:cea3} and taking the infimum over all $w_h\in V_h$ and $\mu_h\in M_h$ yields the claimed estimate.
\end{proof}

If $\ker B_h\subset \ker B$ (i.e., $b(v_h,\mu_h) = 0$ for all $\mu_h\in M_h$ implies $b(v_h,\mu)=0$ for all $\mu \in M$), we can obtain an independent estimate for $u$.
\begin{cor}\label{cor:mixed:cea}
    If\/ $\ker B_h\subset \ker B$, then there exists a constant $C(h)>0$ such that
    \begin{equation*}
        \norm*{V}{u-u_h}\leq C(h) \inf_{w_h\in V_h}\norm*{V}{u-w_h}.
    \end{equation*}
\end{cor}
\begin{proof}
    The assumption in particular implies that $b(v_h,\lambda-\lambda_h)=0$ for all $v_h\in \ker B_h$, and hence \eqref{eq:mixed:cea0} yields
    \begin{equation*}
        \alpha_h\norm*{V}{u_h-z_h} \leq C\norm*{V}{u-z_h}.
    \end{equation*}
    Continuing as above, we obtain the claimed estimate.
\end{proof}

\section{Mixed methods for the Poisson equation}

The classical application of mixed finite element methods is the Stokes equation,\footnote{see, e.g., \cite[Chapter III.6]{Braess:2007}, \cite[Chapter 4]{Ern:2004}} which describes the flow of an incompressible fluid. Here, we want to illustrate the theory using a very simple example. Consider the Poisson equation $-\Delta u = f$ on $\Omega\subset\R^n$ with homogeneous Dirichlet conditions. If we introduce $\sigma = \nabla u\in\m{L2}^n$, we can write this equation equivalently as
\begin{equation}\label{eq:mixed:poisson}
    \left\{\begin{aligned}
            \nabla u -\sigma &= 0,\\
            - \nabla\cdot \sigma &= f.
    \end{aligned}\right.
\end{equation}
This system can be formulated in variational form in two different ways, called \emph{primal} and \emph{dual} approach, respectively.

\paragraph{Primal mixed method}

The primal approach consists in (formally) integrating by parts in the second equation of \eqref{eq:mixed:poisson} and looking for $(\sigma,u)\in\m{L2}^n\times\m{H10}$ satisfying
\begin{equation}\label{eq:mixed:poisson_primal}
    \left\{\begin{aligned}
            (\sigma,\tau) - (\tau,\nabla u) &= 0\quad &&\text{ for all } \tau \in\m{L2}^n,\\
            -(\sigma, \nabla v) &= -(f,v) \quad &&\text{ for all } v \in\m{H10}.
    \end{aligned}\right.
\end{equation}
This fits into the abstract framework of \cref{sec:mixed:saddle} by setting $V:=\m{L2}^n$, $M:=\m{H10}$,
\begin{equation*}
    a(\sigma,\tau) = (\sigma,\tau),\qquad b(\sigma,v) = -(\sigma,\nabla v).
\end{equation*}
Clearly, $a$ is coercive on the whole space $V$ with constant $\alpha =1$. To verify the LBB condition, we insert $\tau = -\nabla v\in\m{L2}^n=V$ for given $v\in\m{H10}=M$ in
\begin{equation*}
    \sup_{\tau\in V}\frac{b(\tau,v)}{\norm*{V}{\tau}}=   \sup_{\tau\in V}\frac{-(\tau,\nabla v)}{\norm*{\m{L2}^n}{\tau}} \geq
    \frac{(\nabla v,\nabla v)}{\norm*{\m{L2}^n}{\nabla v}} = |v|_{\m{H1}} \geq c_\Omega^{-1} \norm*{M}{v}
\end{equation*}
using the Poincaré inequality (\cref{thm:weak:poincare}), i.e., the LBB condition with $\beta=c_\Omega^{-1}$. \Cref{thm:brezzi} thus yields the existence and uniqueness of the solution $(\sigma, u)$ to \eqref{eq:mixed:poisson_primal}.

To obtain a stable mixed finite element method, we take a shape-regular affine triangulation $\calT_h$ of $\Omega$ and set for $k\geq 1$
\begin{align*}
    V_h &:=\setof{\tau_h\in\m{L2}^n}{\tau_h|_K \in P_{k-1}(K)^n \text{ for all } K\in\calT_h},\\
    M_h &:=\setof{v_h\in\m{c}(\Omega)}{v_h|_K \in P_k(K) \text{ for all } K\in\calT_h}\cap M.
\end{align*}
Since $V_h\subset V$, the coercivity of $a$ on $V_h$ follows as above with constant $\alpha_h = \alpha$. Furthermore, it is easy to verify that $\nabla M_h \subset V_h$, e.g., the gradient of any piecewise affine continuous function is piecewise constant.  Hence, the $\m{L2}^n$ projection from $V$ on $V_h$ (which is continuous with norm $\gamma_h=1$) verifies the Fortin criterion: If $\Pi_h \sigma\in V_h$ satisfies $(\Pi_h \sigma-\sigma, \tau_h) = 0$ for all $\tau_h\in V_h$ and given $\sigma\in V$, then
\begin{equation*}
    b(\Pi_h \sigma,v_h) = -(\Pi_h \sigma,\nabla v_h) = -(\sigma,\nabla v_h) = b(\sigma,v_h) \quad\text{ for all }v_h\in M_h
\end{equation*}
since $\nabla v_h\in V_h$.  \Cref{thm:mixed:fortin} therefore yields the discrete LBB condition with constant $\beta_h=\beta$ independent of $h$, and we obtain existence of and (from \cref{thm:mixed:cea} combined with the usual interpolation error estimates) a priori estimates for the mixed finite element discretization of \eqref{eq:mixed:poisson_primal} (which coincide with those from \cref{sec:error:apriori}).

\paragraph{Dual mixed method}

Instead of integrating by parts in the second equation, we can formally integrate by parts in the first equation of \eqref{eq:mixed:poisson}. To make this well-defined, we set
\begin{equation*}
    \m{Hdiv}:=\setof{\tau\in\m{L2}^n}{\div \tau\in\m{L2}},
\end{equation*}
endowed with the \emph{graph norm}
\begin{equation*}
    \norm{Hdiv}{\tau}^2 := \norm*{\m{L2}^n}{\tau}^2 + \norm{L2}{\div \tau}^2.
\end{equation*}
Since $\m{Cinf}{}^n$ is dense in $\m{L2}^n\supset\m{Hdiv}$, one can show that $\tau\in\m{Hdiv}$ has a well-defined \emph{normal trace} $(\tau|_{\partial\Omega}\cdot \nu) \in H^{-1/2}(\partial\Omega)$, and that for any $\tau\in\m{Hdiv}$ and $w\in\m{H1}$ the integration by parts formula
\begin{equation}\label{eq:mixed:partint}
    \int_\Omega (\div \tau) w \,dx + \int_\Omega \tau\cdot\nabla w\,dx = \int_{\partial\Omega} (\tau\cdot\nu) w\, dx
\end{equation}
holds.\footnote{e.g., \cite[Lemma 2.1.1]{Brezzi:2013}}
Similarly to \cref{thm:weak:piecewise}, one can show that for a partition $\{\Omega_j\}_{j\in J}$ of $\Omega$,
\begin{equation*}
    \setof{\tau\in\m{L2}^n}{\tau|_{\Omega_j} \in \m{h1}(\Omega_j) \text{ and } \tau|_{\Omega_j}\cdot\nu = \tau|_{\Omega_i} \cdot\nu\text{ on all } \bar\Omega_j\cap\bar\Omega_i \neq \emptyset} \subset \m{Hdiv}
\end{equation*}
holds, i.e., piecewise differentiable functions with continuous normal traces across elements are in $\m{Hdiv}$. This will be important for constructing conforming approximations of $\m{Hdiv}$.

After integrating by parts in \eqref{eq:mixed:poisson} and using that $u|_{\partial\Omega} = 0$, we are therefore looking for $(\sigma,u)\in\m{Hdiv}\times\m{L2}$ satisfying
\begin{equation}\label{eq:mixed:poisson_dual}
    \left\{\begin{aligned}
            (\sigma,\tau) + (\div\tau,u) &= 0\quad &&\text{ for all } \tau \in\m{Hdiv},\\
            (\div\sigma, v) &= -(f,v) \quad &&\text{ for all } v \in\m{L2}.
    \end{aligned}\right.
\end{equation}
(Note that in contrast to the standard -- and primal -- formulation, the Dirichlet condition appears here as the natural boundary condition.)
This formulation fits into the abstract framework of \cref{sec:mixed:saddle} by setting $V:=\m{Hdiv}$, $M:=\m{L2}$,
\begin{equation*}
    a(\sigma,\tau) = (\sigma,\tau),\qquad b(\sigma,v) = (\div\sigma, v).
\end{equation*}
Boundedness of $a$ and $b$ follows directly from the Cauchy--Schwarz inequality.
Now we note that
\begin{equation*}
    \ker B = \setof{\tau\in\m{Hdiv}}{(\div \tau,v) = 0 \text{ for all } v\in\m{L2}}.
\end{equation*}
Since $\div \tau \in\m{L2}$ and thus $\norm{L2}{\div \tau}^2 = 0$
for all $\tau \in\ker B\subset\m{Hdiv}$, this implies
\begin{equation*}
    a(\tau,\tau) = \norm*{\m{L2}^n}{\tau}^2 = \norm{Hdiv}{\tau}^2\qquad\text{for all }v\in \ker B,
\end{equation*}
yielding coercivity of $a$ with constant $\alpha=1$.
For verification of the LBB condition, we make use of the following lemma showing surjectivity of $B$ on $M$. For simplicity, we assume from here on that $\Omega$ either has a $C^1$ boundary or is convex.
\begin{lemma}\label{lem:mixed:div_surjective}
    For any $f\in\m{L2}$, there exists a function $\tau \in \m{H1}^n$ with $\div \tau = f$ and $\norm*{\m{H1}^n}{\tau}\leq C\norm{L2}{f}$.
\end{lemma}
\begin{proof}
    Due to the regularity of $\Omega$, we can apply \cref{thm:weak:regularity:smooth} or \cref{thm:weak:regularity:convex} to obtain for given $f\in\m{L2}$ a solution $u\in\m{H2}\cap\m{H10}$ to the Poisson equation
    \begin{equation*}
        (\nabla u,\nabla v) = (f,v) \qquad\text{for all  } v\in\m{H10}
    \end{equation*}
    satisfying $\norm{H2}{u}\leq C\norm{L2}{f}$.
    Now set $\tau := -\nabla u \in\m{H1}^n$ and observe that
    \begin{equation*}
        (f,v) = -(\tau,\nabla v) \qquad\text{for all  } v\in\m{H10},
    \end{equation*}
    and thus $f = \div\tau$ by definition of the weak derivative. The a priori bound on $\tau$ then follows from the fact that
    $ \norm*{\m{H1}^n}{\nabla u} \leq \norm{H2}{u}$.
\end{proof}
Using this lemma and the inclusion $\m{H1}^n\subset\m{Hdiv}$, we immediately obtain for any $v\in M$ and corresponding $\tau_v$ with $\div\tau_v = v$ that
\begin{equation*}
    \sup_{\tau\in V}\frac{b(\tau,v)}{\norm*{V}{\tau}} = \sup_{\tau\in V}\frac{(\div\tau,v)}{\norm{Hdiv}{\tau}}
    \geq \frac{(\div\tau_v,v)}{\norm{Hdiv}{\tau_v}}
    \geq \frac{(v,v)}{C\norm{L2}{v}} = \frac1C \norm{L2}{v},
\end{equation*}
which verifies the LBB condition for $\beta = C^{-1}$. From \cref{thm:brezzi} we thus obtain existence of a unique solution $(\sigma,u)\in V\times M$ to \eqref{eq:mixed:poisson_dual} as well as the estimate
\begin{equation*}
    \norm{Hdiv}{\sigma} +\norm{L2}{u} \leq C\norm{L2}{f}.
\end{equation*}
Although this initially yields only a solution $u\in\m{L2}$, one can then use the first equation of \eqref{eq:mixed:poisson_dual} to show that $u$ has a weak derivative and (using integration by parts) satisfies the boundary conditions; i.e., $u\in\m{H10}$ as expected.

\bigskip

We now construct conforming finite element discretizations of $V$ and $M$. Let $\calT_h$ be a shape-regular affine triangulation of $\Omega\subset\R^n$. For $M=\m{L2}$, we again take piecewise (discontinuous) polynomials of degree $k\geq 0$, i.e.,
\begin{equation*}
    M_h =\setof{v_h\in\m{L2}}{v_h|_K \in P_{k}(K) \text{ for all } K\in\calT_h}.
\end{equation*}
For $V=\m{Hdiv}$, we construct a space $V_h$ of piecewise polynomials on the same triangulation that satisfy the two key properties of $V$: Functions $\tau_h\in V_h$ have continuous normal traces across elements, and the divergence is surjective from $V_h$ to $M_h$. One possible choice is
\begin{equation*}
    V_h =\setof{\tau_h\in\m{Hdiv}}{\tau_h|_K\in RT_k(K) \text{ for all }K\in\calT_h},
\end{equation*}
with
\begin{equation*}
    \begin{aligned}
        RT_k(K) &=  P_k(K)^n + x P_k(K) :=\setof{p_1+ p_2\,x}{p_1\in P_k(K)^n, p_2\in P_k(K)}\\
        &= P_k(K)^n \oplus x P_k^0(K),
    \end{aligned}
\end{equation*}
where
\begin{equation*}
    P_k^0(K) =\setof{\sum_{|\alpha| = k}c_\alpha x^\alpha}{c_\alpha\in\R}
\end{equation*}
is the space of \emph{homogeneous} polynomials of degree $k$ (which is chosen in order to have a unique representation). This construction yields the following properties, which guarantee a conforming $\m{Hdiv}$ discretization.
\begin{lemma}\label{lem:mixed:rt_poly}
    For $\tau_h\in RT_k(K)$, we have
    \begin{enumerate}[(i)]
        \item $\div \tau_h\in P_k(K)$ and
        \item $\tau_h|_F\cdot\nu_F\in P_k(F)$ for every $F\subset\partial K$.
    \end{enumerate}
\end{lemma}
The verification is a straightforward computation (recalling that $x\cdot \nu_F(x)$ is constant for every $x\in F$). It remains to specify the degrees of freedom, of which we need
\begin{equation*}
    \dim RT_k(K) = \begin{cases}
        (k+1)(k+3) &\text{for } n=2,\\
        \tfrac12(k+1)(k+2)(k+4) &\text{for } n=3.
    \end{cases}
\end{equation*}
In order to achieve a $\m{Hdiv}$-conforming discretization, we take
\begin{align*}
    N_{i,j}(\tau) &=  \int_{F_i} (\tau\cdot\nu_i) q_{ij}\,ds,\\
    \intertext{where the $q_{ij}$ are a basis of $P_k(F_i)$, $i=1,\dots,n+1$, and if $k\geq 1$,}
    N_{0,j}(\tau) &= \int_K \tau \cdot  q_j\,dx,
\end{align*}
where the $q_{j}$ are a basis of $P_{k-1}(K)^n$.
To show that $(K,RT_k(K),\{N_{ij}\}_{i,j})$ defines a finite element -- called the \emph{Raviart--Thomas element} -- we need to determine whether these conditions form a basis of $RT_k(K)^*$, which we can do via \cref{thm:elements:basis}.
\begin{lemma}\label{lem:mixed:rt_unisolvent}
    If $\tau_h\in RT_k(K)$ satisfies $N_{i,j}(\tau_h)=0$ for all $i,j$, then $\tau_h = 0$.
\end{lemma}
\begin{proof}
    First, observe that $N_{i,j}(\tau_h)=0$ for some $i$ and all $j$ implies that
    \begin{equation*}
        \int_{F_i} (\tau_h\cdot\nu_i) q_k\,ds = 0 \quad\text{for all } q_k\in P_k(F_i),
    \end{equation*}
    and since $\tau_h|_{F_i}\cdot\nu_i\in P_k(F_i)$ by \cref{lem:mixed:rt_poly}\,(ii), $\tau_h|_F\cdot\nu_F = 0$ on each face $F$ of $K$. Similarly, we have that
    \begin{equation}\label{eq:mixed:rt_uni_K}
        \int_{K} \tau_h\cdot \tilde q_k\,dx = 0 \quad\text{for all } \tilde q_k\in P_{k-1}(K)^n,
    \end{equation}
    and hence for all $q_k\in P_k(K)$ that
    \begin{equation*}
        \int_K \div\tau_h q_{k}\,dx = -\int_K\tau_h\nabla q_k \,dx + \int_{\partial K} \tau_h\cdot\nu q_k\,ds = 0
    \end{equation*}
    since $\nabla q_k=0$ for $k=0$ and $\nabla q_k \in P_{k-1}(K)^n$ for $k\geq 1$. As $\div\tau_h\in P_{k}(K)$ by \cref{lem:mixed:rt_poly}\,(i), this yields $\div\tau_h = 0$ on $K$.

    By construction, $\tau_h = p_1 + x p_2$ for some $p_1 \in P_k(K)^n$ and $p_2\in P^0_k(K)$. First, it is straightforward to verify that a homogeneous polynomial $p\in P^0_k(K)$ satisfies $x\cdot \nabla p = k p$ (this is known as \emph{Euler's theorem for homogeneous functions}). Hence by the product rule,
    \begin{equation*}
        0 = \div(\tau_h) = \div p_1 + (n+k) p_2.
    \end{equation*}
    Since $\div p_1$ for $p_1\in P_k(K)^n$ is a polynomial of degree at most $k-1$ and $p_2$ is a homogeneous polynomial of degree $k$, this identity can only hold on $K$ if $p_2=0$. Hence, $\div p_1 = 0$ as well.

    For the remainder of the proof, we assume, without loss of generality, that $K$ is the reference unit simplex spanned by the unit vectors in $\R^n$. Consider now for $1\leq i\leq n$ the face $F_i$ aligned with the coordinate plane $\{x\in\R^n:x_i = 0\}$, which has unit normal $\nu_i=-e_i$. Then
    \begin{equation*}
        0 = \tau_h \cdot \nu_i = p_1 \cdot (-e_i) = -[p_1]_i,
    \end{equation*}
    and hence $[p_1]_i$ is a polynomial of degree $k$ that vanishes for all $x$ with $x_i=0$. By \cref{thm:elements:factor}, there thus exists a $\psi_i \in P_{k-1}(K)$ such that $[p_1]_i = x_i \psi_i$ for all $1\leq i\leq n$. We thus obtain a $\tilde q_k = (\psi_1,\dots,\psi_n)^T\in P_{k-1}(K)^n$, which we can insert into \eqref{eq:mixed:rt_uni_K} to deduce
    \begin{equation*}
        \sum_{i=1}^n\int_K x_i |\psi_i|^2\,dx = 0.
    \end{equation*}
    Since we are on the unit simplex, all terms are non-negative and thus have to vanish separately. This implies that $\psi_i = 0$ for $i=1,\dots,n$ and thus $\tau_h = p_1 = 0$.
\end{proof}

\bigskip

Our next task is to construct interpolants in $V_h$ for functions in $V$. This is complicated by the fact that functions in $\m{hdiv}(K)$ have normal traces on $H^{-1/2}(\partial K)$, which cannot be localized to single faces $F\subset \partial K$. We therefore proceed as follows. For $\tau\in\m{h1}(K)^n$ -- which does have well-defined normal traces in $\m{l2}(F)$ by \cref{thm:weak:trace} -- we define the local \emph{Raviart--Thomas projection} $\Pi_K\tau\in RT_k(K)$ by
\begin{align*}
    &\int_F (\Pi_K \tau\cdot\nu -\tau \cdot \nu) q_k\,ds = 0 &&\text{ for all } q_k\in P_k(F), F\subset \partial K,\\
    &\int_K (\Pi_K \tau -\tau)\cdot  q_k\,dx = 0 &&\text{ for all } q_k\in P_{k-1}(K)^n \text{ if } k\geq 1.
\end{align*}
From \cref{lem:mixed:rt_unisolvent}, we already know that the projection conditions imply the uniqueness (and hence existence) of $\Pi_K\tau$.
The next lemma shows that these conditions are chosen precisely in order to use the Raviart--Thomas projector $\Pi_K$ in the construction of a Fortin projector. (Since $\Pi_K$ is not continuous on $\m{Hdiv}$, it cannot be used directly.)
\begin{lemma}\label{lem:mixed:commuting}
    For any $\tau\in\m{h1}(K)^n$,
    \begin{equation*}
        \int_K \div (\Pi_K\tau) q_k\,dx = \int_K (\div \tau) q_k\,dx \qquad\text{ for all }q_k\in P_k(K).
    \end{equation*}
\end{lemma}
\begin{proof}
    Using integration by parts and the definition of the Raviart--Thomas projector, we have for any $q_k\in P_k(K)$ that
    \begin{equation*}
        \int_K \div (\Pi_K\tau-\tau) q_k\,dx = \int_{\partial K} (\Pi_K\tau\cdot\nu-\tau\cdot\nu) q_k\,ds - \int_K (\Pi_K\tau-\tau)\cdot \nabla q_k\,dx
        = 0,
    \end{equation*}
    since $\nabla q_k = 0$ for $k=0$ and $\nabla q_k\in P_{k-1}(K)^n$ for $k\geq 1$.
\end{proof}
This also yields local projection error estimates.
\begin{lemma}\label{lem:mixed:rt_local_error}
    For any $\tau\in\m{h1}(K)^n$,
    \begin{align*}
        \norm*{\m{l2}(K)^n}{\Pi_K\tau - \tau} &\leq Ch_K|\tau|_{\m{h1}(K)^n},\\
        \norm*{\m{l2}(K)}{\div(\Pi_K\tau - \tau)} &\leq C|\tau|_{\m{h1}(K)^n}.
    \end{align*}
    In addition, if $\tau\in H^2(K)^n$,
    \begin{equation*}
        \norm*{\m{l2}(K)}{\div(\Pi_K\tau - \tau)} \leq Ch_K|\tau|_{H^2(K)^n}.
    \end{equation*}
\end{lemma}
\begin{proof}
    Since the projection conditions define a basis of $RT_k(K)^*$, we can write
    \begin{equation*}
        \Pi_K \tau = \sum_{i=0}^{n+1}\sum_{j=0}^{d(i)}N_{i,j}(\tau) \psi_{i,j},
    \end{equation*}
    where $\{\psi_{i,j}\}_{i,j}$ is the corresponding nodal basis of $RT_k(K)$. The trace theorem and Hölder's inequality imply that for every $q_k\in P_k(F)$, the mapping $\tau\mapsto\int_F \tau\cdot \nu q_k\,ds$ is continuous on $\m{h1}(K)^n$. We argue similarly for the degrees of freedom on $K$. Furthermore, from \cref{lem:mixed:commuting} and the fact that $\div(\Pi_K\tau)\in P_k(K)$ by \cref{lem:mixed:rt_poly}\,(i), we obtain
    \begin{equation*}
        \begin{aligned}
            \norm*{\m{l2}(K)}{\div(\Pi_K\tau)}^2 &= \int_K\div(\Pi_K\tau)\div(\Pi_K\tau)\,dx = \int_K(\div\tau)\div(\Pi_K\tau)\,dx \\
            &\leq  \norm*{\m{l2}(K)}{\div\tau} \norm*{\m{l2}(K)}{\div(\Pi_K\tau)}.
        \end{aligned}
    \end{equation*}
    The projection errors thus define bounded linear functionals on $\m{h1}(K)^n$.
    The estimates then follow from the \nameref{thm:bramble} and suitable scaling arguments.\footnote{Since the local coordinate $x$ appears explicitly in the definition of $RT_k(K)$, Raviart--Thomas elements are not affine-equivalent. One thus has to use the \emph{Piola transform}: If $K$ is generated from $\hat K$ by the affine transformation $\hat x\mapsto A_K\hat x +b_K$ and $\hat p\in RT_k(\hat K)$, then $p = \det(A_K)^{-1} A_K \hat p\in RT_k(K)$. Furthermore, the transformed elements are interpolation equivalent; see \cite{Raviart:1977} and \cite{Nedelec:1980}.}
\end{proof}

The global Raviart--Thomas projector $\Pi_\calT$ for $\tau\in\m{H1}^n$ is now defined via $(\Pi_\calT\tau)_K = \Pi_K \tau|_K$ for all $K\in\calT_h$. This projector is bounded in the $\m{Hdiv}$ norm by \cref{lem:mixed:rt_local_error}. Similarly, we obtain from the definition of $M_h$ and \cref{lem:mixed:commuting} that $(\div (\Pi_\calT\tau),v_h) = (\div \tau,v_h)$ for all $v_h\in M_h$. It remains to argue that $\Pi_\calT\tau\in V_h$. Since $\Pi_\calT\tau$ is a piecewise polynomial, it suffices to show that the normal trace is continuous across elements. Let $K_1$ and $K_2$ be two elements sharing a face $F$. Then $\tau\in\m{H1}^n$ has a well-defined normal trace $\tau\cdot\nu\in\m{l2}(F)$ and thus by construction,
\begin{equation*}
    \int_F (\Pi_{K_1}\tau)\cdot \nu\, q_k\,ds = \int_F \tau\cdot\nu \,q_k\,ds = \int_F (\Pi_{K_2}\tau)\cdot \nu\, q_k\,ds\quad\text{for all } q_k\in P_k(F).
\end{equation*}
Since $(\Pi_K\tau)\cdot \nu\in P_k(F)$ by \cref{lem:mixed:rt_poly}\,(ii), we obtain as in the proof of \cref{lem:mixed:rt_unisolvent} that $(\Pi_{K_1}\tau-\Pi_{K_2}\tau)\cdot \nu = 0$ on $F$.

\bigskip

We are now in a position to apply the abstract saddle point framework to the mixed finite element discretization of \eqref{eq:mixed:poisson_dual}: Find $(\sigma_h,u_h)\in V_h\times M_h$ satisfying
\begin{equation}\label{eq:mixed:poisson_dual_h}
    \left\{\begin{aligned}
            (\sigma_h,\tau_h) + (\div\tau_h,u_h) &= 0\quad &&\text{ for all } \tau_h \in V_h,\\
            (\div\sigma_h, v_h) &= -(f,v_h) \quad &&\text{ for all } v_h \in M_h.
    \end{aligned}\right.
\end{equation}
Since $V_h\subset V$ and $M_h\subset M$, the bilinear forms $a:V_h\times V_h\to \R$ and $b:V_h\times M_h\to \R$ are continuous. Furthermore, for $\tau_h\in V_h$ we have $\div\tau_h \in M_h$ and hence the coercivity of $a$ on $\ker B_h\subset V_h$ follows exactly as in the continuous case. For the discrete LBB condition, we proceed as in the proof of the Fortin criterion: For given $v_h\in M_h\subset\m{L2}$, let $\tau_{v_h}\in \m{H1}^n$ be the function given by \cref{lem:mixed:div_surjective}. Then $\Pi_\calT \tau_{v_h}\in V_h\subset V$ and thus
\begin{equation*}
    \sup_{\tau_h\in V_h} \frac{b(\tau_h,v_h)}{\norm*{V}{\tau_h}}
    \geq \frac{(\div (\Pi_\calT \tau_{v_h}),v_h)}{\norm{Hdiv}{\Pi_\calT\tau_{v_h}}}
    \geq \frac{(\div \tau_{v_h},v_h)}{C\norm*{\m{H1}^n}{\tau_{v_h}}}
    \geq \frac{(v_h,v_h)}{C\norm{L2}{v_h}} = \frac1C \norm*{M}{v_h}
\end{equation*}
by the properties of the Raviart--Thomas projector and \cref{lem:mixed:div_surjective}. The conditions of the discrete Brezzi splitting theorem (\cref{thm:brezzi_h}) are thus satisfied, and we deduce well-posedness of \eqref{eq:mixed:poisson_dual_h}.
\begin{theorem}
    For given $f\in\m{L2}$, there exists a unique solution $(\sigma_h,u_h)\in V_h\times M_h$ to \eqref{eq:mixed:poisson_dual_h} satisfying
    \begin{equation*}
        \norm{Hdiv}{\sigma_h}+\norm{L2}{u_h}\leq C\norm{L2}{f}.
    \end{equation*}
\end{theorem}

Using \cref{thm:mixed:cea} to bound the discretization error by the projection error and applying \cref{lem:mixed:rt_local_error} yields a priori error estimates.
\begin{theorem}
    Assume the exact solution $(\sigma,u)\in \m{Hdiv}\times\m{L2}$ to \eqref{eq:mixed:poisson_dual} satisfies $u\in H^3(\Omega)$. Then the solution $(\sigma_h,u_h)\in V_h\times M_h$ satisfies
    \begin{equation*}
        \norm{Hdiv}{\sigma-\sigma_h} + \norm{L2}{u-u_h} \leq C h \norm*{H^3(\Omega)}{u}.
    \end{equation*}
\end{theorem}

\chapter{Discontinuous Galerkin methods}

Discontinuous Galerkin methods are based on nonconforming finite element spaces consisting of piecewise polynomials that are not continuous across elements. These allow handling irregular meshes with hanging nodes and different degrees of polynomials on each element. They also provide a natural framework for first order partial differential equations and for imposing Dirichlet boundary conditions in a weak form, on which we will focus here. We consider a simple \emph{advection-reaction} equation
\begin{equation*}
    \beta\cdot\nabla u +\mu u = f,
\end{equation*}
which models the transport of a solute concentration $u$ along the vector field $\beta$. The reaction coefficient $\mu$ determines the rate with which the solute is destroyed or created due to interaction with its environment, and $f$ is a source term. This is complemented by (for simplicity) homogeneous Dirichlet conditions of a form to be specified below.

\section{Weak formulation of advection-reaction equations}

We consider $\Omega\subset\R^n$ (polyhedral) with unit outer normal $\nu$ and assume that
\begin{equation*}
    \mu\in\m{Linf},\qquad \beta\in{W}^{1,\infty}(\Omega)^n,\qquad f\in\m{L2}.
\end{equation*}
Our first task is to define the space in which we look for our solution. Let
\begin{equation*}
    \partial\Omega^- = \setof{s\in\partial\Omega}{\beta(s)\cdot\nu(s) < 0}
\end{equation*}
denote the \emph{inflow boundary} and
\begin{equation*}
    \partial\Omega^+ = \setof{s\in\partial\Omega}{\beta(s)\cdot\nu(s) > 0}
\end{equation*}
denote the \emph{outflow boundary}, and assume that they are well-separated, i.e.,
\begin{equation*}
    \inf_{s\in\partial\Omega^-,t\in\partial\Omega^+} |s-t| > 0.
\end{equation*}
Then we define the so-called \emph{graph space}
\begin{equation*}
    W = \setof{v\in\m{L2}}{\beta\cdot\nabla v \in\m{L2}},
\end{equation*}
which is a Hilbert space if endowed with the inner product
\begin{equation*}
    \scalprod*{W}{v}{w} = (v,w)+(\beta\cdot\nabla v,\beta\cdot\nabla w).
\end{equation*}
The latter induces the \emph{graph norm}
\begin{equation*}
    \norm*{W}{v}^2 = \norm{L2}{v}^2 + \norm{L2}{\beta\cdot\nabla v}^2.
\end{equation*}
One can show\footnote{e.g., \cite[Lemma 2.5]{DiPietro:2012}} that functions in $W$ have traces in the space
\begin{equation*}
    \m{l2}_\beta(\partial\Omega) = \setof{v \text{ measurable on }\partial\Omega}{\int_{\partial\Omega}|\beta\cdot\nu|\,v^2\,ds <\infty},
\end{equation*}
and that the following integration by parts formula holds:
\begin{equation}\label{eq:dg:intpart}
    \int_\Omega (\beta\cdot \nabla v)w + (\beta\cdot\nabla w)v+(\nabla\cdot \beta)vw \,dx = \int_{\partial\Omega}(\beta\cdot \nu)vw\, ds
\end{equation}
for all $v,w\in W$.

We can now define our weak formulation: set
\begin{equation*}
    U:=\setof{v\in W}{v|_{\partial\Omega^-} = 0}
\end{equation*}
and find $u\in U$ satisfying
\begin{equation}\label{eq:dg:weak}
    a(u,v) := (\beta\cdot\nabla u,v) +(\mu u,v) = (f,v)\tag{$\mathcal{W}$}
\end{equation}
for all $v\in V=\m{L2}$. Note that the test space is now different from the solution space.

Since $U$ is a closed subspace of the Hilbert space $W$, it is a Banach space. Moreover, $\m{L2}$ is a reflexive Banach space, and the right-hand side defines a continuous linear functional on $\m{L2}$. We can thus apply the \nameref{thm:BNB} Theorem to show well-posedness.
\begin{theorem}\label{thm:dg:wellposed}
    If
    \begin{equation*}
        \mu(x) - \tfrac12\nabla\cdot\beta(x) \geq \mu_0 > 0 \quad \text{for almost all } x\in\Omega,
    \end{equation*}
    then there exists a unique $u\in U$ satisfying \eqref{eq:dg:weak}. Furthermore, there exists a $C>0$ such that
    \begin{equation*}
        \norm*{W}{u}\leq C\norm{L2}{f}.
    \end{equation*}
\end{theorem}
\begin{proof}
    We begin by showing the continuity of $a$ on $U\times V$. For arbitrary $u\in U$ and $v\in V= \m{L2}$,  the Cauchy--Schwarz inequality yields
    \begin{equation*}
        \begin{aligned}
            |a(u,v)| &\leq \norm{L2}{\beta\cdot\nabla u}\norm{L2}{v} + \norm{L2}{\mu u}\norm{L2}{v}\\
            &\leq \max\{1,\norm{Linf}{\mu}\}\sqrt{2}\norm*{W}{u}\norm*{V}{v}.
        \end{aligned}
    \end{equation*}

    To verify the inf--sup condition, we first prove coercivity with respect to the $\m{L2}$ part of the graph norm. For any $u\in U\subset V$, we integrate by parts using \eqref{eq:dg:intpart} for $v=w=u$ to obtain
    \begin{equation*}
        \begin{aligned}
            a(u,u) &= \int_\Omega(\beta\cdot\nabla u)u + \mu u^2\,dx \\
            &= \int_\Omega (\mu -\tfrac12 \nabla\cdot \beta)u^2 \,dx +\int_{\partial\Omega}\tfrac12(\beta\cdot\nu)u^2\,ds\\
            &\geq \mu_0\norm{L2}{u}^2,
        \end{aligned}
    \end{equation*}
    where we have used that $u$ vanishes on $\partial\Omega^-$ due to the boundary conditions and that $\beta\cdot\nu>0$ on $\partial\Omega^+$. This implies that
    \begin{equation*}
        \norm{L2}{u} \leq \mu_0^{-1}\frac{a(u,u)}{\norm{L2}{u}} \leq \sup_{v\in\m{L2}} \mu_0^{-1}\frac{a(u,v)}{\norm{L2}{v}} .
    \end{equation*}
    For the other term in the graph norm, we use a duality trick to write
    \begin{equation*}
        \begin{aligned}
            \norm{L2}{\beta\cdot\nabla u} &= \sup_{v\in\m{L2}}\frac{(\beta\cdot\nabla u,v)}{\norm{L2}{v}} \\
            &= \sup_{v\in\m{L2}}\frac{a(u,v) - (\mu u,v)}{\norm{L2}{v}} \\
            &\leq  \sup_{v\in\m{L2}}\frac{a(u,v)}{\norm{L2}{v}} + \norm{Linf}{\mu}\norm{L2}{u}\\
            &\leq  (1+\mu_0^{-1}\norm{Linf}{\mu})\sup_{v\in\m{L2}}\frac{a(u,v)}{\norm{L2}{v}}.
        \end{aligned}
    \end{equation*}
    Summing the last two inequalities and taking the infimum over all $u\in U$ verifies the inf--sup condition.

    For the injectivity condition, we assume that $v\in\m{L2}$ is such that $a(u,v)=0$ for all $u\in U$ and show that $v=0$. Since $\m{Cinf0}\subset U$, we deduce from $a(u,v)=0$ that $\nabla \cdot(\beta v)$ exists as a weak derivative and that $\nabla \cdot(\beta v) = \mu v$. By the product rule, we furthermore have $\beta\cdot\nabla v = (\mu-\nabla\cdot\beta)v\in\m{L2}$, which implies $v\in W$. Inserting this into the integration by parts formula \eqref{eq:dg:intpart} and adding the productive zero yields for all $u\in U$
    \begin{equation}\label{eq:dg:inj1}
        \begin{aligned}[t]
            \int_{\partial\Omega}(\beta\cdot\nu)uv\,dx &=
            \int_\Omega (\beta\cdot \nabla v)u + (\beta\cdot\nabla u)v+(\nabla\cdot \beta)vu \,dx \\
            &= a(u,v) -((\mu-\nabla\cdot\beta)v-\beta\cdot\nabla v,u)\\
            &= 0.
        \end{aligned}
    \end{equation}
    Since $\partial\Omega^+$ and $\partial\Omega^-$ are well separated, there exists a smooth cut-off function $\chi\in\m{Cinf}$ with $\chi(s) = 0$ for $s\in\partial\Omega^-$ and $\chi(s) = 1$ for $s\in\partial\Omega^+$. Applying \eqref{eq:dg:inj1} to $u=\chi v\in U$ yields $\int_{\partial\Omega^+}(\beta\cdot\nu)v^2\,dx = 0$.
    Using again that $\mu v = \nabla \cdot(\beta v)$ and integrating by parts, we deduce that
    \begin{equation*}
        \begin{aligned}
            0 &= \int_\Omega \mu v^2-\nabla\cdot(\beta v)v\,dx \\
            &= \int_\Omega( \mu-\tfrac12 \nabla\cdot\beta)v^2 \,dx - \int_{\partial\Omega}\tfrac12(\beta\cdot\nu)v^2\,ds\\
            &\geq \mu_0\norm{L2}{v}
        \end{aligned}
    \end{equation*}
    since the remaining boundary integral over $\partial\Omega^-$ is non-positive. This shows that $v=0$.
\end{proof}

Note that the graph norm is the strongest norm in which we could have shown coercivity, and that $a$ would not have been bounded on $U\times U$.

\section{Galerkin approach}

The \emph{discontinuous Galerkin} approach now consists in choosing for $k\geq 0$ and a given triangulation $\calT_h$ of $\Omega$  \emph{both} of the discrete spaces as
\begin{equation*}
    U_h = V_h = \setof{v\in\m{L2}}{v|_K \in P_k, K\in\calT_h}
\end{equation*}
(no continuity across elements is assumed, hence the name). We then search for $u_h\in V_h$ satisfying
\begin{equation}\label{eq:dg:galerkin}
    a_h(u_h,v_h) = (f,v_h) \qquad \text{ for all } v_h\in V_h,\tag{$\mathcal{W}_h$}
\end{equation}
for a bilinear form $a_h$ to be specified. Here, we consider the simplest choice that leads to a convergent scheme. Recall that the set of all faces of $\calT_h$ is denoted by $\partial\calT_h$ and the set of all interior faces by $\Gamma_h$.
Let $F\in\Gamma_h$ be the face common to the elements $K_1,K_2\in\calT_h$ with exterior normal $\nu_1$ and $\nu_2$, respectively.
For a (sufficiently regular) function $v\in\m{L2}$, we denote the \emph{jump} across $F$ as
\begin{equation*}
    \jump{v}_F = v|_{K_1}\nu_1 + v|_{K_2}\nu_2 \in L^2(F)^n,
\end{equation*}
and the \emph{average} as
\begin{equation*}
    \avg{v}_F = \tfrac12(v|_{K_1}+v|_{K_2})\in L^2(F).
\end{equation*}
We will omit the subscript $F$ if it is clear which face is meant. It is also convenient to introduce for $v_h\in V_h$ the \emph{broken gradient} $\nabla_h v_h\in L^2(\Omega)$ via
\begin{equation*}
    (\nabla_h v_h)|_{K} = \nabla (v_h|_K) \qquad \text{for all } K\in\calT_h.
\end{equation*}
We then define the bilinear form
\begin{equation}\label{eq:dg:cf}
    \begin{aligned}[t]
        a_h(u_h,v_h) &= (\mu u_h+\beta\cdot\nabla_h u_h,v_h) - \int_{\partial\Omega^-}(\beta\cdot\nu)u_h v_h\,ds\\
        \MoveEqLeft[-1]  - \sum_{F\in\Gamma_h}\int_F \beta\cdot\jump{u_h}\avg{v_h}\,ds.
    \end{aligned}
\end{equation}
The second term enforces the homogeneous Dirichlet conditions in a weak sense. The last term can be thought of as weakly enforcing continuity by penalizing the jump across each face; the reason for its specific form will become apparent in the following proof of coercivity with respect to the ``discrete energy norm''
\begin{equation*}
    \norm{dg}{u_h}^2 := \mu_0\norm{L2}{u_h}^2  + \int_{\partial\Omega} \tfrac12|\beta\cdot \nu|u_h^2\,ds,
\end{equation*}
which is clearly a norm on $V_h\subset\m{L2}$.
\begin{lemma}\label{lem:dg:coercive}
    Under the assumption of \cref{thm:dg:wellposed}, there exists a constant $C>0$ independent of $h$ such that
    \begin{equation*}
        a_h(u_h,u_h) \geq C \norm{dg}{u_h}^2\qquad\text{for all }u_h\in V_h.
    \end{equation*}
\end{lemma}
\begin{proof}
    We begin by applying integration by parts on each element to the first term of \eqref{eq:dg:cf} for $v_h=u_h$ to obtain
    \begin{equation*}
        \begin{aligned}
            (\mu u_h+\beta\cdot\nabla_h u_h,u_h) &= \sum_{K\in\calT_h}\int_K \mu u_h^2 + (\beta\cdot\nabla u_h)u_h\,dx \\
            &= \sum_{K\in\calT_h}\int_K \mu u_h^2 - \tfrac12 (\nabla \cdot\beta) u_h^2\,dx +\int_{\partial K}\tfrac12 (\beta\cdot\nu)u_h^2\,ds.
        \end{aligned}
    \end{equation*}
    The last term can be reformulated as a sum over faces. Since $\beta\in{W}^{1,\infty}(\Omega)^n$ is continuous, we have
    \begin{equation*}
        \sum_{K\in\calT_h} \int_{\partial K}\tfrac12 (\beta\cdot\nu)u_h^2\,ds = \sum_{F\in\Gamma_h}  \int_{F}\tfrac12 \beta\cdot\jump{u_h^2}\,ds + \sum_{F\in\partial\calT_h\setminus\Gamma_h}  \int_{F}\tfrac12 (\beta\cdot\nu)u_h^2\,ds .
    \end{equation*}
    Using that $\nu:=\nu_1 = -\nu_2$ and therefore
    \begin{equation*}
        \tfrac12\jump{w^2}_F = \tfrac12(w|_{K_1}^2 - w|_{K_2}^2)\nu = \tfrac12(w|_{K_1}+w|_{K_2})(w|_{K_1} - w|_{K_2})\nu = \avg{w}_F\jump{w}_F,
    \end{equation*}
    and combining the terms involving integrals over $\partial\Omega$, we obtain
    \begin{equation*}
        \sum_{K\in\calT_h}    \int_{\partial K}\tfrac12 (\beta\cdot\nu)u_h^2\,ds - \int_{\partial\Omega^-}(\beta\cdot\nu)u_h^2\,ds =  \sum_{F\in\Gamma_h}  \int_{F}\beta\cdot\jump{u_h}\avg{u_h}\,ds + \int_{\partial\Omega}\tfrac12|\beta\cdot\nu| u_h^2\,ds.
    \end{equation*}
    Note that we have no control over the sign of the first term on the right-hand side, which is why we had to introduce the penalty term in $a_h$ to cancel it. Combining these equations yields
    \begin{equation*}
        \begin{aligned}[b]
            a_h(u_h,u_h) &= \sum_{K\in\calT_h}\int_K \left(\mu   - \tfrac12 (\nabla \cdot\beta)\right) u_h^2\,dx +\int_{\partial\Omega}\tfrac12|\beta\cdot\nu| u_h^2\,ds\\
            &\geq \mu_0\norm{L2}{u_h}^2 + \int_{\partial\Omega}\tfrac12|\beta\cdot\nu| u_h^2\,ds\\
            &= \norm{dg}{u_h}^2.
        \end{aligned}
        \qedhere
    \end{equation*}
\end{proof}
We will show continuity of $a$ on $V_h\times V_h$ (with respect to an equivalent norm) later (\cref{lem:dg:bounded}), from which we then obtain existence of a unique solution $u_h\in V_h$ to \eqref{eq:dg:galerkin}.

\section{Error estimates}

To derive error estimates for the discontinuous Galerkin approximation $u_h\in V_h$ to $u\in U$, we wish to apply the \nameref{thm:strang2}. Our first task is to show boundedness of $a_h$ on a sufficiently large space containing the exact solution. Since the corresponding norm will involve traces, we make the additional assumption that the exact solution satisfies
\begin{equation*}
    u\in U_* := U\cap \m{H1}.
\end{equation*}
By the trace theorem (\cref{thm:weak:trace}), $u|_F$ is then well-defined in the sense of $\m{l2}(F)$ traces. We now define on $U(h) := U_* + V_h$ the norm
\begin{equation*}
    \norm{dg}{w}_*^2 := \norm{dg}{w}^2 + \sum_{K\in\calT_h} \left(\norm*{\m{l2}(K)}{\beta\cdot\nabla w}^2 + h_K^{-1}\norm*{\m{l2}(\partial K)}{w}^2\right).
\end{equation*}
We can then show boundedness of $a_h$ on $U(h)\times V_h$ if the triangulation is shape-regular.
\begin{lemma}\label{lem:dg:bounded}
    If $\calT_h$ is a shape-regular triangulation of $\Omega\subset \R^2$, then there exists a constant $C>0$ independent of $h$ such that
    \begin{equation*}
        a_h(u,v_h) \leq C \norm{dg}{u}_* \norm{dg}{v_h}\quad\text{for all }u\in U(h),\ v_h\in V_h.
    \end{equation*}
\end{lemma}
\begin{proof}
    Using the Cauchy--Schwarz inequality and some generous upper bounds, we immediately obtain
    \begin{equation}\label{eq:dg:bound1}
        (\mu u+\beta\nabla_h u,v_h) + \int_{\partial\Omega^-}(\beta\cdot\nu)uv_h\,ds \leq C\norm{dg}{u}_*\norm{dg}{v_h},
    \end{equation}
    with a constant $C>0$ depending only on $\mu$. For the last term of $a_h(u,v_h)$, we first insert $1=(2\avg{h})(2\avg{h})^{-1}$, where in a slight abuse of notation we consider $h:\Omega\to \R$ as a function mapping $x\in K$ to $h_K$. Since this function $h$ is constant on each element, we obtain using the Cauchy--Schwarz inequality that
    \begin{equation}\label{eq:dg:cs_face}
        \sum_{F\in\Gamma_h}\int_F \beta\cdot\jump{u}\avg{v_h}\,ds\leq C \left(\sum_{F\in\Gamma_h}\tfrac12 \avg{h}^{-1}\norm*{\m{l2}(F)^n}{\jump{u}}^2\right)^{\frac12}
        \left(\sum_{F\in\Gamma_h} 2 \avg{h}\norm*{\m{l2}(F)}{\avg{v_h}}^2\right)^{\frac12},
    \end{equation}
    where $C>0$ depends only on $\beta$. Now we use that
    \begin{equation*}
        \tfrac12\jump{w}_F^2 \leq (w|_{K_1}^2+w|_{K_2}^2) ,\qquad 2\avg{w}_F^2\leq  (w|_{K_1}^2+w|_{K_2}^2),
    \end{equation*}
    and that for a shape-regular mesh, the element size $h_K$ cannot change arbitrarily between neighboring elements, i.e., there exists a $c>0$ such that
    \begin{equation*}
        c^{-1} \max(h_{K_1},h_{K_2}) \leq \avg{h}_F \leq  c \min(h_{K_1},h_{K_2}).
    \end{equation*}
    Combining the terms arising from the faces of each element and applying the discrete trace inequality (obtained in the usual way)\footnote{e.g., \cite[Lemma 1.46]{DiPietro:2012}}
    \begin{equation}\label{eq:dg:inverse}
        h_K^{1/2}\norm*{\m{l2}(\partial K)}{v_h} \leq C \norm*{\m{l2}(K)}{v_h},
    \end{equation}
    we thus obtain
    \begin{equation}\label{eq:dg:bound2}
        \begin{aligned}[t]
            \sum_{F\in\Gamma_h}\int_F \beta\cdot\jump{u}\avg{v_h}\,ds
            &\leq C \left(\sum_{K\in\calT_h} h_K^{-1} \norm*{\m{l2}(\partial K)}{u}^2\right)^{\frac12}\left(\sum_{K\in\calT_h} h_K \norm*{\m{l2}(\partial K)}{v_h}^2\right)^{\frac12}\\
            & \leq C \norm{dg}{u}_*\norm{dg}{v_h}.
        \end{aligned}
    \end{equation}
    Adding \eqref{eq:dg:bound1} and \eqref{eq:dg:bound2} yields the claim.
\end{proof}
On the finite-dimensional subspace $V_h\subset U(h)$, the norm $\norm{dg}{\cdot}_*$ is equivalent to $\norm{dg}{\cdot}$, and hence together with \cref{lem:dg:coercive} we now have verified the conditions necessary for applying \cref{thm:bnb_h} to deduce well-posedness of \eqref{eq:dg:galerkin}. However, since $\norm{dg}{\cdot}_*$ involves $h_K$, the constants of equivalence also depend on the mesh size $h$, and hence the a priori estimate is no longer uniform in $h$.
\begin{cor}\label{cor:dg:wellposed}
    If $\calT_h$ is a shape-regular triangulation of $\Omega\subset\R^2$ and
    \begin{equation*}
        \mu(x) - \tfrac12\nabla\cdot\beta(x) \geq \mu_0 > 0 \quad \text{for almost all } x\in\Omega,
    \end{equation*}
    then there exists a unique solution $u_h \in V_h$ to \eqref{eq:dg:galerkin}. Furthermore, there exists a constant $C>0$ such that
    \begin{equation*}
        \norm{dg}{u_h} \leq C \norm{L2}{f}.
    \end{equation*}
\end{cor}

Before we derive error estimates, we show that our discontinuous Galerkin approximation is consistent and hence that the consistency error in the \nameref{thm:strang2} vanishes.
\begin{lemma}\label{lem:dg:consistent}
    A solution $u\in U_*$ to \eqref{eq:dg:weak} satisfies
    \begin{equation*}
        a_h(u,v_h) = (f,v_h)
    \end{equation*}
    for all $v_h \in V_h$.
\end{lemma}
\begin{proof}
    By definition, $u\in U_*=U\cap \m{H1}$ satisfies $\nabla_h u = \nabla u$ and thus
    \begin{equation*}
        (\mu u + \beta \cdot \nabla_h u,v_h) = (f,v_h) \quad\text{for all }v_h\in V_h\subset V.
    \end{equation*}
    Furthermore, due to the boundary conditions,
    \begin{equation*}
        \int_{\partial\Omega^-}(\beta\cdot\nu)uv_h\,ds = 0.
    \end{equation*}
    It remains to show that the penalty term $ (\beta\cdot\nu)\jump{u_h}_F\avg{v_h}_F$ vanishes on each face $F\in\Gamma_h$. Let $\phi\in\m{Cinf0}$ have support contained in $S\subset\bar K_1\cup\bar K_2 \subset\Omega$ and intersecting $F = \partial K_1\cap \partial K_2$. Then the integration by parts formula \eqref{eq:dg:intpart} yields
    \begin{equation*}
        \begin{aligned}
            0 &= \int_\Omega (\beta\cdot \nabla u)\phi + (\beta\cdot\nabla \phi)u+(\nabla\cdot \beta)u\phi \,dx \\
            & = \int_{S\cap K_1} (\beta\cdot \nabla u)\phi + (\beta\cdot\nabla \phi)u+(\nabla\cdot \beta)u\phi \,dx\\
            \MoveEqLeft[-1]+  \int_{S\cap K_2} (\beta\cdot \nabla u)\phi + (\beta\cdot\nabla \phi)u+(\nabla\cdot \beta)u\phi \,dx \\
            &= \int_{\partial K_1\cap S}(\beta\cdot \nu)u\phi\, ds + \int_{\partial K_2\cap S}(\beta\cdot \nu)u\phi\, ds\\
            & = \int_F \beta\cdot\jump{u}\phi\,ds.
        \end{aligned}
    \end{equation*}
    The claim then follows from a density argument.
\end{proof}

We thus obtain the following error estimate.
\begin{theorem}
    Assume that the solution $u\in U(h)$ to \eqref{eq:dg:weak} satisfies $u\in H^{k+1}(\Omega)$. Then there exists a $c>0$ independent of $h$ such that
    \begin{equation*}
        \norm{dg}{u-u_h} \leq c h^k |u|_{H^{k+1}(\Omega)}.
    \end{equation*}
\end{theorem}
\begin{proof}
    Since $a_h:U(h)\times V_h\to\R$ is consistent, continuous with respect to the $\norm{dg}{\cdot}_*$ norm, and coercive with respect to the $\norm{dg}{\cdot}$ norm, we deduce as in the \nameref{thm:strang2} that
    \begin{equation*}
        \norm{dg}{u-u_h} \leq c \inf_{w_h\in V_h} \norm{dg}{u-w_h}_*.
    \end{equation*}
    Assuming that $u$ is sufficiently smooth that the local interpolant $\calI_K u$ is well-defined, we can show by the usual arguments that
    \begin{align*}
        \norm*{L^2(K)}{u-\calI_K u} &\leq c h_K^{k+1}|u|_{H^{k+1}(K)},\\
        |{u-\calI_K u}|_{H^1(K)} &\leq c h_K^{k}|u|_{H^{k+1}(K)},\\
        \norm*{L^2(\partial K)}{u-\calI_K u} &\leq c h_K^{k+1/2}|u|_{H^{k+1}(K)}.
    \end{align*}
    Applying these bounds in turn to each term in $\norm{dg}{u-\calI_\calT u}_*$ yields the desired estimate.
\end{proof}
Note that since we could only show coercivity with respect to $\norm{dg}{\cdot}$ (and $u-u_h$ is not in a finite-dimensional space), we only get an error estimate in this (weaker) norm of $\m{l2}$ type, while the approximation error needs to be estimated in the (stronger) $\m{h1}$-type norm $\norm{dg}{\cdot}_*$. On the other hand, we would expect a convergence order $h^{k+1/2}$  for the discretization error in an $\m{l2}$-type norm (involving interface terms). This discrepancy is due to the simple penalty we added, which is insufficient to control oscillations. (The penalty only canceled the interface terms arising in the integration by parts, but did not contribute further in the coercivity). A more stable alternative is \emph{upwinding}: Take
\begin{equation*}
    a_h^+(u_h,v_h) = a_h(u_h,v_h) + \sum_{F\in\Gamma_h} \int_F \frac\eta{2}|\beta\cdot\nu|\jump{u_h}\cdot\jump{v_h}\,ds
\end{equation*}
for a sufficiently large penalty parameter $\eta>0$. It can be shown\footnote{e.g., \cite[Chapter 2.3]{DiPietro:2012}} that this bilinear form is consistent as well, and is coercive in the norm
\begin{equation*}
    \norm{dg}{w}_+^2 = \norm{dg}{w}^2 +  \sum_{F\in\Gamma_h} \int_F \frac\eta{2}|\beta\cdot\nu|\jump{w}^2\,ds +  \sum_{K\in\calT_h}  h_K \norm*{\m{l2}(K)}{\beta\cdot\nabla w}^2
\end{equation*}
and continuous in
\begin{equation*}
    \norm{dg}{w}_{+,*}^2 =   \norm{dg}{w}_+^2 +  \sum_{K\in\calT_h} \left(h_K^{-1}\norm*{\m{l2}(K)}{w}^2 + \norm*{\m{l2}(\partial K)}{w}^2\right),
\end{equation*}
which can be used to obtain the expected convergence order of $h^{k+1/2}$ (which is useful in the case $k=0$ as well).

\section{Discontinuous Galerkin methods for elliptic equations}

Due to their flexibility, discontinous Galerkin methods have become popular for elliptic second-order problems as well. We illustrate the approach with the simplest example, the Poisson equation $-\Delta u = f$ on $\Omega\subset \R^n$ with homogeneous Dirichlet conditions.
The basic idea is to write the second-order equation as a system of first-order equations, for which we can proceed as before via element-wise integration by parts to obtain face integrals that can be used as penalty terms in place of the dropped continuity requirement and boundary condition on the discrete solution. For $u\in H^1(\Omega)$, we thus again introduce $\sigma:=\nabla u\in L^2(\Omega)^n$ so that the Poisson equation reduces to $-\nabla \cdot \sigma = f$. We now multiply these two equations with (sufficiently smooth) test functions $\tau \in C^\infty(\overline\Omega)^n$ and $v\in C^\infty(\overline\Omega)$, respectively,
and integrate by parts separately on each element $K$ of a triangulation $\calT_h$ of $\Omega$ to obtain
\begin{equation}\label{eq:dg:poisson:flux}
    \left\{\begin{aligned}
            \sum_{K\in\calT_h} \int_K \sigma\cdot \tau\,dx + \sum_{K\in \calT_h} \int_K u \nabla\cdot\tau\,dx - \sum_{K\in \calT_h} \int_{\partial K} u\,( \tau\cdot\nu)\,ds &= 0,\\
            \sum_{K\in\calT_h} \int_K \sigma\cdot \nabla v \,dx - \sum_{K\in \calT_h} \int_{\partial K} (\sigma\cdot\nu)\, v\,ds &= (f,v).
    \end{aligned}\right.
\end{equation}
The idea is now to replace $u$ and $\sigma$ in the face integrals by a suitable approximations $\hat u_F$ of $u$ and $\hat \sigma_F$ of $\nabla u$ (sometimes called \emph{potential} and  \emph{diffusive flux}, respectively) and then eliminating $\sigma$ (but not $\hat \sigma$). Inserting $\tau=\nabla v$ in the first equation of \eqref{eq:dg:poisson:flux} and integrating by parts in the second term yields, on each element, after rearranging
\begin{equation*}
    \int_K \sigma\cdot \nabla v\,dx = \int_K \nabla u\cdot \nabla v\,dx - \int_{\partial K} u\, (\nabla v\cdot\nu)\,ds  +\int_{\partial K} \hat u_F\, (\nabla v\cdot\nu)\,ds.
\end{equation*}
Inserting this into the left-hand side of the second equation then yields (using the definition of the broken gradient)
\begin{equation}\label{eq:dg:poisson:primal}
    \begin{aligned}[t]
        a_h(u,v) &:=  (\nabla_h u, \nabla_h v) + \sum_{K\in \calT_h} \int_{\partial K} (\hat u_F-u)\, (\nabla v\cdot\nu)\,ds - \sum_{K\in \calT_h} \int_{\partial K} (\hat\sigma_F\cdot\nu)\, v\,ds \\
        &= (f,v).
    \end{aligned}
\end{equation}

The next step is to rearrange the sum over element boundary integrals into a sum over faces. A straightforward computation shows that for piecewise smooth scalar-valued $v$ and vector-valued $\tau$,
\begin{equation}\label{eq:dg:discrete_partint}
    \sum_{K\in \calT_h} \int_{\partial K} v\, (\tau\cdot\nu)\,ds = \sum_{F\in\partial\calT_h} \int_F \jump{v}\cdot\avg{\tau}\,ds + \sum_{F\in\Gamma_h} \int_F \avg{v}\jump{\tau}\,ds
\end{equation}
(recalling that the jump of a scalar function is vector-valued, while that of a vector-valued is scalar; see \cref{sec:error:residual}). Before applying this to the terms in \eqref{eq:dg:poisson:primal}, however, we first discuss the choice of fluxes, each of which leads to a different discontinous Galerkin approach. A popular choice\footnote{Other choices are discussed in \cite{Arnold:2002}.} is the \emph{symmetric interior penalty} method, which corresponds to setting
\begin{equation*}
    \hat u_F := \avg{u}_F\quad\text{for }F\in\Gamma_h,\qquad \hat u_F = 0\quad\text{for }F\in\calT_h\setminus\Gamma_h,\qquad  \hat \sigma_F := \avg{\nabla_h u}_F - \frac\eta{h_F}\jump{u}_F,
\end{equation*}
where $h_F$ is the diameter of the face $F\in \partial\calT_h$ and $\eta>0$ has to be chosen sufficiently large. (The specific form of the second term will again become clear when discussing coercivity below.)
With these choices, applying \eqref{eq:dg:discrete_partint} to \eqref{eq:dg:poisson:primal} and using that $\avg{\avg{w}}=\avg{w}$ and $\jump{\avg{w}}=\jump{\jump{w}}=0$ for all $w$, we arrive at
\begin{equation}\label{eq:dg:poisson:sipg}
    a_h(u,v) = (\nabla_h u, \nabla_h v) - \sum_{F\in\partial\calT_h}\int_F\jump{u}\cdot\avg{\nabla_h v}+\avg{\nabla_h u}\cdot\jump{v}\,ds +\int_F \frac\eta{h_F}\jump{u}\jump{v}\,ds.
\end{equation}

\bigskip

As usual in a discontinuous Galerkin method, we now choose
\begin{equation*}
    V_h = \setof{v\in\m{L2}}{v|_K \in P_k, K\in\calT_h}
\end{equation*}
and search for $u_h\in V_h$ satisfying
\begin{equation}\label{eq:dg:sip}
    a_h(u_h,v_h) = (f,v_h) \qquad\text{for all }v_h\in V_h.
\end{equation}
To show well-posedness using the \nameref{thm:BNB} theorem, we need to show continuity and coercivity of $a_h$ with respect to appropriate norms. We again postpone continuity (in an equivalent norm) to later, and address coercivity with respect to the discrete norm
\begin{equation*}
    \norm{dg}{v_h}^2 := \norm*{\m{L2}^n}{\nabla_h v}^2 + |v_h|_{\Gamma_h}^2,
\end{equation*}
with the \emph{jump seminorm}
\begin{equation*}
    |v_h|_{\Gamma_h}^2 := \sum_{F\in\partial\calT_h} h_F^{-1}\norm*{L^2(F)^n}{\jump{v_h}}^2;
\end{equation*}
for $F\subset \partial\Omega$ we use the convention that $u=0$ outside of $\Omega$.
This is indeed a norm on $V_h$ since $\norm{dg}{v_h}=0$ implies first that $v_h$ is piecewise constant; and since the function vanishes on the boundary and the interface jumps are zero, these constants are zero.

Again we postpone continuity to later and first verify the coercivity of $a_h$ with respect to $\norm{dg}{\cdot}$.
\begin{lemma}\label{lem:dg:elliptic_dg_coercive}
    For all $\eta> 0$ sufficiently large, there exists a $C>0$ independent of $h$ such that
    \begin{equation*}
        a_h(u_h,u_h) \geq C \norm{dg}{u_h}^2\qquad\text{for all }u_h\in V_h.
    \end{equation*}
\end{lemma}
\begin{proof}
    For arbitrary $u_h\in V_h$, we have using the definition of the broken gradient and the jump seminorm that
    \begin{equation*}
        a_h(u_h,u_h) = \norm*{\m{L2}^n}{\nabla_h u_h}^2 - 2\sum_{F\in \partial\calT_h} \int_F \avg{\nabla_h u_h}\cdot\jump{u_h} \,ds + \eta |u_h|_{\Gamma_h}^2.
    \end{equation*}
    Since the second term has the wrong sign, we need to absorb it into the other terms. For this, we use that for any piecewise smooth $v,w$ and every $F\in \partial\calT_h$, the Cauchy--Schwarz inequality yields
    \begin{equation*}
        \begin{aligned}
            \int_F \avg{\nabla_h v}\cdot \jump{w}\,ds &=
            \int_F \frac12\left(\nabla_h v|_{K_1}+\nabla_h v|_{K_2}\right)\cdot \jump{w}\,ds\\
            &\leq \frac{h_F^{1/2}}{2}\left(\norm*{\m{l2}(F)^n}{\nabla_h v|_{K_1}}^2+\norm*{\m{l2}(F)^n}{\nabla_h v|_{K_2}}^2\right)^{\frac12} h_F^{-1/2} \norm*{\m{l2}(F)^n}{\jump{w}}.
        \end{aligned}
    \end{equation*}
    Summing over all faces and using the fact that each interior face occurs twice and that for boundary faces we set $v=w=0$ outside of $\Omega$, we obtain
    \begin{equation}\label{eq:dg:sip:estimate1}
        \sum_{F\in\partial\calT_h}\int_F  \avg{\nabla_h v}\cdot \jump{w}\,ds \leq
        \left(\sum_{K\in\calT_h}\sum_{F\subset\partial K} h_F \norm*{\m{l2}(F)^n}{\nabla_h v}^2\right)^{\frac12} |w|_{\Gamma_h}.
    \end{equation}
    For $v_h\in V_h$, we can further use the discrete trace inequality \eqref{eq:dg:inverse} together with $h_F\leq h_K$ for all faces $F$ of $K$ to arrive at
    \begin{equation} \label{eq:dg:sip:estimate2}
        \begin{aligned}[t]
            \sum_{F\in\partial\calT_h}\int_F  \avg{\nabla_h v_h}\cdot \jump{w}\,ds &\leq
            C \left(\sum_{K\in\calT_h} \norm*{\m{l2}(K)^n}{\nabla_h v_h}^2\right)^{\frac12} |w|_{\Gamma_h}\\
            &=C\norm*{\m{L2}^n}{\nabla_h v_h}|w|_{\Gamma_h}.
        \end{aligned}
    \end{equation}
    Applying this estimate for $v_h=w=u_h$ together with the generalized Young inequality $ab \leq \frac\eps2 a^2 + \frac1{2\eps}b^2$ for arbitrary $\eps>0$ then yields that
    \begin{equation*}
        \begin{aligned}
            a_h(u_h,u_h) &\geq \norm*{\m{L2}^n}{\nabla_h u_h}^2 - 2C \norm*{\m{L2}^n}{\nabla_h u_h}|u_h|_{\Gamma_h} + \eta |u_h|_{\Gamma_h}^2\\
            &\geq (1-C\eps) \norm*{\m{L2}^n}{\nabla_h u_h}^2 + (\eta - C\eps^{-1})|u_h|_{\Gamma_h}^2.
        \end{aligned}
    \end{equation*}
    We can now first choose $\eps>0$ sufficiently small that the first term is positive, and then $\eta>0$ sufficiently large that the second term is positive, which implies coercivity in the desired norm.
\end{proof}

\bigskip

For error estimates, we again need to show boundedness of $a_h$ on a space containing both discrete and exact solutions. Here we assume that the exact solution of the Poisson equation satisfies
\begin{equation*}
    u\in U_* := \m{H10}\cap\m{H2},
\end{equation*}
see \cref{thm:weak:regularity:smooth} or \cref{thm:weak:regularity:convex}, and endow $U(h):=U_* + V_h$ with the norm
\begin{equation*}
    \norm{dg}{w}_*^2 := \norm{dg}{w}^2 + \sum_{K\in\calT_h} h_K \norm*{\m{l2}(\partial K)^n}{\nabla_h w}^2.
\end{equation*}
With respect to this norm, $a_h$ is bounded in $u$.
\begin{lemma}\label{lem:dg:elliptic_bounded}
    If $\calT_h$ is a shape-regular triangulation of $\Omega$, then there exists a constant $C>0$ independent of $h$ such that
    \begin{equation*}
        a_h(u,v_h) \leq C \norm{dg}{u}_* \norm{dg}{v_h}\quad\text{for all }u\in U(h),\ v_h\in V_h.
    \end{equation*}
\end{lemma}
\begin{proof}
    We estimate for $u\in U_*$ and $v_h\in V_h$ each term in $a_h(u,v_h)$ separately.
    \begin{enumerate}[(i)]
        \item For the first term, the Cauchy--Schwarz inequality immediately yields
            \begin{equation*}
                (\nabla_h u,\nabla_h v_h) \leq \norm*{\m{L2}^n}{\nabla_h u}\norm*{\m{L2}^n}{\nabla_h v_h}.
            \end{equation*}

        \item For the second term, we apply the estimate \eqref{eq:dg:sip:estimate2} for $v= v_h$ and $w=u$ to obtain
            \begin{equation*}
                \sum_{F\in\partial\calT_h}\int_F\jump{u}\cdot\avg{\nabla_h v_h} \leq
                C \norm*{\m{L2}^n}{\nabla_h v_h}|u|_{\Gamma_h}.
            \end{equation*}

        \item For the third term, we apply the estimate \eqref{eq:dg:sip:estimate1} for $v= u$ and $w=v_h$ to obtain
            \begin{equation*}
                \sum_{F\in\partial\calT_h}\int_F\jump{v_h}\cdot\avg{\nabla_h u} \leq
                \left(\sum_{K\in\calT_h} h_K \norm*{\m{l2}(\partial K)^n}{\nabla_h u}^2\right)^{\frac12} |v_h|_{\Gamma_h},
            \end{equation*}
            using again that $h_F \leq h_K$ for all faces $F$ of $K$.

        \item For the last term, we again obtain from the Cauchy--Schwarz inequality that
            \begin{equation*}
                \sum_{F\in\partial\calT_h}\int_F \frac\eta{h_F}\jump{u}\jump{v_h}\,ds \leq \eta |u|_{\Gamma_h} |v_h|_{\Gamma_h}.
            \end{equation*}
    \end{enumerate}
    Since all of the terms appearing on the right-hand sides are parts of the definition of $\norm{dg}{u}_*$ and $\norm{dg}{v_h}$, respectively, we conclude the desired estimate.
\end{proof}
Note that for $u_h\in V_h$, we could have used in step (iii) the estimate \eqref{eq:dg:sip:estimate2} as well to avoid the extra term in the definition of $\norm{dg}{u}_*$. From this, we have for $u_h\in V_h$ that
\begin{equation*}
    a_h(u_h,v_h) \leq C \norm{dg}{u} \norm{dg}{v_h},
\end{equation*}
i.e., the continuity necessary to apply the \nameref{thm:BNB} theorem.
\begin{cor}\label{cor:dg:elliptic_wellposed}
    If $\calT_h$ is a shape-regular triangulation of $\Omega$ and $\eta$ is sufficiently large,
    then there exists a unique solution $u_h \in V_h$ to \eqref{eq:dg:sip}. Furthermore, there exists a constant $C>0$ such that
    \begin{equation*}
        \norm{dg}{u_h} \leq C \norm{L2}{f}.
    \end{equation*}
\end{cor}

With the same arguments as in \cref{lem:dg:consistent}, one can show that any $u\in \m{H2}$ satisfies $\jump{u}_F=0$ and $\jump{\nabla u}_F = 0$. Hence the exact  solution $u\in U_*$ satisfies \eqref{eq:dg:sip}, and we can apply the
\nameref{thm:strang2} to obtain
\begin{equation*}
    \norm{dg}{u-u_h} \leq C \inf_{w_h\in V_h}\norm{dg}{u-w_h}_*.
\end{equation*}
Estimating the best approximation error by the interpolation error and applying the usual estimates for each term in $\norm{dg}{\cdot}_*$ (noting that the appearance of $h_K$ in the gradient term compensates for the lower power $h_K^{k-1}$ in the corresponding estimate), we obtain for a solution $u\in H^{k+1}(\Omega)$ the a priori error estimate
\begin{equation*}
    \norm{dg}{u-u_h} \leq C h^k |u|_{H^{k+1}(\Omega)}.
\end{equation*}
Due to the face term in $\norm{dg}{\cdot}$, this estimate is optimal; a duality trick then yields a convergence rate of $\mathcal{O}(h^{k+1})$ for the discretization error in the $L^2$ norm (which is useful even for $k=0$).

\section{Implementation}

As in the standard Galerkin approach, the assembly of the stiffness matrix is carried out by choosing a suitable nodal basis $\phi_1,\dots,\phi_N$ of $V_h$ and computing the entries  $a_h(\phi_i,\phi_j)$ element-wise by transformation to a reference element. For discontinous Galerkin methods, there are two important differences:
\begin{enumerate}
    \item Since the functions in $V_h$ can be discontinuous across elements, the degrees of freedom of each element decouple from the remaining elements.
    \item There are terms arising from integration over interior as well as boundary faces.
\end{enumerate}
These require some modifications to the assembly procedure described in \cref{sec:implementation:assembly}.

Due to the first point, we can take each basis function $\phi_i$ to have support on only one element. Our set of global basis functions is thus just the union of the sets of local basis functions on each element $K\in\calT_h$ (extended to zero outside $K$), which are constructed as in \cref{chap:elements}.
Note that this implies that nodes (the interpolation points for each degree of freedom) common to multiple element domains have to be treated as distinct (e.g., a node on a vertex where $m$ elements meet corresponds to $m$ degrees of freedom, one for each element). The dimension of $V_h$ is thus equal to the sum of the local degrees of freedom over all elements, and thus greater than for standard finite elements.

In particular, if the global basis functions are enumerated such that the local basis functions in each element are numbered contiguously, the mass matrix $\mathbf{M}$ with elements $M_{ij}=(\phi_i,\phi_j)$ is then \emph{block diagonal}, where each block corresponds to one element. For the stiffness matrix $\mathbf{K}$, the terms arising from volume integrals are similarly block diagonal, but they are coupled via the terms arising from the integrals over interior faces. It is thus convenient to separately assemble the contributions to the bilinear form $a_h$ from volume integrals, interior face integrals and boundary face integrals:
\begin{itemize}
    \item The \emph{volume terms} are assembled as described in \cref{sec:implementation:assembly}, making use of the simple form of the local-to-global index.

    \item  For the \emph{interior face terms}, one needs a list \texttt{interfaces} of interior faces, which contains for each face $F$ the two elements $K_1,K_2$ sharing it, as well as the location of the face relative to each element. For each pair of basis functions from the two elements (obtained via the list \texttt{elements}), one can then (by transformation to the reference element and, if necessary, numerical quadrature) compute the corresponding integrals, recalling for the computation of jumps and averages that each local basis function is zero outside its element, and that the unit normals can be obtained from the reference element (where they are known) by transformation.

    \item The \emph{boundary terms} are similarly assembled using the list \verb!bdy_faces!, where for advection-reaction equations, one has to check on each face the sign of $\beta(x)\cdot \nu_F$ to decide whether it is part of the inflow boundary $\partial\Omega^-$ where the boundary condition has to be prescribed.
\end{itemize}

\part{Time-dependent problems}

\chapter{Variational theory of parabolic PDEs}

In this chapter, we study time-dependent partial differential equations. For example, if $-\Delta u = f$ (together with appropriate boundary conditions) describes the temperature distribution $u$ in a body due to the heat source $f$ at equilibrium, the \emph{heat equation}
\begin{equation*}
    \left\{\begin{aligned}
            \partial_t u(t,x) - \Delta u(t,x) &= f(t,x),\\
            u(0,x) &= u_0(x),
    \end{aligned}\right.
\end{equation*}
describes the evolution in time of $u$ starting from the given initial temperature distribution $u_0$ (called \emph{initial condition} in this context). This is a \emph{parabolic} equation, since the spatial partial differential operator $-\Delta$ is elliptic and only the first time derivative of $u$ appears.

\section{Function spaces}

To specify the weak formulation of parabolic problems, we first need to fix the proper functional-analytic framework. Let $T>0$ be a fixed time and $\Omega\subset\R^n$ be a domain, and set $Q:=(0,T)\times\Omega$. To respect the special role of the time variable, we consider a real-valued function $u(t,x)$ on $Q$ as a function of $t$ with values in a Banach space $V$ consisting of functions depending on $x$ only:
\begin{equation*}
    u:(0,T) \to V, \qquad t\mapsto u(t) \in V.
\end{equation*}
Similarly to the real-valued case, we define the following function spaces:
\begin{itemize}
    \item \emph{Hölder spaces}: For $k\geq 0$, define $\m{CkT}{V}$ as the space of all $V$-valued functions on $[0,T]$ which are $k$ times continuously differentiable with respect to $t$. Denote by $d^j_t u$ the $j$th derivative of $u$. Then $\m{CkT}{V}$ is a Banach space when equipped with the norm
        \begin{equation*}
            \norm*{\m{CkT}{V}}{u} := \sum_{j=0}^k\sup_{t\in[0,T]}  \norm*{V}{d^j_t u(t)}
        \end{equation*}

    \item \emph{Lebesgue spaces} (also called \emph{Bochner spaces}):\footnote{For a rigorous definition, see \cite[\S\,24]{Wloka:1987}} For $1\leq p\leq \infty$, define $\m{LpT}{V}$ as the space of all $V$-valued functions on $(0,T)$ for which $t\mapsto\norm*{V}{u(t)}$ is a function in $L^p(0,T)$. This is a Banach space if equipped with the norm
        \begin{equation*}
            \norm*{\m{LpT}{V}}{u} := \begin{cases}
                \left(\int_0^T \norm*{V}{u(t)}^p\,dt\right)^{\frac1p}  & \text{ if }p<\infty,\\
                \mathrm{ess}\sup_{t\in(0,T)} \norm*{V}{u(t)} & \text{ if } p = \infty.
            \end{cases}
        \end{equation*}

    \item \emph{Sobolev spaces}: If $u\in\m{LpT}{V}$ has a weak derivative $d_t u$ (defined in the usual fashion via the integration-by-parts formula \eqref{eq:weak:derivative} with \emph{scalar} test functions) in $\m{LpT}{V}$, we say that $u\in W^{1,p}(0,T;V)$. This is a Banach space if equipped with the norm
        \begin{equation*}
            \norm*{ W^{1,p}(0,T;V)}{u} := \norm*{\m{LpT}{V}}{u} +  \norm*{\m{LpT}{V}}{d_t u}.
        \end{equation*}
        More generally, for $1 < p,q <  \infty$ and two reflexive Banach spaces $V_0,V_1$ with continuous embedding $V_0\hookrightarrow V_1$, we set        \begin{equation*}
            W^{1,p,q}(0,T;V_0,V_1) := \setof{v\in \m{LpT}{V_0}}{d_tv \in \m{LqT}{V_1}}.
        \end{equation*}
        This is a Banach space if equipped with the norm
        \begin{equation*}
            \norm*{W(0,T;V_0,V_1)}{u} := \norm*{\m{LpT}{V_0}}{u} +  \norm*{\m{LqT}{V_1}}{d_t u}.
        \end{equation*}
\end{itemize}

Of particular importance is the case $q = p/(p-1)$ (i.e., $1/p+1/q=1$) and $V_1 = V_0^*$, since in this case $\m{LpT}{V}^*$ can be identified with $\m{LqT}{V^*}$;\footnote{see, e.g., \cite[Theorem 8.20.3]{Edwards1965}} this is relevant because we later want to test $d_j u(t)$ with $v\in \m{LpT}{V}$.
We can then transfer (via mollifiers)\footnote{For proofs of this and the following result, see, e.g., \cite[Proposition III.1.2, Corollary III.1.1]{Showalter:1997}, \cite[Theorem 25.5 (with obvious modifications)]{Wloka:1987}} the usual calculus rules to $W^{p}(0,T;V):=W^{1,p,q}(0,T;V,V^*)$.

Similarly to the Rellich--Kondrachov theorem, we can now ask whether we can use the integrability of $d_tu$ to obtain more regularity for $u$ itself and, in particular, to deduce that $u$ is continuous in time.
This requires an additional assumption linking $V$ and $V^*$.
Let $V$ be a reflexive Banach space with continuous and dense embedding into a Hilbert space $H$. Identifying $H^*$ with $H$ using the Riesz representation theorem, we have
\begin{equation*}
    V\hookrightarrow H \cong H^* \hookrightarrow V^*
\end{equation*}
with dense embeddings. We call $(V,H,V^*)$ \emph{Gelfand} or \emph{evolution triple}.
\begin{theorem}\label{thm:aubin}
    Let $1<p<\infty$ and $(V,H,V^*)$ be a Gelfand triple. Then the embedding
    \begin{equation*}
        W^{p}(0,T;V) \hookrightarrow \m{CT}{H}
    \end{equation*}
    is continuous.
\end{theorem}
This result guarantees that functions in $W^{p}(0,T;V)$ have well-defined traces $u(0),u(T)\in H$, which is important to make sense of the initial condition $u(0)=u_0$.

We also need the following integration by parts formula (where now the test function is Banach-space valued).
\begin{lemma}\label{lem:evolution:intpart}
    Let $(V,H,V^*)$ be a Gelfand triple. Then for every $u,v\in W^{p}(0,T;V)$,
    \begin{equation*}
        \frac{d}{dt}\scalprod*{H}{u(t)}{v(t)} = \dual{V}{d_t u(t)}{v(t)} + \dual{V}{d_t v(t)}{u(t)}\quad\text{for a.e. } t\in (0,T),
    \end{equation*}
    and hence
    \begin{equation*}
        \int_0^T\dual{V}{d_t u(t)}{v(t)}\,dt = \scalprod*{H}{u(T)}{v(T)}- \scalprod*{H}{u(0)}{v(0)} -  \int_0^T\dual{V}{d_t v(t)}{u(t)}\,dt.
    \end{equation*}
\end{lemma}
In the following, we focus for simplicity only the case $p=q=2$, for which we set $W(0,T;V) := W^{2}(0,T;V)$.

\section{Weak solution of parabolic PDEs}

We can now formulate our parabolic evolution problem. Let for almost every $t\in(0,T)$ a bilinear form $a(t;\cdot,\cdot):V\times V\to \R$ and a linear form $f\in \m{L2T}{V^*}$ as well as a $u_0\in H$ be given.
The problem in strong form (in time) is then to find $u\in W(0,T;V)$ such that
\begin{equation}\label{eq:evolution:parabolic}
    \left\{\begin{aligned}
            \dual{V}{d_t u(t)}{v} + a(t;u(t),v) &= \dual{V}{f(t)}{v} \text{ for all } v\in V \text{ and a.e. } t\in(0,T),\\
            u(0) &= u_0.
    \end{aligned}\right.
\end{equation}
(For, e.g., the heat equation, we have $V=\m{H10}\hookrightarrow \m{L2}=H$ and $a(t;u,v) = (\nabla u,\nabla v)$.)
Just as in the stationary case, we now formulate this in fully variational or weak form. For simplicity, assume $u_0 = 0$ (the inhomogeneous case can be treated in the same fashion as inhomogeneous Dirichlet conditions) and consider the Banach spaces
\begin{equation*}
    X=\setof{w\in W(0,T;V)}{w(0) = 0}, \qquad Y=\m{L2T}{V},
\end{equation*}
such that $Y^* =\m{L2T}{V^*}$.
Setting
\begin{equation*}
    b:X\times Y\to\R,\qquad b(u,y)  = \int_0^T \dual{V}{d_t u(t)}{y(t)} + a(t;u(t),y(t))\,dt
\end{equation*}
and
\begin{equation*}
    \dual{Y}{f}{y} = \int_0^T \dual{V}{f(t)}{y(t)}\,dt,
\end{equation*}
we look for $u\in X$ such that
\begin{equation}\label{eq:evolution:weak}
    b(u,y) = \dual{Y}{f}{y} \quad \text{ for all } y\in Y.
\end{equation}
The equivalence to \eqref{eq:evolution:parabolic} follows from considering $y(t)=\phi(t)v$ for arbitrary $\phi\in C_0^\infty(0,T)$ and $v\in V$ and using the fundamental theorem of the calculus of variations.\footnote{see, e.g., \cite[Lemma 65.4]{Ern:2021c}.}

Well-posedness of \eqref{eq:evolution:parabolic} can then be shown using the \nameref{thm:BNB} theorem.
\begin{theorem}\label{thm:evolution:well-posed}
    Assume that the bilinear form $a(t;\cdot,\cdot):V\times V\to\R$ satisfies the following properties:
    \begin{enumerate}[(i)]
        \item The mapping $t\mapsto a(t;u,v)$ is measurable for all $u,v\in V$.
        \item There exists $M>0$ such that $|a(t;u,v)| \leq M \norm*{V}{u}\norm*{V}{v}$ for almost every $t\in(0,T)$ and all $u,v\in V$.
        \item There exists $\alpha >0 $ such that $a(t;u,u)\geq \alpha \norm*{V}{u}^2$ for almost every $t\in(0,T)$ and all $u\in V$.
    \end{enumerate}
    Then \eqref{eq:evolution:weak} has a unique solution $u\in W(0,T;V)$ satisfying
    \begin{equation*}
        \norm*{W(0,T;V)}{u}\leq C\norm*{Y^*}{f}.
    \end{equation*}
\end{theorem}
\begin{proof}
    Continuity of $b$ and $y\mapsto \dual{Y}{f}{y}$ follows from their definition and the continuity of $a$. To verify the inf--sup condition, we define for almost every $t\in (0,T)$ the operator
    \begin{equation*}
        A(t):V\to V^*,\qquad \dual{V}{A(t)u}{v} := a(t;u,v) \quad\text{for all }u,v\in V.
    \end{equation*}
    Continuity of $a$ implies that for almost every $t\in(0,T)$, the operator $A(t)$ is bounded with constant $M$. Similarly, coercivity of $a$ and the \nameref{thm:weak:laxmilgram} shows that $A(t)$ is an isomorphism, hence $A(t)^{-1}:V^*\to V$ is bounded as well with constant $\alpha^{-1}$. Therefore, for almost every $t\in(0,T)$ and all $v^*\in V^*$
    \begin{equation}\label{eq:evolution:well1}
        \begin{aligned}[t]
            \dual{V}{v^*}{A(t)^{-1}v^*} &= \dual{V}{A(t)A(t)^{-1}v^*}{A(t)^{-1}v^*}
            \geq \alpha \norm*{V}{A(t)^{-1}v^*}^2 \\&\geq \frac{\alpha}{M^2}\norm*{V^*}{v^*}^2.
        \end{aligned}
    \end{equation}

    For arbitrary $u\in X$ and $\mu >0 $, set $z = A(t)^{-1}d_t u + \mu u$. By the triangle inequality, the uniform continuity of $A(t)^{-1}$, and the definition of the norms in $X$ and $Y$, we have that
    \begin{equation*}
        \norm*{Y}{z}^2 \leq 2\alpha^{-2} \int_0^T\norm*{V^*}{d_t u(t)}^2\,dt + 2\mu^2 \int_0^T \norm*{V}{u(t)}^2\,dt \leq c \norm*{X}{u}^2,
    \end{equation*}
    and thus in particular that $z\in Y$.
    Moreover, using \eqref{eq:evolution:well1}, integration by parts, and continuity of $A(t)$ and $A(t)^{-1}$, respectively, we can estimate term by term in
    \begin{equation*}
        \begin{aligned}
            b(u,z) &= \int_0^T \dual{V}{d_tu(t) + A(t)u(t)}{A(t)^{-1}d_tu(t) + \mu u(t)}\,dt\\
            &\geq \frac{\alpha}{M^2}\int_0^T\norm*{V^*}{d_t u(t)}^2\,dt + \frac\mu2 \norm*{H}{u(T)}^2 - \frac{M}{\alpha}\int_0^T\norm*{V}{u(t)}\norm*{V^*}{d_t u(t)}\,dt \\
            \MoveEqLeft[-1] + \mu\alpha\int_0^T\norm*{V}{u(t)}^2\,dt\\
            &\geq \frac{\alpha}{2M^2}\int_0^T\norm*{V^*}{d_t u(t)}^2\,dt + \left(\mu\alpha - \tfrac{M^4}{2\alpha^3}\right)\int_0^T\norm*{V}{u(t)}^2\,dt,
        \end{aligned}
    \end{equation*}
    using the generalized Young's inequality with $\eps=\alpha/M^2$.

    Taking $\mu = M^4\alpha^{-4}$, the term in parenthesis is positive, which yields (for generic constants $c>0$)
    \begin{equation*}
        b(u,z)  \geq c \norm*{X}{u}^2 \geq c\norm*{X}{u}\norm*{Y}{z}.
    \end{equation*}
    This implies the inf--sup condition via
    \begin{equation*}
        \inf_{u\in X}\sup_{y\in Y}\frac{b(u,y)}{\norm*{X}{u}\norm*{Y}{y}}\geq
        \inf_{u\in X}\frac{b(u,z)}{\norm*{X}{u}\norm*{Y}{z}} \geq c.
    \end{equation*}

    It remains to show that the injectivity condition holds. Assume $y\in Y$ is such that $b(u,y) =0$ for all $u\in X$. For any $\phi\in C_0^\infty(0,T)$ and $v\in V$, we have $\phi v\in X$. Due to the definition of the weak time derivative and $b(\phi v,y)=0$, we thus obtain that
    \begin{equation*}
        \begin{aligned}
            \int_0^T\dual{V}{d_t y(t)}{v}\phi(t)\,dt &= - \int_0^T\dual{V}{d_t \phi(t)v}{y(t)}\,dt
            = \int_0^T a(t;\phi(t)v,y(t))\,dt\\
            &= \int_0^T \dual{V}{A(t)^*y(t)}{v}{\phi(t)}\,dt,
        \end{aligned}
    \end{equation*}
    and hence (by density of $C_0^\infty(0,T)$ in $\m{l2}(0,T)$) that $d_t y(t) = A(t)^*y(t)$ for almost all $t\in(0,T)$. In particular, we deduce that $d_t y  \in \m{L2T}{V^*}$ and therefore $y\in W(0,T;V)$.

    Since $d_ty = A^*y$ in $Y^*$ and $tv\in X\hookrightarrow Y$ for any $v\in V$, we obtain using \cref{lem:evolution:intpart} that
    \begin{equation*}
        \begin{aligned}
            0 &=   \int_0^T \dual{V}{-d_t y(t)}{tv} + \dual{V}{A(t)^*y(t)}{tv} \, dt \\
            &= -\scalprod*{H}{y(T)}{Tv} + \int_0^T \dual{V}{d_t (tv)}{y(t)} + a(t;tv,y(t)) \, dt\\
            &= -T\scalprod*{H}{y(T)}{v}.
        \end{aligned}
    \end{equation*}
    By density of $V$ in $H$, this implies that $y(T) =0$. Similary, $y\in W(0,T;V)$ and the first part of \cref{lem:evolution:intpart} yields
    \begin{equation*}
        \begin{aligned}
            0 &= \int_0^T -\dual{V}{d_t y(t)}{y(t)} + \dual{V}{A(t)^*y(t)}{y} \, dt \\
            &\geq \int_0^T -\frac{d}{dt}\left(\frac12\norm*{H}{y(t)}^2\right) +\alpha\norm*{V}{y(t)}^2\,dt\\
            &= \frac12 \norm*{H}{y(0)}^2 + \alpha\norm*{Y}{y}^2
        \end{aligned}
    \end{equation*}
    and hence $y = 0$. We can thus apply the \nameref{thm:BNB} theorem, and the claim follows.
\end{proof}

\chapter{Galerkin approach for parabolic problems}

To obtain a finite-dimensional approximation of \eqref{eq:evolution:parabolic}, we need to discretize in both space and time: either separately (combining finite elements in space with a time stepping method for ordinary differential equations) or all-at-once (using a Galerkin approach with suitable discrete test spaces). Only a brief overview over the different approaches is given here.

\section{Time stepping methods}

These approaches can be further discriminated based on the order of operations:

\paragraph{Method of lines} This method starts with a discretization in space to obtain a (very large) system of ordinary differential equations, which are then solved with one of the vast number of available methods. In the context of finite element methods, we use a discrete space $V_h$ of piecewise polynomials defined on the triangulation $\calT_h$ of the domain $\Omega$. Given a nodal basis $\{\phi_j\}_{j=1}^{N_h}$ of $V_h$, we approximate the unknown solution as $u_h(t,x) = \sum_{j=1}^{N_h} U_j(t)\phi_j(x)$. Letting $\calP_h$ denote the $L^2$ projection on $V_h$ and using the mass matrix $\mathbf{M}_{ij} = (\phi_i,\phi_j)$ and the (time-dependent) stiffness matrix $\mathbf{K}(t)_{ij} = a(t;\phi_i,\phi_j)$ yields the following linear system of ordinary differential equations for the coefficient vector $U(t) = (U_1(t),\dots U_{N_h}(t))^T$:
\begin{equation*}
    \left\{\begin{aligned}
            \mathbf{M} \frac{d}{dt}U(t) +\mathbf{K}(t) U(t) = \mathbf{M} F(t),\\
            U(0) = U_0,
    \end{aligned}\right.
\end{equation*}
where $U_0$ and $F(t)$ are the coefficients vectors of $\calP_h u_0$ and $\calP_h f(t)$, respectively. The choice of integration method for this system depends on the properties of $\mathbf{K}$ (such as its stiffness, which can lead to numerical instability). Some details can be found, e.g., in \cite[Chapter 6.1]{Ern:2004}.

\paragraph{Rothe's method} This method consists in treating \eqref{eq:evolution:parabolic} as an ordinary differential equation in the Banach space $V$, which is discretized in time by replacing the time derivative $d_t u$ by a difference quotient:
\begin{itemize}
    \item The \emph{implicit Euler scheme} uses the backward difference quotient
        \begin{equation*}
            d_t u(t+\tau) \approx \frac{u(t+\tau)-u(t)}{\tau}
        \end{equation*}
        for $\tau>0$ at time $t+\tau$ to obtain for given $u(t)$ and unknown $u(t+\tau)\in V$  the  \emph{stationary} partial differential equation
        \begin{equation*}
            \scalprod*{H}{u(t+\tau)}{v} + \tau\, a(t+\tau;u(t+\tau),v) = \scalprod*{H}{u(t)}{v} + \tau \dual{V}{f(t+\tau)}{v}
        \end{equation*}
        for all $v\in V$.
    \item The \emph{Crank--Nicolson scheme} uses the central difference quotient
        \begin{equation*}
            d_t u(t+\tfrac{\tau}{2}) \approx \frac{u(t+\tau)-u(t)}{\tau}
        \end{equation*}
        for $\tau>0$ at time $t+\frac{\tau}2$ to obtain
        \begin{multline*}
            \scalprod*{H}{u(t+\tau)}{v} + \tfrac{\tau}{2}\, a(t+\tfrac{\tau}2;u(t+\tau),v) =
            \scalprod*{H}{u(t)}{v} - \tfrac{\tau}{2}\, a(t+\tfrac{\tau}2;u(t),v) \\ + \tau \dual{V}{f(t+\tfrac{\tau}{2})}{v}
        \end{multline*}
        for all $v\in V$.
\end{itemize}

Starting with $t=0$, these are then approximated and solved in turn for $u(t_m)$, $t_m := m \tau$, using a finite element discretization in space. This approach is discussed in detail in \cite[Chapters 7--9]{Thomee:2006}.
The advantage of Rothe's method is that at each time step, a different spatial discretization can be used.

\section{Galerkin methods}

Proceeding as in the stationary case, we can apply a Galerkin approximation to
\eqref{eq:evolution:weak} by replacing $X$ and $Y$ with finite-dimensional spaces $X_h$ and $Y_h$. Again, we can further discriminate between conforming and non-conforming approaches.

\paragraph{Conforming Galerkin methods}

In a conforming approach, we choose $X_h\subset X$ and $Y_h\subset Y$ and seek $u_h\in X_h$ such that
\begin{equation}\label{eq:evolution:cg}
    \int_0^T \dual{V}{d_t u_h(t)}{y_h(t)} + a(t;u_h(t),y_h(t)) \,dt = \int_0^T \dual{V}{f(t)}{y_h(t)}\,dt
\end{equation}
for all $y_h\in Y_h$. We now choose the discrete spaces as tensor products in space and time: Let
\begin{equation*}
    0=t_0 < t_1 < \dots < t_N = T
\end{equation*}
and choose for each $t_m$, $ 1\leq m \leq N$, a (possibly different) finite-dimensional subspace $V_m \subset V$. Let $P_r(t_{m-1},t_{m};V_m)$ denote the space of polynomials on the interval $[t_{m-1},t_m]$ with degree up to $r$ with values in $V_m$. Then we define for $r\geq 1$
\begin{align*}
    X_h &= \setof{w_h\in  \m{CT}{V}}{w_h|_{[t_{m-1},t_{m}]}\in P_r(t_{m-1},t_{m};V_m),\ 1\leq m\leq N,\ w_h(0)=u_0},\\
    Y_h &= \setof{y_h\in  \m{L2T}{V}}{y_h|_{(t_{m-1},t_{m}]}\in P_{r-1}(t_{m-1},t_{m};V_m),\ 1\leq m \leq N}.
\end{align*}
Since this is a conforming approximation, we can deduce well-posedness of the corresponding discrete problem in the usual fashion (noting that $d_t u_h\in Y_h$ for $u_h\in X_h$). (Since functions in $X$ -- and hence in $X_h$ -- are continuous in time by \cref{thm:aubin}, this approach is often called \emph{continuous Galerkin} or $cG(r)$ method.)

This approach is closely related to Rothe's method. Consider the case $r=1$ (i.e., piecewise linear in time) and, for simplicity, a time-independent bilinear form.
We also assume that we choose the same space discretization at each time step, i.e., that $V_1 = \dots = V_N =V_h$.
Since functions in $X_h$ are continuous at $t=t_m$ for all $0\leq m \leq N$ and linear on each intervall $[t_{m-1},t_m]$, we can write
\begin{equation*}
    u_h(t) = \frac{t_m-t}{t_m-t_{m-1}}u_h(t_{m-1}) + \frac{t-t_{m-1}}{t_m-t_{m-1}} u_h(t_m), \qquad t\in[t_{m-1},t_m],
\end{equation*}
with coefficients $u_h(t_{m-1}),u_h(t_m)\in V_h$. (For $t_0=0$, we fix $u_h(t_0) = u_0$.)
Similarly, functions in $Y_h$ are constant in time and thus
\begin{equation*}
    y_h(t) \equiv y_h(t_{m}) =:v_h \in V_h, \qquad t\in (t_{m-1},t_m].
\end{equation*}
Inserting this into \eqref{eq:evolution:cg} and setting $\tau_m :=t_m-t_{m-1}$ yields for all $v_h\in V_h$ that
\begin{equation*}
    \dual{V}{u_h(t_{m})-u_h(t_{m-1})}{v_h} + \frac{\tau_m}2 a(u_h(t_{m-1})+u_h(t_m),v_h) = \int_{t_{m-1}}^{t_m} \dual{V}{f(t)}{v_h}\,dt,
\end{equation*}
which is a modified Crank--Nicolson scheme (which, in fact, can be obtained by approximating the integral on the right-hand side using the midpoint rule, which is exact for $y_h\in Y_h$).\footnote{If the discrete spaces are different for each time interval, we need to use the $H$-projection of $u_{m-1}$ on $V_m$.}
For this method, one can show error estimates of the form\footnote{\cite[Theorem 7.8]{Thomee:2006}}
\begin{equation*}
    \norm{L2}{u_h(t_m) - u(t_m)}\leq C(h^s\norm*{H^s(\Omega)}{u_0} + \tau^2 \norm*{H^4(\Omega)}{u_0}),
\end{equation*}
for $f=0$ and $u_0 \neq 0$, where $s$ depends on the accuracy of the spatial discretization, and $\tau=\max_{1\leq m\leq N} \tau_m$.

\paragraph{Discontinuous Galerkin methods}
Instead of enforcing continuity of the discrete solution $u_h$ through the definition of $X_h$, we can also use $X_h=Y_h$ and modify the bilinear form. Let $J_m := (t_{m-1},t_m]$ denote the half-open interval between two time steps of length $\tau_m = t_m - t_{m-1}$. Then we set for $r\geq 0$
\begin{equation*}
    X_h = Y_h = \setof{y_h\in  \m{L2T}{V}}{y_h|_{J_m}\in P_{r}(t_{m-1},t_{m};V_m),\ 1\leq m \leq N}\subset Y,
\end{equation*}
where $V_m$ is again a finite-dimensional subspace of $V$.
Note that functions in $X_h$ can be discontinuous at the points $t_m$ but are continuous from the left with limits from the right, and so we will write for $u_h\in X_h$
\begin{equation*}
    u_m := u_h(t_m) = \lim_{\eps \to 0} u_h(t_m -\eps) ,\qquad u_m^+ := \lim_{\eps\to 0} u_h(t_m+\eps)
\end{equation*}
and
\begin{equation*}
    \jump{u_h}_m = u_m^+-u_m.
\end{equation*}
Similarly to the stationary case, we now define the discrete bilinear form
\begin{equation*}
    \begin{aligned}
        b_h(u_h,y_h) &= \sum_{m=1}^N \int_{J_m}\scalprod*{V}{d_t u_h(t)}{y_h(t)} + a(t;u_h(t),y_h(t))\,dt
        + \sum_{m=1}^{N} \scalprod*{H}{\jump{u_h}_{m-1}}{y_{m-1}^+}
    \end{aligned}
\end{equation*}
(which can be derived by integration by parts on each interval $J_m$ and rearranging the jump terms).
As $0\notin J_1$, we will need to specify $u_h(0)=u_0$ separately, which we do by setting $\jump{u_h}_0 := u_0^+ - u_0$. Note that this makes $b_h$ \emph{affine} instead of bilinear unless $u_0=0$. (In other words, we should actually split the jump term for $m=0$ into the part involving $u_0^+$, which remains part of $b_h$, and the part involving $u_0$, which should be part of the right-hand side. However, we stick with the above formulation for the sake of presentation.)

We then search for $u_h\in X_h$ satisfying
\begin{equation}\label{eq:evolution:dg}
    b_h(u_h,y_h) = \dual{Y}{f}{y_h} \qquad\text{for all }y_h\in X_h.
\end{equation}
Since the exact solution $u\in X$ is continuous and satisfies $u(0) = u_0$, we have
\begin{equation*}
    b_h(u,y_h) = b(u,y_h) =  \dual{Y}{f}{y_h} \qquad\text{for all }y_h\in X_h,
\end{equation*}
and hence this is a consistent approximation.
To prove well-posedness of the discrete problem, we again define a discrete ``jump-norm''
\begin{equation*}
    \norm{dg}{u_h}^2 = \sum_{m=1}^N \int_{J_m} \norm*{H}{d_t u_h(t)}^2 + \norm*{V}{u_h(t)}^2\,dt + \sum_{m=1}^N\norm*{H}{\jump{u_h}_m}^2.
\end{equation*}
We can then proceed as in the proof of \cref{thm:bnb_h}.
\begin{theorem}\label{thm:evolution:dg_well-posed}
    Under the assumptions of \cref{thm:evolution:well-posed}, there exists a unique solution $u_h\in X_h$ to \eqref{eq:evolution:dg}, and
    \begin{equation*}
        \norm{dg}{u_h} \leq C \left(\norm*{Y^*}{f}^2 + \norm*{H}{u_0}^2\right).
    \end{equation*}
\end{theorem}
\begin{proof}
    Continuity of $b_h$ with respect to $\norm{dg}{\cdot}$ follows from the definition. It remains to show injectivity of $B_h:X_h\to Y_h^*$, $u_h\mapsto b_h(u_h,\cdot)$, (which suffices for bijectivity since $X_h=Y_h$ are finite-dimensional). Instead of verifying the inf--sup condition, we do this directly. Let $u_h\in X_h$ satisfy $b_h(u_h,y_h) = 0$ for all $y_h\in X_h$ with $u_0=0$. Since functions in $Y_h$ can be discontinuous at the time points $t_m$, we can insert $y_h = \1_{J_m}u_h\in Y_h$ for each $1\leq m\leq N$, where $\1_{J_m}(t) = 1$ if $t\in J_m$ and zero else. We start with $J_1= (t_{0},t_1]$. Since $\1_{J_1}$ is constant on $J_1$ and zero outside $J_1$, we have using $u_0=0$ that
    \begin{equation*}
        \begin{aligned}
            0 &= b_h(u_h,\1_{J_1}u_h) \\
            &= \int_{J_1} \dual{V}{d_t u_h(t)}{u_h(t)} + a(t;u_h(t),u_h(t))\,dt + \scalprod*{H}{u_{0}^+-u_{0}}{u_{0}^+}\\
            &\geq \half\norm*{H}{u_1}^2 -\half\norm*{H}{u_{0}^+}^2 + \alpha\int_{J_1} \norm*{V}{u_h(t)}^2\,dt + \norm*{H}{u_{0}^+}^2\\
            &\geq \half\norm*{H}{u_{1}}^2 +  \alpha\int_{J_1} \norm*{V}{u_h(t)}^2\,dt .
        \end{aligned}
    \end{equation*}
    Hence, $u_h|_{J_1} = 0$ and $u_{1}=0$, and we can proceed in a similar way for $J_{2},J_{3},\dots,J_N$ to deduce that $u_h = 0$. The estimate then follows from bijectivity using the closed range theorem.
\end{proof}

Before we address a priori error estimates, we discuss how to formulate discontinuous Galerkin methods as time stepping methods.
For simplicity, we again assume that the bilinear form $a$ is time-independent and that $V_1 = \dots = V_N =V_h$.
First consider the case $r=0$, i.e., piecewise constant functions in time. Then $d_t (u_h|_{J_m}) = 0$ and $u_h|_{J_m} \equiv u_m = u_{m-1}^+\in V_h$. Using as test functions $y_h = \1_{J_m} v_h\in Y_h$ for arbitrary $v_h\in V_h$ and $m=1,\dots, N$, we obtain
\begin{equation*}
    \scalprod*{H}{u_m}{v_h} + \tau_m\, a(u_m,v_h) = \scalprod*{H}{u_{m-1}}{v_h} +  \int_{J_m}\dual{V}{f(t)}{v_h}\,dt
\end{equation*}
for all $v_h\in V_h$, which is a variant of the implicit Euler scheme.
For $r=1$ (piecewise linear functions), we make the ansatz
\begin{equation*}
    u_h|_{J_m} (t) = u_m^0 + \frac{t-t_{m-1}}{\tau_m} u_m^1 \in X_h
\end{equation*}
for coefficients $u_m^0,u_m^1\in V_h$ (such that $u_{m-1}^+ = u_m^0$ and $u_m = u_m^0 + u_m^1$). Again, we choose for each $J_m$ test functions which are zero outside $J_m$; specifically, we take $\1_{J_m}(t)v_h$ and $\1_{J_m}(t) \frac{t-t_{m-1}}{\tau_m}w_h$ for arbitrary $v_h,w_h\in V_h$. Inserting these in turn into the bilinear form and computing the integrals yields the coupled system
\begin{align*}
    \scalprod*{H}{u_m^0}{v_h} + \tau_m\, a(u_m^0,v_h) +
    \scalprod*{H}{u_m^1}{v_h} +\frac{\tau_m}{2}\, a(u_m^1,v_h) \\
    \MoveEqLeft[8]= \scalprod*{H}{u_{m-1}}{v_h} +  \int_{J_m}\dual{V}{f(t)}{v_h}\,dt\quad\text{for all }v_h\in V_h,\\
    \frac{\tau_m}{2}\, a(u_m^0,w_h) +\frac12
    \scalprod*{H}{u_{m}^1}{w_h} +\frac{\tau_m}{3}\, a(u_m^1,w_h) \\
    \MoveEqLeft[8]= \frac{1}{\tau_m} \int_{J_m}(t-t_{m-1})\dual{V}{f(t)}{w_h}\,dt \quad\text{for all }w_h\in V_h.
\end{align*}
By solving this system successively at each time step and setting $u_m = u_m^0 + u_m^1$, we obtain the approximate solution $u_h$. Similarly, discontinuous Galerkin methods for $r\geq 2$ lead to $(r+1)$-stage implicit Runge--Kutta time-stepping schemes.

\section{A priori error estimates for discontinuous Galerkin methods}

To derive a priori error estimates for discontinuous Galerkin approximations, we will first show a discrete  stability result. For simplicity, we assume from now on that the bilinear form $a$ is time-independent and  symmetric, and that $V_1 = \dots = V_N =V_h$. Let $A:V\to V^*$ again denote the operator corresponding to the bilinear form $a$, i.e., $\dual{V}{Au}{v} = a(u,v)$ for all $u,v\in V$. We also assume for the sake of presentation that $f\in L^2(0,T; H)$ and that the discrete solution $u_h$ is sufficiently regular that $Au_h(t)\in H$.
\begin{theorem}\label{thm:evolution:dg_stability}
    For given $f\in\m{L2T}{H}$ and $u_0\in H$, the solution $u_h$ of \eqref{eq:evolution:dg} satisfies
    \begin{equation*}
        \sum_{m=1}^N\left(\int_{J_m}\norm*{H}{d_t u_h(t)}^2 + \norm*{H}{Au_h(t)}^2\,dt + \tau_m^{-1}\norm*{H}{\jump{u_h}_{m-1}}^2\right) \leq C \left(\norm*{L^2(0,T; H)}{f}^2 + \norm*{H}{u_0}^2\right).
    \end{equation*}
\end{theorem}
\begin{proof}
    We estimate in turn each term on the left-hand side by inserting suitable test functions $y_h$ in \eqref{eq:evolution:dg}.

    \emph{Step 1.} To estimate $ \norm*{H}{A u_h(t)}$, we set $y_h = \1_{J_m}Au_h$ for $1\leq m\leq N$ to obtain
    \begin{equation*}
        \begin{multlined}
            \int_{J_m} \scalprod*{H}{d_t u_h(t)}{A u_h(t)} + \norm*{H}{A u_h(t)}^2\,dt + \scalprod*{H}{\jump{u_h}_{m-1}}{(Au_h)_{m-1}^+}\\ = \int_{J_m} \scalprod*{H}{f(t)}{A u_h(t)}\,dt.
        \end{multlined}
    \end{equation*}
    Due to the bilinearity and symmetry of $a$, we have
    \begin{equation*}
        \begin{aligned}
            \int_{J_m}  \scalprod*{H}{d_t  u_h(t)}{Au_h(t)} \,dt &=  \int_{J_m}  a(u_h(t),d_t u_h(t)) \,dt
            =\int_{J_m} \frac{d}{dt}\left(\frac12 a(u_h(t),u_h(t)) \right)\,dt\\
            &= \frac12 a(u_m,u_m) -\frac12 a(u_{m-1}^+,u_{m-1}^+).
        \end{aligned}
    \end{equation*}
    Similarly, since $A$ is time-independent,
    \begin{equation*}
        \begin{aligned}
            \scalprod*{H}{\jump{u_h}_{m-1}}{(Au_h)_{m-1}^+}
            &= a(\jump{u_h}_{m-1},u_{m-1}^+)\\
            &= \frac12 a(\jump{u_h}_{m-1},u_{m-1}^+ +u_{m-1}+ \jump{u_h}_{m-1})\\
            &= \frac12 a(u_{m-1}^+,u_{m-1}^+) - \frac12 a(u_{m-1},u_{m-1})
            + \frac12 a(\jump{u_h}_{m-1},\jump{u_h}_{m-1}).
        \end{aligned}
    \end{equation*}
    Inserting these into the bilinear form $b_h(u_h,y_h)$ yields
    \begin{multline*}
        a(\jump{u_h}_{m-1},\jump{u_h}_{m-1}) + a(u_m,u_m) - a(u_{m-1},u_{m-1})+ 2\int_{J_m}\norm*{H}{Au_h(t)}^2\,dt  \\= 2\int_{J_m}\scalprod*{H}{f}{Au_h(t)}\,dt.
    \end{multline*}
    Summing over all $1\leq m\leq N$ yields
    \begin{multline*}
        \sum_{m=1}^N a(\jump{u_h}_{m-1},\jump{u_h}_{m-1})+  \sum_{m=1}^N\int_{J_m}2 \norm*{H}{Au_h(t)}^2\,dt \\ \leq \sum_{m=1}^N\int_{J_m} 2\scalprod*{H}{f(t)}{Au_h(t)}\,dt + a(u_0, u_0).
    \end{multline*}
    For $2\leq m\leq N$, we can simply use coercivity of $a$ to eliminate the jump terms  and apply Young's inequality to $\scalprod*{H}{f(t)}{Au_h(t)}$ to absorb the norm of $A u_h$ on $J_m$ in the left-hand side.
    For $m=1$, we use that
    \begin{equation*}
        a(\jump{u_h}_0,\jump{u_h}_0) - a(u_0,u_0) = a(u_0^+,u_0^+) -2 a(u_0,u_0^+)
    \end{equation*}
    and for $\eps>0$ the generalized Young's inequality
    \begin{equation*}
        a(u_0,u_0^+) = \scalprod*{H}{u_0}{A u_0^+} \leq \frac{\eps}{2} \norm*{H}{A u_0^+}^2 + \frac{1}{2\eps}\norm*{H}{u_0}^2.
    \end{equation*}
    Since $t\mapsto\norm*{H}{Au_h(t)}^2$ is a polynomial in $t$ of degree up to $2r$ on $J_1$, we have the estimate
    \begin{equation*}
        \tau_1 \norm*{H}{Au_0^+}^2 \leq C \int_{t_0}^{t_1}\norm*{H}{Au_h(t)}^2\,dt.
    \end{equation*}
    Choosing $\eps>0$ small enough such that $\eps C \tau_1^{-1} < 1$ yields
    \begin{equation}\label{eq:evolution:dg_stab1}
        \sum_{m=1}^N\int_{J_m} \norm*{H}{Au_h(t)}^2\,dt \leq C\left(\int_0^T \norm*{H}{f(t)}^2\,dt + \norm*{H}{u_0}^2\right).
    \end{equation}

    \emph{Step 2.} For the bound on $d_t u_h$, we use the inverse estimate
    \begin{equation*}
        \int_{J_m}\norm*{H}{y_h(t)}^2\,dt \leq C \tau_m^{-1}\int_{J_m}(t-t_{m-1})\norm*{H}{y_h(t)}^2\,dt
    \end{equation*}
    for all $y_h\in P_r(t_{m-1},t_m;V_h)$, which follows from a scaling argument in time and equivalence of norms on the finite-dimensional space $P_r(0,1;V_h)$. Now choose $y_h = \1_{J_m}(t-t_{m-1})d_t u_h$ for $1\leq m\leq N$. Since $y_{m-1}^+ = 0$, we have using the  Cauchy--Schwarz inequality that
    \begin{equation*}
        \begin{aligned}
            \int_{J_m}(t-t_{m-1})\norm*{H}{d_t u_h(t)}^2\,dt
            &= \int_{J_m}(t-t_{m-1}) \scalprod*{H}{f(t)-A u_h(t)}{d_t u_h(t)} \,dt\\
            &\leq \left(\int_{J_m}(t-t_{m-1}) \norm*{H}{f(t)-Au_h(t)}^2\,dt\right)^{\half}\\
            \MoveEqLeft[-5]
            \cdot\left(\int_{J_m}(t-t_{m-1}) \norm*{H}{d_t u_h(t)}^2\,dt\right)^{\half}.
        \end{aligned}
    \end{equation*}
    Applying the inverse estimate for $y_h = d_t u_h$, the Cauchy--Schwarz inequality for the first integral and estimating the norm there using \eqref{eq:evolution:dg_stab1} yields
    \begin{equation}\label{eq:evolution:dg_stab2}
        \sum_{m=1}^N\int_{J_m} \norm*{H}{d_t u_h(t)}^2\,dt  \leq C \left(\int_0^T\norm*{H}{f(t)}^2\,dt + \norm*{H}{u_0}^2\right).
    \end{equation}

    \emph{Step 3.} It remains to estimate the jump terms. For this, we set $y_h = \1_{J_m} \jump{u_h}_{m-1}$ for $1\leq m\leq N$. This yields
    \begin{equation*}
        \begin{aligned}
            \norm*{H}{\jump{u_h}_{m-1}}^2 &= \int_{J_m} \scalprod*{H}{f(t)-Au_h(t)}{\jump{u_h}_{m-1}} - \scalprod*{H}{d_t u_h(t)}{\jump{u_h}_{m-1}}\,dt\\
            &\leq \frac{\tau_m}{2}\int_{J_m}\norm*{H}{f(t)-A u_h(t) - d_t u_h(t)}^2 \,dt
            +\frac{1}{2\tau_m}\int_{J_m} \norm*{H}{\jump{u_h}_{m-1}}^2\,dt,
        \end{aligned}
    \end{equation*}
    where we have used the generalized Young's inequality. Since $\jump{u_h}_{m-1}$ is constant in time, we have
    \begin{equation*}
        \int_{J_m} \norm*{H}{\jump{u_h}_{m-1}}^2\,dt = \tau_m  \norm*{H}{\jump{u_h}_{m-1}}^2.
    \end{equation*}
    From \eqref{eq:evolution:dg_stab1} and \eqref{eq:evolution:dg_stab2}, we thus obtain
    \begin{equation*}
        \sum_{m=1}^N \tau_m^{-1} \norm*{H}{\jump{u_h}_{m-1}}^2 \leq  C \left(\int_0^T\norm*{H}{f(t)}^2\,dt + \norm*{H}{u_0}^2\right),
    \end{equation*}
    which completes the proof.
\end{proof}

As before, we will estimate the error $u-u_h$ using the approximation properties of the space $X_h$. Due to the discontinuity of the functions in $X_h$, we can use a local projection on each time intervall $J_m$ to bound the approximation error. It will be convenient to split this error into two parts: one due to the temporal and one due to the spatial discretization.

We first consider the temporal discretization error. Let
\begin{equation*}
    X_r = \setof{y_r\in  \m{L2T}{V}}{y_r|_{J_m}\in P_{r}(t_{m-1},t_{m};V),\ 1\leq m \leq N}
\end{equation*}
and consider the local projection $\pi_r u \in X_r$ of $u\in X$ defined by $\pi_r u (t_0) = u(t_0)$ and
\begin{equation*}
    \left\{\begin{aligned}
            \pi_r u (t_{m}) &= u(t_{m}),\\
            \int_{J_m} (u(t)-\pi_r u (t))\phi(t) \,dt &= 0 & \text{ for all }\phi\in P_{r-1}(J_m;V),
    \end{aligned}\right.
\end{equation*}
for all $1\leq m\leq N$. (For $r=0$, the second condition is void.) This projection is well-defined since $u\in X$ is continuous in time, and hence the interpolation conditions make sense. Using the \nameref{thm:bramble} and a scaling argument, we obtain for sufficiently smooth $u$ the following error estimate for every $t\in J_m$, $1\leq m\leq N$:
\begin{equation}\label{eq:evolution:proj_time}
    \norm*{H}{u(t)-\pi_r u (t)} \leq C \tau_m^{r+1} \int_{J_m}\norm*{H}{d_t^{r+1}u(\tau)}\,d\tau.
\end{equation}

Similarly, we assume that for each $t\in [0,T]$ the spatial interpolation error in $V_h$ satisfies the estimate
\begin{equation*}
    \norm*{H}{u(t)-\calI_h u(t)}  + h \norm*{V}{u(t)-\calI_h u(t)} \leq C h^{s+1}\norm*{H^{s+1}(\Omega)}{u(t)}.
\end{equation*}
(This is the case, e.g., if $H=\m{L2}$, $V=\m{H10}$, and $V_h$ consists of continuous piecewise polynomials of degree $s\geq 1$; see \cref{thm:interp:global}.)

Finally, we will make use of a duality argument, which requires considering for given $\phi\in H$ the solution of the adjoint equation
\begin{equation*}
    b_h(y_h,z_h) = 0 \qquad\text{with}\quad z_N = \phi.
\end{equation*}
Integrating by parts on each interval $J_m$ and rearranging the jump terms, we can express the adjoint equation as
\begin{equation}\label{eq:evolution:adjoint}
    \begin{multlined}[t][0.9\displaywidth]
        \sum_{m=1}^N  \int_{J_m} - \scalprod*{H}{y_h(t)}{d_t z_h(t)} + a(y_h(t),z_h(t)) \,dt\\ + \sum_{m=1}^{N-1}\scalprod*{H}{y_m}{\jump{z_h}_m} + \scalprod*{H}{y_N}{z_N} = \scalprod*{H}{y_N}{\phi}.
    \end{multlined}
\end{equation}
This can be interpreted as a backwards in time equation with \enquote{initial value} $z_h(t_N) = \phi$.
Making the substitution $t\mapsto t_N -t$, we can apply \cref{thm:evolution:dg_stability} to obtain
\begin{equation}\label{eq:evolution:adjoint_stability}
    \sum_{m=1}^N\int_{J_m}\norm*{H}{d_t z_h(t)}^2 + \norm*{H}{A^*z_h(t)}^2\,dt + \sum_{m=1}^{N}\norm*{H}{\jump{z_h}_{m-1}}^2 \leq C \norm*{H}{\phi}^2,
\end{equation}
where $A$ is again the operator corresponding to the bilinear form $a$ and we have used that $\tau_m<1$ for $N$ sufficiently large.

Now everything is in place to show the following a priori estimate for the discrete solution at each time step.\footnote{It is possible -- though more involved -- to show error estimates for arbitrary $t\in[0,T]$; see, e.g., \cite[Theorem 12.2]{Thomee:2006}.}
\begin{theorem}
    For $r=0$, the solutions $u\in X$ to \eqref{eq:evolution:weak} and $u_h\in X_h$ to \eqref{eq:evolution:dg} satisfy
    \begin{equation*}
        \norm*{H}{u(t_m)-u_m} \leq C \max_{1\leq n\leq m} \left(h^{s+1}\sup_{t\in J_n}\norm*{H^{s+1}(\Omega)}{u(t)} + \tau_n \int_{J_n} \norm*{H}{d_t u}\,dt\right)
    \end{equation*}
    for all $1\leq m\leq N$.
\end{theorem}
\begin{proof}
    We write the error $e(t)$ at almost every $t\in (0,T]$ as
    \begin{equation*}
        \begin{aligned}
            e(t) := u(t) - u_h(t) &= (u(t) - \calI_h\pi_r u (t)) + (\calI_h\pi_r u (t) - u_h(t))\\
            &=: e_1(t) + e_2(t).
        \end{aligned}
    \end{equation*}
    (Note that pointwise a.e., $e(t)\in V$ since $V_h\subset V$, but as a function in time, only $e_2\in X_h$ is in a meaningful function space.)
    For $t=t_m$, we have $\pi_r u (t_m) = u(t_m)$ by construction, and hence
    \begin{equation}\label{eq:evolution:apriori1}
        \norm*{H}{ e_1(t_m)} = \norm*{H}{\calI_h u(t_m) - u(t_m)} \leq C h^{s+1}\norm*{H^{s+1}(\Omega)}{u(t_m)}.
    \end{equation}

    To bound $e_2(t_m)$, we use the duality trick. For arbitrary $\phi\in H$, let $z_h$ denote the solution of \eqref{eq:evolution:adjoint} with $N=m$. Since we have a consistent approximation, we can insert $u$ and $u_h$ in \eqref{eq:evolution:dg} for $y_h\in Y_h \subset Y$ and subtract to deduce that
    \begin{equation*}
        0 = b_h(e,y_h) = b_h(e_1,y_h) + b_h(e_2,y_h)\quad \text{ for all } y_h\in X_h.
    \end{equation*}
    From this and $d_t (z_h|_{J_n}) = 0$, we obtain with $y_h=e_2\in X_h$ that $z_h$ satisfies
    \begin{equation*}
        \begin{aligned}[t]
            \scalprod*{H}{e_2(t_m)}{\phi} &= b_h(e_2,z_h) = - b_h(e_1,z_h)\\
            &= -\sum_{n=1}^m\int_{J_n} a(e_1(t),z_h(t))\,dt - \sum_{n=1}^{m-1}\scalprod*{H}{e_1(t_n)}{\jump{z_h}_n} - \scalprod*{H}{e_1(t_m)}{\phi}.
        \end{aligned}
    \end{equation*}
    Introducing $\scalprod*{H}{Ae_1}{z_h(t)} = a(e_1,z_h)$ as above, using the Cauchy--Schwarz inequality, and estimating $e_1$ by its maximum pointwise in time yields
    \begin{equation*}
        |\scalprod*{H}{e_2(t_m)}{\phi}| \leq \left( \sup_{t\leq t_m}\norm*{H}{e_1(t)}\right) \left(\sum_{n=1}^m \int_{J_n}\norm*{H}{A^*z_h(t)}\,dt + \sum_{n=1}^{m-1}\norm*{H}{\jump{z_h}_{n}} + \norm*{H}{\phi}\right).
    \end{equation*}
    From the dual definition of the norm in $H$ and estimate \eqref{eq:evolution:adjoint_stability}, we obtain
    \begin{equation}\label{eq:evolution:apriori2}
        \norm*{H}{e_2(t_m)} \leq C \max_{1\leq n\leq m} \sup_{t\in J_n}\norm*{H}{e_1(t)}.
    \end{equation}

    It remains to bound $e_1(t)$ for arbitrary $t\in J_n$, which we do by estimating
    \begin{equation}\label{eq:evolution:apriori3}
        \begin{aligned}[t]
            \norm*{H}{e_1(t)}
            & = \norm*{H}{u(t) - \calI_h\pi_r u (t)}\\
            &\leq \norm*{H}{u(t) - \pi_r u (t)} + \norm*{H}{\pi_r u (t) - \calI_h\pi_r u (t)}\\
            &\leq  C\tau_n \int_{J_n}\norm*{H}{d_t u(\tau)}\,d\tau + Ch^{s+1}\norm*{H^{s+1}(\Omega)}{u(t)} ,
        \end{aligned}
    \end{equation}
    where we have used that the spatial approximation properties are independent of time and that $\pi_r u(t)$ has the same spatial regularity as $u(t)$.
    Combining \eqref{eq:evolution:apriori1}, \eqref{eq:evolution:apriori2} and \eqref{eq:evolution:apriori3} yields the claim.
\end{proof}

For $r=1$, one can proceed similarly (using that $d_t z_h|_{J_m} \in P_{r-1}(J_m,V_h)$, and hence that $\int_{J_n}\scalprod*{H}{d_t z_h(t)}{u(t)-\pi_r u(t)}dt$ vanishes by definition of $\pi_r$) to obtain\footnote{e.g., \cite[Theorem 12.7]{Thomee:2006}}
\begin{theorem}
    For $r=1$, the solutions $u\in X$ to \eqref{eq:evolution:weak} and $u_h\in X_h$ to \eqref{eq:evolution:dg} satisfy
    \begin{equation*}
        \norm*{H}{u(t_m)-u_m} \leq C \max_{1\leq n\leq m} \left(h^{s+1}\sup_{t\in J_n}\norm*{H^{s+1}(\Omega)}{u(t)} + \tau_n^3 \int_{J_n} \norm{H2}{d_t^2 u(t)}\,dt\right)
    \end{equation*}
    for all $1\leq m\leq N$.
\end{theorem}

The general case (including time-dependent bilinear form $a$ and different discrete spaces $V_m$) can be found in \cite{Chrysafinos:2006a}.

\backmatter

\printbibliography

@book{Verfuerth:2013,
    AUTHOR = {Verf{\"u}rth, R{\"u}diger},
    TITLE = {A Posteriori Error Estimation Techniques for Finite Element Methods},
    SERIES = {Numerical Mathematics and Scientific Computation},
    PUBLISHER = {Oxford University Press},
    ADDRESS = {Oxford},
    YEAR = {2013},
    DOI = {10.1093/acprof:oso/9780199679423.001.0001},
}

@incollection{Strang:1972,
    AUTHOR = {Strang, Gilbert},
    TITLE = {Variational crimes in the finite element method},
    BOOKTITLE = {The Mathematical Foundations of the Finite Element Method with Applications to Partial Differential Equations ({P}roc. {S}ympos., {U}niv. {M}aryland, {B}altimore, {M}d., 1972)},
    PAGES = {689--710},
    PUBLISHER = {Academic Press},
    ADDRESS = {New York},
    YEAR = {1972},
    doi = {10.1016/B978-0-12-068650-6.50030-7},
}

@book{Brezzi:2013,
    AUTHOR = {Boffi, Daniele and Brezzi, Franco and Fortin, Michel},
    TITLE = {Mixed and Finite Element Methods and Applications},
    SERIES = {Springer Series in Computational Mathematics},
    VOLUME = {44},
    PUBLISHER = {Springer},
    ADDRESS = {New York},
    YEAR = {2013},
    doi = {10.1007/978-3-642-36519-5},
}

@incollection{Raviart:1977,
    year = {1977},
    booktitle = {Mathematical Aspects of Finite Element Methods},
    volume = {606},
    series = {Lecture Notes in Mathematics},
    editor = {Galligani, Ilio and Magenes, Enrico},
    doi = {10.1007/BFb0064470},
    title = {A mixed finite element method for 2-nd order elliptic problems},
    publisher = {Springer},
    address = {Berlin},
    author = {Raviart, P.A. and Thomas, J.M.},
    pages = {292--315},
}

@article{Nedelec:1980,
    year = {1980},
    journal = {Numerische Mathematik},
    volume = {35},
    number = {3},
    doi = {10.1007/BF01396415},
    title = {Mixed finite elements in {$\mathbb{R}^3$}},
    publisher = {Springer},
    author = {N{\'e}d{\'e}lec, Jean-Claude},
    pages = {315--341},
}

@book{DiPietro:2012,
    Address = {New York},
    Author = {Di Pietro, Daniele Antonio and Ern, Alexandre},
    Publisher = {Springer},
    Series = {Mathématiques et Applications},
    Volume = {69},
    Title = {Mathematical Aspects of Discontinuous Galerkin Methods},
    Year = {2012},
    doi = {10.1007/978-3-642-22980-0},
}

@book{Brenner:2008,
    Address = {New York},
    Author = {Brenner, Susanne C. and Scott, L. Ridgway},
    Edition = {3},
    Publisher = {Springer},
    Series = {Texts in Applied Mathematics},
    Title = {The Mathematical Theory of Finite Element Methods},
    Volume = {15},
    Year = {2008},
    doi = {10.1007/978-0-387-75934-0},
}

@book{Ern:2004,
    Address = {New York},
    Author = {Ern, Alexandre and Guermond, Jean-Luc},
    Publisher = {Springer},
    Series = {Applied Mathematical Sciences},
    Title = {Theory and Practice of Finite Elements},
    Volume = {159},
    Year = {2004},
    doi = {10.1007/978-1-4757-4355-5},
}

@book{Thomee:2006,
    Address = {Berlin},
    Author = {Thom{\'e}e, Vidar},
    Edition = {2},
    Publisher = {Springer},
    Series = {Springer Series in Computational Mathematics},
    Title = {Galerkin Finite Element Methods for Parabolic Problems},
    Volume = {25},
    Year = {2006},
    doi = {10.1007/3-540-33122-0},
}

@book{Braess:2007,
    Address = {Cambridge},
    Author = {Braess, Dietrich},
    Edition = {3},
    Publisher = {Cambridge University Press},
    Title = {Finite Elements},
    Year = {2007},
    doi = {10.1017/CBO9780511618635},
}

@book{Ciarlet:2002,
    Address = {Philadelphia, PA},
    Author = {Ciarlet, Philippe G.},
    Isbn = {0-89871-514-8},
    Note = {Reprint of the 1978 original [North-Holland, Amsterdam]},
    Publisher = {Society for Industrial and Applied Mathematics (SIAM)},
    Series = {Classics in Applied Mathematics},
    Title = {The Finite Element Method for Elliptic Problems},
    Volume = {40},
    Year = {2002},
    doi = {10.1137/1.9780898719208},
}

@unpublished{Suli,
    Author = {Endre S{\"u}li},
    Howpublished = {Lecture notes},
    Institution = {Mathematical Institute, University of Oxford},
    Title = {Finite Element Methods for Partial Differential Equations},
    Year = {2011},
    Url = {http://people.maths.ox.ac.uk/suli/fem.pdf},
}

@unpublished{Rannacher,
    Author = {Rolf Rannacher},
    Howpublished = {Lecture notes},
    Institution = {Institut f{\"u}r Angewandte Mathematik, Universit{\"a}t Heidelberg},
    Title = {Numerische Mathematik 2},
    Url = {http://numerik.iwr.uni-heidelberg.de/~lehre/notes/num2/numerik2.pdf},
    Year = {2008},
}

@book{Evans:2010,
    Address = {Providence, RI},
    Author = {Evans, Lawrence C.},
    Edition = {2},
    Isbn = {978-0-8218-4974-3},
    Publisher = {American Mathematical Society},
    Series = {Graduate Studies in Mathematics},
    Title = {Partial Differential Equations},
    Volume = {19},
    Year = {2010},
    doi = {10.1090/gsm/019},
}

@book{Zeidler:1995b,
    AUTHOR = {Zeidler, Eberhard},
    TITLE = {Applied Functional Analysis},
    SUBTITLE = {Main principles and their applications},
    SERIES = {Applied Mathematical Sciences},
    VOLUME = {109},
    PUBLISHER = {Springer},
    ADDRESS = {New York},
    YEAR = {1995},
    doi = {10.1007/978-1-4612-0821-1},
}

@book{Zeidler:1995a,
    AUTHOR = {Zeidler, Eberhard},
    TITLE = {Applied Functional Analysis},
    SERIES = {Applied Mathematical Sciences},
    VOLUME = {108},
    SUBTITLE = {Applications to mathematical physics},
    PUBLISHER = {Springer},
    ADDRESS = {New York},
    YEAR = {1995},
    doi = {10.1007/978-1-4612-0815-0},
}

@book{Grisvard:1985,
    AUTHOR = {Grisvard, Pierre},
    TITLE = {Elliptic Problems in Nonsmooth Domains},
    publisher = {Society for Industrial and Applied Mathematics},
    series = {Classics in Applied Mathematics},
    number = {69},
    year = {2011},
    doi = {10.1137/1.9781611972030},
}

@book{Renardy:2004,
    AUTHOR = {Renardy, Michael and Rogers, Robert C.},
    TITLE = {An Introduction to Partial Differential Equations},
    SERIES = {Texts in Applied Mathematics},
    VOLUME = {13},
    EDITION = {2},
    PUBLISHER = {Springer},
    ADDRESS = {New York},
    YEAR = {2004},
    DOI = {10.1007/b97427},
}

@book{Adams:2003a,
    AUTHOR = {Adams, Robert A. and Fournier, John J. F.},
    TITLE = {Sobolev {S}paces},
    EDITION = {2},
    PUBLISHER = {Academic Press},
    ADDRESS = { Amsterdam},
    YEAR = {2003},
    ISBN = {0-12-044143-8},
}

@manual{dealII,
    Author = {W. Bangerth and R. Hartmann and G. Kanschat},
    Title = {deal.II Differential Equations Analysis Library, Technical Reference},
    Url = {http://www.dealii.org/},
    year = {2013},
}

@book{fenics,
    title = {Automated Solution of Differential Equations by the Finite Element Method},
    author = {Anders Logg and Kent-Andre Mardal and Garth N. Wells and others},
    year = {2012},
    publisher = {Springer},
    doi = {10.1007/978-3-642-23099-8},
    Url = {http://fenicsproject.org},
}

@article{Courant:1943,
    Author = {Courant, Richard},
    Journal = {Bull. Amer. Math. Soc.},
    Pages = {1--23},
    Title = {Variational methods for the solution of problems of equilibrium and vibrations},
    Volume = {49},
    doi = {10.1090/S0002-9904-1943-07818-4},
    Year = {1943},
}

@article{Zlamal:1968,
    Author = {Zl{\'a}mal, Milo{\v{s}}},
    Journal = {Numer. Math.},
    Pages = {394--409},
    Title = {On the finite element method},
    Volume = {12},
    doi = {10.1007/BF02161362},
    Year = {1968},
}

@book{Ladyzhenskaya:1968,
    AUTHOR = {Ladyzhenskaya, Olga A. and Ural{'}tseva, Nina N.},
    TITLE = {Linear and Quasilinear Elliptic Equations},
    SERIES = {Translated from the Russian by Scripta Technica, Inc. Translation editor: Leon Ehrenpreis},
    PUBLISHER = {Academic Press},
    ADDRESS = {New York},
    YEAR = {1968},
}

@book{Troianiello:1987a,
    AUTHOR = {Troianiello, Giovanni Maria},
    TITLE = {Elliptic Differential Equations and Obstacle Problems},
    SERIES = {The University Series in Mathematics},
    PUBLISHER = {Plenum Press},
    ADDRESS = {New York},
    YEAR = {1987},
    DOI = {10.1007/978-1-4899-3614-1},
}

@article{Bramble:1970,
    author = "Bramble, James H. and Hilbert, Stephen R.",
    title = "{Estimation of linear functionals on Sobolev spaces with application to Fourier transforms and spline interpolation.}",
    journal = "SIAM J. Numer. Anal.",
    volume = "7",
    pages = "112--124",
    year = "1970",
    doi = {10.1137/0707006},
}

@book{Wloka:1987,
    AUTHOR = {Wloka, J.},
    TITLE = {Partial Differential Equations},
    NOTE = {Translated from the German by C. B. Thomas and M. J. Thomas},
    PUBLISHER = {Cambridge University Press},
    ADDRESS = {Cambridge},
    YEAR = {1987},
    doi = {10.1017/CBO9781139171755},
}

@book{Showalter:1997,
    AUTHOR = {Showalter, R. E.},
    TITLE = {Monotone Operators in {B}anach Space and Nonlinear Partial Differential Equations},
    SERIES = {Mathematical Surveys and Monographs},
    VOLUME = {49},
    PUBLISHER = {American Mathematical Society},
    ADDRESS = {Providence, RI},
    YEAR = {1997},
    doi = {10.1090/surv/049},
}

@article{Chrysafinos:2006a,
    AUTHOR = {Chrysafinos, K. and Walkington, Noel J.},
    TITLE = {Error estimates for the discontinuous {G}alerkin methods for parabolic equations},
    JOURNAL = {SIAM J. Numer. Anal.},
    FJOURNAL = {SIAM Journal on Numerical Analysis},
    VOLUME = {44},
    YEAR = {2006},
    NUMBER = {1},
    PAGES = {349--366 (electronic)},
    doi = {10.1137/030602289},
}

@article{Arnold:2002,
    author = {Arnold, D. and Brezzi, F. and Cockburn, B. and Marini, L.},
    title = {Unified analysis of discontinuous {Galerkin} methods for elliptic problems},
    journal = {SIAM Journal on Numerical Analysis},
    volume = {39},
    number = {5},
    pages = {1749--1779},
    year = {2002},
    doi = {10.1137/S0036142901384162},
}

@book{Edwards1965,
    Location = {Rinehart and Winston, New York},
    Author = {Edwards, R. E.},
    Publisher = {Holt},
    Title = {Functional Analysis. {T}heory and Applications},
    Year = {1965},
}

@book{Ern:2021c,
    Author = {Ern, Alexandre and Guermond, Jean-Luc},
    Title = {Finite Elements III},
    Subtitle = {First-Order and Time-Dependent PDEs},
    Publisher = {Springer},
    Series = {Texts in Applied Mathematics},
    Volume = {74},
    Year = {2021},
    doi = {10.1007/978-3-030-57348-5},
}

\end{document}